\documentclass[11pt]{article}
\usepackage[margin=1in]{geometry}
\usepackage{xspace,exscale,relsize}
\usepackage{fancybox,shadow}
\usepackage{graphicx}
\usepackage{color,soul}
\usepackage{algorithm}
\usepackage{algpseudocode}
\usepackage{verbatim}

\usepackage{prettyref,enumerate,setspace}

\usepackage[normalem]{ulem}
\usepackage{booktabs,siunitx}

\newrefformat{eq}{(\ref{#1})}
\newrefformat{thm}{Theorem~\ref{#1}}
\newrefformat{sec}{Section~\ref{#1}}
\newrefformat{algo}{Algorithm~\ref{#1}}
\newrefformat{fig}{Fig.~\ref{#1}}
\newrefformat{tab}{Table~\ref{#1}}
\newrefformat{rmk}{Remark~\ref{#1}}
\newrefformat{clm}{Claim~\ref{#1}}
\newrefformat{def}{Definition~\ref{#1}}
\newrefformat{cor}{Corollary~\ref{#1}}
\newrefformat{lmm}{Lemma~\ref{#1}}
\newrefformat{prop}{Proposition~\ref{#1}}
\newrefformat{pr}{Proposition~\ref{#1}}
\newrefformat{property}{Property~\ref{#1}}
\newrefformat{app}{Appendix~\ref{#1}}
\newrefformat{apx}{Appendix~\ref{#1}}
\newrefformat{ex}{Example~\ref{#1}}
\newrefformat{exer}{Exercise~\ref{#1}}
\newrefformat{soln}{Solution~\ref{#1}}
\newrefformat{supplmm}{Lemma~\ref{#1}}
\newrefformat{suppprop}{Prosposition~\ref{#1}}

\usepackage[
            bookmarksnumbered=true,
            bookmarksopen=true,
            colorlinks=true,
            citecolor=red,
            linkcolor=blue,
            anchorcolor=red,
            urlcolor=blue
            ]{hyperref}

\graphicspath{ {../exercises/figs/} }

\usepackage{epstopdf}
\usepackage{amsfonts}
\usepackage{amssymb}
\usepackage{url}
\usepackage{amsmath,amsthm}
	\makeatletter
	\let\over=\@@over \let\overwithdelims=\@@overwithdelims
	\let\atop=\@@atop \let\atopwithdelims=\@@atopwithdelims
  	\let\above=\@@above \let\abovewithdelims=\@@abovewithdelims
  	\makeatother
\usepackage{rotating}
\usepackage{ifpdf}
\usepackage{subfigure}
\usepackage{psfrag}

\ifx\eqref\undefined
	\newcommand{\eqref}[1]{~(\ref{#1})}
\fi
\ifx\mod\undefined
	\def\mod{\mathop{\rm mod}}
\fi

\newcommand{\Fnorm}[1]{\Vert #1\Vert_{\rm F}}
\newcommand{\Unif}{\text{Unif}}
\newcommand{\Law}{\text{Law}}

\def\exp{\mathop{\rm exp}}

\def\Cov{\mathrm{Cov}}

\def\unifto{\mathop{{\mskip 3mu plus 2mu minus 1mu%
	\setbox0=\hbox{$\mathchar"3221$}%
	\raise.6ex\copy0\kern-\wd0%
	\lower0.5ex\hbox{$\mathchar"3221$}}\mskip 3mu plus 2mu minus 1mu}}

\ifx\lesssim\undefined
\def\simleq{{{\mskip 3mu plus 2mu minus 1mu%
	\setbox0=\hbox{$\mathchar"013C$}%
	\raise.2ex\copy0\kern-\wd0%
	\lower0.9ex\hbox{$\mathchar"0218$}}\mskip 3mu plus 2mu minus 1mu}}
\else
\def\simleq{\lesssim}
\fi

\ifx\gtrsim\undefined
\def\simgeq{{{\mskip 3mu plus 2mu minus 1mu%
	\setbox0=\hbox{$\mathchar"013E$}%
	\raise.2ex\copy0\kern-\wd0%
	\lower0.9ex\hbox{$\mathchar"0218$}}\mskip 3mu plus 2mu minus 1mu}}
\else
\def\simgeq{\gtrsim}
\fi

\newtheorem{theorem}{Theorem}
\newtheorem{lemma}[theorem]{Lemma}
\newtheorem{corollary}[theorem]{Corollary}
\newtheorem{proposition}[theorem]{Proposition}
\newtheorem{suppprop}[theorem]{Proposition}
\newtheorem{supplemma}[theorem]{Lemma}

\theoremstyle{definition}

\newtheorem{example}{Example}

\newcommand{\reals}{\mathbb{R}}

\newcommand{\Expect}{\mathbb{E}}

\newcommand{\Prob}{\mathbb{P}}
\newcommand{\prob}[1]{\mathbb{P}\left[#1\right]}

\newcommand{\ie}{i.e.\xspace}
\newcommand{\iid}{iid\xspace}
\newcommand{\pth}[1]{\left( #1 \right)}

\newcommand{\toprob}{{\xrightarrow{\Prob}}}

\newcommand{\iiddistr}{{\overset{\text{\iid}}{\sim}}}

\newcommand{\var}{\mathrm{var}}

\newcommand\indep{\protect\mathpalette{\protect\independenT}{\perp}}
\def\independenT#1#2{\mathrel{\rlap{$#1#2$}\mkern2mu{#1#2}}}
\newcommand{\Bern}{\mathrm{Bern}}
\newcommand{\Binom}{\mathrm{Binom}}

\newcommand{\indc}[1]{{\mathbf{1}\left\{{#1}\right\}}}

\newcommand{\calI}{{\mathcal{I}}}
\newcommand{\calJ}{{\mathcal{J}}}

\newcommand{\calN}{{\mathcal{N}}}

\newcommand{\calR}{{\mathcal{R}}}

\newcommand{\calX}{{\mathcal{X}}}

\renewcommand{\iff}{\Leftrightarrow}

\newcommand{\Lfdr}{\mathrm{Lfdr}}
\newcommand{\mFDR}{\mathrm{mFDR}}
\newcommand{\mFNR}{\mathrm{mFNR}}
\newcommand{\FDR}{\mathrm{FDR}}
\newcommand{\FNR}{\mathrm{FNR}}
\newcommand{\FDP}{\mathrm{FDP}}
\newcommand{\FNP}{\mathrm{FNP}}

\newcommand{\EFN}{\mathrm{EFN}}
\newcommand{\FN}{\mathrm{FN}}

\newcommand{\optFNR}{{\sf FNR}^*}
\newcommand{\optEFNnorm}{{\sf EFN}^*}
\newcommand{\optmEFNnorm}{{\sf mEFN}^*}
\newcommand{\optmFNR}{{\sf mFNR}^*}
\newcommand{\maroracle}{\delta^n_{\rm NP}}
\newcommand{\NP}{\delta_{\rm NP}}
\newcommand{\oracle}{\delta^n_{\rm OR}}

\begin{document}

\title{Large-scale Multiple Testing: Fundamental Limits of False Discovery Rate Control and Compound Oracle}
\author{Yutong Nie and Yihong Wu\thanks{The authors are with the Department of Statistics and Data Science, Yale University, New Haven CT, USA, 
\texttt{$\{$yihong.wu,yutong.nie$\}$@yale.edu}.
Y.~Wu is supported in part by the NSF Grant CCF-1900507, an NSF CAREER award CCF-1651588, and an Alfred Sloan fellowship.}}
\date{\today}

\maketitle

\begin{abstract}

The false discovery rate (FDR) and the false non-discovery rate (FNR), defined as the expected false discovery proportion (FDP) and the false non-discovery proportion (FNP), are the most popular benchmarks for multiple testing. Despite the theoretical and algorithmic advances in recent years, the optimal tradeoff between the FDR and the FNR has been largely unknown except for certain restricted classes of decision rules, e.g., separable rules, 
or for other performance metrics, e.g., the marginal FDR and the marginal FNR (mFDR and mFNR). In this paper, we determine the asymptotically optimal FDR-FNR tradeoff under the two-group random mixture model when the number of hypotheses tends to infinity. Distinct from the optimal mFDR-mFNR tradeoff, which is achieved by separable decision rules, the optimal FDR-FNR tradeoff requires compound rules even in the large-sample limit and for models as simple as the Gaussian location model. 
This suboptimality of separable rules also holds for other objectives, such as maximizing the expected number of true discoveries. Finally, to address the limitation of the FDR which only controls the expectation but not the fluctuation of the FDP, we also determine the optimal tradeoff when the FDP is controlled with high probability and show it coincides with that of the mFDR and the mFNR.
Extensions to models with a fixed non-null proportion are also obtained.

\end{abstract}

\tableofcontents

\section{Introduction}

\subsection{Background and problem formulation}
\label{sec:Backgrounds}
One of the central topics in modern statistics,
large-scale hypothesis testing has been widely applied in a variety of fields such as genetics, astronomy and brain imaging, in which hundreds or thousands of tests are carried out simultaneously, with the primary goal of identifying the non-null hypotheses while controlling the false discoveries. One of the most popular figures of merit in multiple testing is the \emph{false discovery rate} (FDR), formally introduced by Benjamini and Hochberg in 1995 \cite{benjamini1995controlling}. 
Let $\theta^n\triangleq (\theta_1,\theta_2,\cdots,\theta_n)\in\{0,1\}^n$ denote the true labels of $n$ hypotheses, where $\theta_i=0$ when the $i$-th hypothesis is a null and $\theta_i=1$ otherwise. Given $n$ observations $X^n\triangleq (X_1,X_2,\cdots,X_n)\in\calX^n$, a (possibly randomized) decision rule $\delta^n \triangleq  (\delta_1,\cdots,\delta_n)$ is represented by a Markov kernel $P_{\delta^n|X^n}$ from the sample space $\calX^n$ to $\{0,1\}^n$. The FDR of this procedure is defined to be the expectation of the \emph{false discovery proportion} (FDP), \ie
\begin{align*}
    \FDR(\delta^n) \triangleq \Expect[\FDP(\delta^n)],\quad\FDP(\delta^n)\triangleq  \frac{\sum_{i=1}^n \delta_i(1-\theta_i)}{1\vee\sum_{i=1}^n \delta_i}.
\end{align*}
The FDR criterion is less stringent than the traditional family-wise error rate (the probability of having at least one false discovery). As such, it has been commonly adopted in multiple testing and significant progress has been achieved in designing decision rules that control the FDR at a desired level \cite{benjamini1995controlling,benjamini1999step,benjamini2001control,sarkar2002some,benjamini2006adaptive,gavrilov2009adaptive}.

To test a single (simple) hypothesis, it is well-known that the Neyman-Pearson test \cite{neyman1933ix} based on the likelihood ratio is the most powerful which minimizes the Type \text{\uppercase\expandafter{\romannumeral2}} error rate (the probability of falsely accepting a non-null) while controlling the Type \text{\uppercase\expandafter{\romannumeral1}} error rate (the probability of rejecting a true null) at a prescribed level $\alpha$. Moving to multiple testing, a natural question is to find the optimal decision rule under a meaningful  objective function.
To this end, analogous to Type \text{\uppercase\expandafter{\romannumeral1}} and \text{\uppercase\expandafter{\romannumeral2}} error rates, Genovese and Wasserman \cite{genovese2002operating} introduced a dual quantity of the FDR, called the \emph{false non-discovery rate} (FNR), which is the expectation of the \emph{false non-discovery proportion} (FNP):
\begin{align*}
    \FNR(\delta^n) \triangleq \Expect[\FNP(\delta^n)],\quad\FNP(\delta^n)\triangleq  \frac{\sum_{i=1}^n (1-\delta_i)\theta_i}{1\vee\sum_{i=1}^n (1-\delta_i)}.
\end{align*}
In the same paper \cite[Section 5, p.~506]{genovese2002operating}, they put forth  the fundamental question of finding the optimal procedure that minimizes the FNR among all procedures that control the FDR at a prespecified level $\alpha$.

In this article, 
we resolve this question in the large-$n$ limit under a Bayesian setting.
Specifically, we focus on the so-called \emph{two-group random mixture model}, which is widely used in the large-scale testing literature \cite{efron2001empirical,genovese2004stochastic,sun2007oracle,efron2008microarrays,tony2017optimal,heller2021optimal}. Suppose the true labels $\theta^n=(\theta_1,\theta_2,\cdots,\theta_n)$ are independent and identically distributed as $\Bern(\pi_1)$, where $\pi_1$ represents the average non-null fraction. Conditioning on $\theta^n$, the observations are independently generated as
\begin{align}\label{eq:two-group_model}
		X_i\,|\,\theta_i \sim f_{\theta_i},\quad i=1,\cdots,n,
\end{align}
where $f_0$ and $f_1$ are probability density functions (pdfs) corresponding to the null and non-null hypothesis, respectively. 
The fundamental limit of multiple testing in this model is:
\begin{align}
\begin{split}
    \optFNR_n(\alpha) \triangleq  & \inf \,\,\,\FNR(\delta^n) = \Expect\left[\frac{\sum_{i=1}^n (1-\delta_i)\theta_i}{1\vee\sum_{i=1}^n (1-\delta_i)}\right]\\
    & \,\text{s.t.}\,\,\,\, \FDR(\delta^n)= \Expect\left[\frac{\sum_{i=1}^n \delta_i(1-\theta_i)}{1\vee\sum_{i=1}^n \delta_i}\right]\leq \alpha,
\end{split}\label{eq:opt_prob_FDR&FNR}
\end{align}
where the infimum is taken over all (possibly randomized)\footnote{We will show later that in the large-$n$ limit it is without loss of optimality to restrict to deterministic rules in \prettyref{eq:opt_prob_FDR&FNR}; see \prettyref{sec:MainThm}.} feasible decision rules $\delta^n$. 
In other words, $ \optFNR_n(\alpha)$ is the minimum FNR achieved by the optimal decision rule that controls the FDR at level $\alpha$. 
For multiple testing, there is no counterpart of the Neyman-Pearson lemma. Despite the significant progress on the optimal FDR control with a restricted class of separable procedures \cite{genovese2002operating,sun2007oracle,tony2017optimal,cao2022optimal}, it still remains open how to find the optimal decision rule that achieves the optimal $\optFNR_n(\alpha)$ in \prettyref{eq:opt_prob_FDR&FNR} for general null and alternative distributions.\footnote{A different asymptotic setting is considered in \cite{arias2017distribution,rabinovich2020optimal}: It is assumed that $f_0$ is a centered distribution with tails decaying as $\exp(-|x|^\gamma / \gamma)$ for some $\gamma\geq 1$ and $f_1$ is shifted by $\mu_n = (\gamma r_n\log n)^{1/\gamma}>0$. The non-null proportion scales as $n^{-\beta_n}$, $\beta_n\in(0,1)$. Under these conditions, the optimal rate at which $\FDR+\FNR$ vanishes is determined.}

\subsection{mFDR: a related but nonequivalent criterion}
\label{sec:mFDR}

In view of the difficulty of solving the optimal FDR control problem \prettyref{eq:opt_prob_FDR&FNR}, a number of proxies to FDR have been proposed. 
\cite{genovese2002operating} considered asymptotic approximations to the FDR and FNR when the decision rules are restricted to those that threshold $p$-values with a fixed cutoff. Later in \cite{storey2003positive,sun2007oracle}, these approximations were referred to as the marginal false discovery rate (mFDR) and marginal false non-discovery rate (mFNR), formally defined as
\begin{align*}
    \mFDR(\delta^n) \triangleq \frac{\Expect[\sum_{i=1}^n \delta_i(1-\theta_i)]}{\Expect[\sum_{i=1}^n \delta_i]}\quad\text{and}\quad\mFNR(\delta^n)\triangleq  \frac{\Expect[\sum_{i=1}^n (1-\delta_i)\theta_i]}{\Expect[\sum_{i=1}^n (1-\delta_i)]},
\end{align*}
replacing the expectation of ratios in the FDR and FNR by ratios of two expectations, with the understanding that $\frac{0}{0}=0$.
The problem of optimal mFDR control in the two-group model \eqref{eq:two-group_model} has been studied in \cite{sun2007oracle}:
\begin{align}
\begin{split}
    \optmFNR(\alpha)\triangleq  & \inf \,\,\,\mFNR(\delta^n)\\
    & \,{\rm s.t.}\,\,\,\mFDR(\delta^n)\leq \alpha.
\end{split}
\label{eq:opt_prob_mFDR&mFNR}
\end{align}
This problem turns out to not depend on $n$ so we omit the subscript; see  \eqref{eq:opt_prob_a&b} 
for a more explicit equivalent form of $\optmFNR(\alpha)$. It was shown in \cite{sun2007oracle,xie2011optimal} that the optimal mFDR control, namely, the optimal solution to \eqref{eq:opt_prob_mFDR&mFNR}, is a separable procedure (each $\delta_i$ only depending on $X_i$), 
given by thresholding the likelihood ratio (in this context known as the \emph{local false discovery rate} (Lfdr) \cite{efron2001empirical}) with a fixed cutoff. 
In addition to mFNR, other objective functions have also been considered subject to an mFDR constraint, including maximizing the expected number of discoveries \cite{tony2017optimal} 
and maximizing the expected number of true discoveries \cite{cao2022optimal}. 

Overall, controlling the mFDR is technically less challenging than controlling the FDR since the ratio of two expectations is easier to handle than the expectation of a ratio. 
As pointed out in \cite{storey2003positive}, despite its technical simplicity, the mFDR criterion is unable to jointly control the numerator and denominator of the random fraction FDP. 
It was argued in \cite{genovese2002operating,sun2007oracle,tony2017optimal,basu2018weighted} that the mFDR and FDR are asymptotically equivalent when the decision rule is separable with a fixed rejection region, so it is reasonable to control the mFDR rather than FDR for simplicity. However, as we will see soon, the optimal oracle decision rules of controlling the FDR need not be separable. 
In fact, the classical result of Robbins \cite{robbins1951asymptotically} showed that compound rules can substantially outperform separable rules, which form the premise of the empirical Bayes and compound decision theory. 
For this reason, we do not place any restriction on the form of the decision rules in the optimal FDR control problem \prettyref{eq:opt_prob_FDR&FNR}.

Recent works \cite{heller2021optimal,rosset2022optimal} provided valuable insight to address the difference between controlling the FDR and mFDR for a finite number of hypotheses. In addition to giving conditions that guarantee the existence and uniqueness of the optimal FDR control, they showed that for any number of hypotheses in the two-group model, the optimal procedure is \emph{compound} and takes the form of thresholding the Lfdr with a data-driven cutoff that depends on all observations, in contrast to the fixed threshold for the optimal mFDR control. An even more surprising result, as we will show in this paper, is that controlling the FDR and mFDR remains non-equivalent even when the number of hypotheses tends to infinity. 
In fact, the separable rule that optimally controls the mFDR is not optimal in controlling the FDR even in models as simple as the Gaussian location model. 
To demonstrate this point, let us consider two objective functions:
(a) minimizing the FNR as in \prettyref{eq:opt_prob_FDR&FNR}; and 
(b) maximizing the expected number of true discoveries divided by $n$ (ETD), or equivalently, minimizing the expected number of false non-discoveries divided by $n$ (EFN), \ie
\begin{align}
\begin{split}
    \optEFNnorm_n(\alpha) \triangleq  & \inf \,\,\,\EFN(\delta^n) = \frac{1}{n}\Expect\left[\sum_{i=1}^n\theta_i(1-\delta_i)\right]\\
    & \,\,\text{\rm s.t.}\,\,\,\,\,\, \FDR(\delta^n)\leq \alpha,
\end{split}\label{eq:opt_prob_FDR&EFN}
\end{align}  
and its marginal version
\begin{align}
\begin{split}
    \optmEFNnorm(\alpha) \triangleq  & \inf \,\,\,\EFN(\delta^n) = \frac{1}{n}\Expect\left[\sum_{i=1}^n\theta_i(1-\delta_i)\right]\\
    & \,\,\text{\rm s.t.}\,\,\,\,\, \mFDR(\delta^n)\leq \alpha.
\end{split}\label{eq:opt_prob_mFDR&EFN}
\end{align}
Problem \prettyref{eq:opt_prob_mFDR&EFN} turns out to not depend on $n$ and have the same optimal solution as \eqref{eq:opt_prob_mFDR&mFNR}. (See \cite{xie2011optimal,heller2021optimal} for a more general result that allows dependent data. For completeness, we give a formal statement in \prettyref{lmm:optimal_sol_of_mFDR&EFN} and a proof in \prettyref{app:OptRulemFDR}.)

We start with some notations. Let $f\triangleq \pi_0f_0 + \pi_1f_1$ be the marginal pdf, where $\pi_0+\pi_1=1$ and $\pi_0\in(0,1)$. Define
\[
\Lfdr(x)\triangleq\pi_0f_0(x)/f(x)
\]
to be the local false discovery rate \cite{efron2001empirical} at $x\in\calX$. It is well-known (cf.~e.g.~\cite{sun2007oracle,xie2011optimal,cao2022optimal}) that the optimal mFDR control, \ie, the optimal solution to \eqref{eq:opt_prob_mFDR&mFNR} and \eqref{eq:opt_prob_mFDR&EFN}, consists of $n$ independent and identical Neyman-Pearson tests. Specifically, the Neyman-Pearson lemma \cite{neyman1933ix} states that for testing ${\rm H}_0:\,X \sim f_0$ against ${\rm H}_1:\,X \sim f_1$, the most powerful test that controls the Type \text{\uppercase\expandafter{\romannumeral1}} error rate at level $\alpha$ is given by
\begin{align}
    \NP(X;\alpha)=
    \begin{cases}
        1 & \text{if } \Lfdr(X) < c\\
        \Bern(p) & \text{if } \Lfdr(X) = c\\
        0 & \text{if } \Lfdr(X) > c
    \end{cases}
		\label{eq:deltaNP}
\end{align}
with constants $c,p\in[0,1]$ chosen such that $\Prob(\NP(X;\alpha)=1\,|\,{\rm H}_0)=\alpha$. Then, the optimal mFDR control is given by $\maroracle(\alpha)\triangleq (\NP(X_1;\alpha),\cdots,\NP(X_n;\alpha))$; see \prettyref{thm:optimal_sol_of_mFDR&mFNR} for a formal statement and \prettyref{app:OptRulemFDR} for a proof. To simplify the notation, henceforth we denote 
\begin{equation}
W_i =\Lfdr(X_i)=\frac{\pi_0f_0(X_i)}{f(X_i)}, \quad i=1,\cdots,n.
\label{eq:Lfdri}
\end{equation}

\begin{example}[Gaussian mixture model]
\label{ex:gaussian}
Suppose that each $X_i \sim \calN(0,1)$ under the null and $\calN(\mu,1)$ under the alternative. (By symmetry, we only consider $\mu>0$.)
With $\varphi$ denoting the standard normal density, we have
\[
W_i = \frac{\pi_0 \varphi(X_i)}{\pi_0 \varphi(X_i) + \pi_1 \varphi(X_i-\mu)} 
= \frac{\pi_0}{\pi_0 + \pi_1\exp(\mu X_i - \mu^2/2)}
\]
which is a continuous random variable and decreasing in $X_i$. 
Hence, $\maroracle(\alpha)$ takes the form $\delta^n(t)=\left(\indc{X_1\geq t},\cdots,\indc{X_n\geq t}\right)$ for some constant $t\in[0,1]$ depending on $\pi_0$, $\mu$ and $\alpha$. 
Let $\Prob$ and $\Expect$ be taken with respect to the marginal pdf $f$, and $\Prob_i$ and $\Expect_i$ with respect to $f_i$, $i=0,1$. Let $\Phi$ denote the cumulative distribution function of the standard normal. Then, 
\begin{align*}
&\mFDR(t)\triangleq\mFDR(\delta^n(t))=\frac{\pi_0\Expect_0[\indc{X_1\geq t}]}{\Prob\pth{X_1\geq t}} = \frac{\pi_0\pth{1-\Phi(t)}}{1-\pth{\pi_0\Phi(t)+\pi_1\Phi(t-\mu)}},\\
&\mFNR(t) \triangleq\mFNR(\delta^n(t))=\frac{\pi_1\Expect_1[\indc{X_1< t}]}{\Prob\pth{X_1< t}} = \frac{\pi_1\Phi(t-\mu)}{\pi_0\Phi(t)+\pi_1\Phi(t-\mu)}\quad\text{and}\\
&\EFN(t)\triangleq\EFN(\delta^n(t))=\pi_1\Expect_1[\indc{X_1< t}]=\pi_1\Phi(t-\mu).
\end{align*}
\end{example}

\prettyref{fig:normal_model} displays the (m)FNR-(m)FDR curves of the separable rule $\delta^n(t)$ on the left and EFN-(m)FDR curves on the right, for $\pi_0=0.75$ and $\mu=1,1.6,2$. Since the FDR and mFDR, the FNR and mFNR are asymptotically equivalent for $\delta^n(t)$, we do not distinguish between them in this example. A simple observation is that if a point $(\FDR(t),\FNR(t))=(\alpha,\FNR(\NP^n(\alpha)))$ is above the straight line connecting $(0,\pi_1)$ and $(\pi_0,0)$, referred to as the baseline, then the separable rule $\NP^n(\alpha)$ is strictly dominated by a trivial procedure that randomizes between all rejections and all acceptances, attaining FDR equal to $\alpha$ and FNR equal to $\pi_1(1-\frac{\alpha}{\pi_0})<\FNR(\NP^n(\alpha))$. 
It is clear from \prettyref{fig:normal_model} that when $\mu=1$ and $1.6$, the FNR-FDR and EFN-FDR curves are partially or even entirely above the baseline. 
In other words, there exist some FDR level $\alpha$ such that $\NP^n(\alpha)$ is dominated by the trivial randomization procedure. In fact, this provably happens for any $\mu$ regardless of its magnitude. 
We show in \prettyref{app:GaussianModel} the following:
\begin{itemize}
	\item If the signal is sufficiently weak, namely, $\mu\leq \mu_0$ for some universal constant $\mu_0$,
	then for any FDR level $\alpha\in(0,\pi_0)$, we have $\FNR(\NP^n(\alpha)) > \pi_1(1 - \frac{\alpha}{\pi_0})$.
	As a consequence, the separable rule $\maroracle(\alpha)$ is strictly suboptimal in controlling the FDR for any $\alpha$ level.
	
    \item On the other hand, regardless of how strong the signal is, there always exists some FDR level $\alpha$ such that $\maroracle(\alpha)$ is strictly dominated by the trivial randomization procedure. To be more specific, for all $\mu>0$ and $\pi_0\in(0,1)$, there exist $\alpha_0=\alpha_0(\mu,\pi_0)$ such that $\FNR(\NP^n(\alpha)) > \pi_1(1 - \frac{\alpha}{\pi_0})$ and $\EFN(\NP^n(\alpha)) > \pi_1(1 - \frac{\alpha}{\pi_0})$ for any $\alpha\in(0,\alpha_0)$.

    \item Finally, unlike the FNR-FDR curve which can be entirely dominated by the baseline for some sufficiently weak signal, the EFN-FDR curve is never entirely above the baseline. Namely, for all $\mu>0$ and $\pi_0\in(0,1)$, there exist $\alpha_1=\alpha_1(\mu,\pi_0)$ such that $\EFN(\NP^n(\alpha)) < \pi_1(1 - \frac{\alpha}{\pi_0})$ for any $\alpha\in(\alpha_1,1)$.
\end{itemize}
\begin{figure}[thpb]
	\centering
	\includegraphics[scale=0.46]{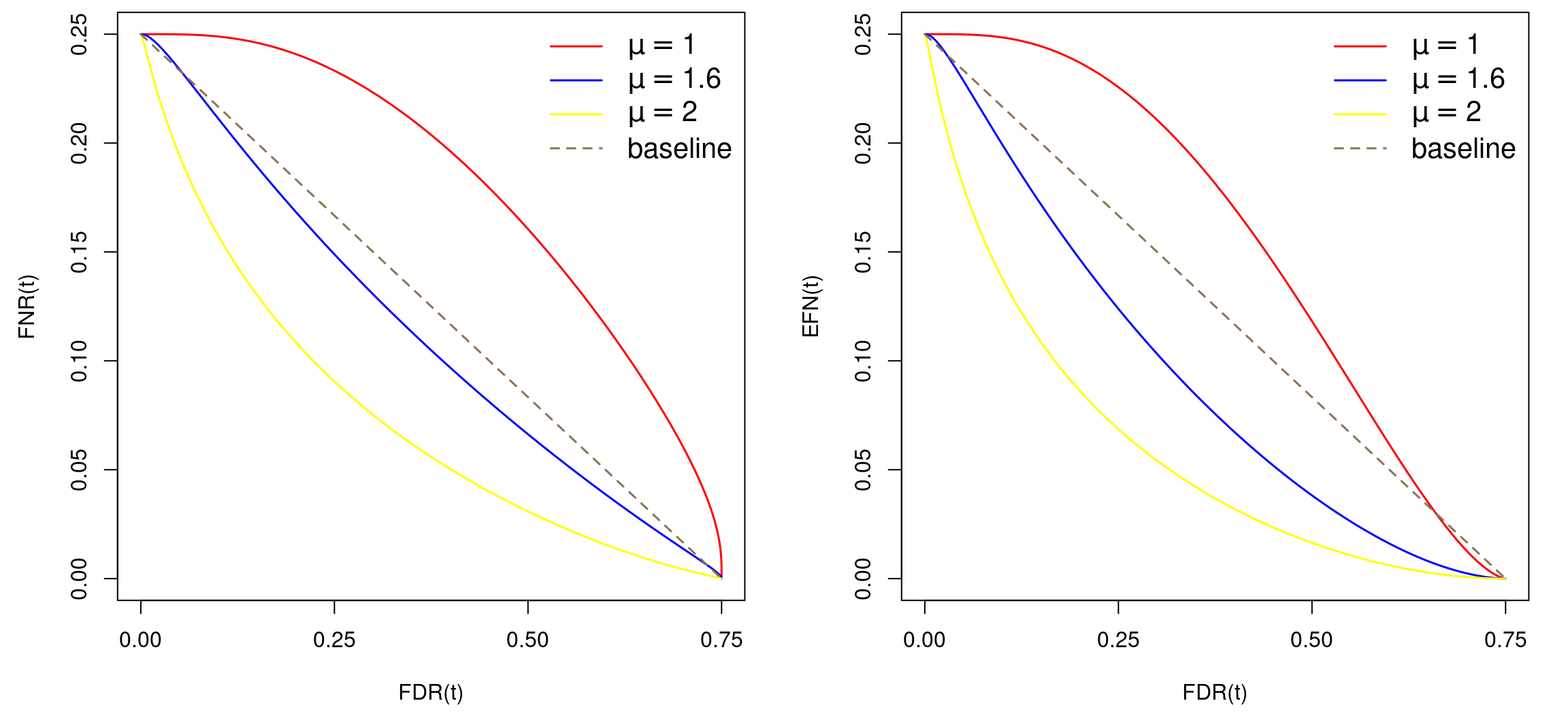}
	\caption{FNR-FDR (left) and EFN-FDR (right) curves of $\NP^n$ for $\mu= 1, 1.6, 2$ and $\pi_0=0.75$ in Gaussian mixture models. The dashed line stands for the baseline which randomizes between all rejections and all acceptances.}
	\label{fig:normal_model}
\end{figure}

Furthermore, as long as the FNR-FDR curve of $\NP^n$ is non-convex, \ie, there exists $\lambda\in(0,1)$ and $t_1<t<t_2$ such that $\lambda\FDR(t_1)+(1-\lambda)\FDR(t_2)=\FDR(t)$ and $\lambda\FNR(t_1)+(1-\lambda)\FNR(t_2)<\FNR(t)$, then $\delta^n(t)$ is strictly dominated by the rule which randomizes between $\delta^n(t_1)$ and $\delta^n(t_2)$. In general, we will show in the next section that the minimized FNR is asymptotically achieved by randomizing between two separable rules. 

Let us emphasize that the suboptimality of separable rules is not due to the allowance for randomization, because we can simulate the randomized rule by deterministic ones without affecting the performance in the asymptotic setting. As a result, even if we restrict to deterministic rules, the asymptotically optimal FDR control is still given by compound rules rather than separable ones. See the discussion in \prettyref{sec:MainThm} after \prettyref{thm:ora_FNR_vs_ora_mFNR}.

\subsection{Main results}\label{sec:MainThm}

One of the main contributions of this paper is the asymptotic solution of the optimization problem \eqref{eq:opt_prob_FDR&FNR}. Along the way, we will characterize the relations between the optimal controls of the FDR and mFDR, discuss the limitations of the FDR control and how to remedy them. The following theorem, which is the first main result of this paper, characterizes the large-sample fundamental limit of the FDR control in multiple testing. (Similar results hold for optimal FDR-EFN tradeoff in \eqref{eq:opt_prob_FDR&EFN}. The formal statements and proofs are postponed to \prettyref{sec:EFN}.)
\begin{theorem}[Asymptotically optimal FDR control] 
\label{thm:ora_FNR_vs_ora_mFNR}
For the optimization problem \prettyref{eq:opt_prob_FDR&FNR} in the two-group random mixture model, the limit
\[
\optFNR(\alpha)\triangleq \lim\limits_{n\to\infty}\optFNR_n(\alpha)
\] 
exists, which, as a function of $\alpha$, is given by the greatest convex minorant (GCM) of the function $\optmFNR$, \ie \[\optFNR(\alpha) = \sup \{C(\alpha): C(t) \leq \optmFNR(t),\forall\,t\in[0,1],\,C(\cdot) \text{ is convex on }[0,1]\}.\] Furthermore, for any $\alpha\in[0,1]$, there exist constants $\alpha_1$ and $\alpha_2$, $0\leq \alpha_1\leq\alpha\leq \alpha_2\leq 1$, such that the following randomized decision rule, referred to as the oracle,
\begin{align}\label{eq:asym_opt_sol_of_FDR&FNR}
    \oracle(\alpha) = 
    \begin{cases}
    \maroracle(\alpha_1) & \text{with probability }\frac{\alpha_2-\alpha}{\alpha_2-\alpha_1}\\
    \maroracle(\alpha_2) & \text{with probability }\frac{\alpha-\alpha_1}{\alpha_2-\alpha_1}
    \end{cases}
\end{align}
achieves $\FDR(\oracle(\alpha))\leq\alpha$ and $\lim\limits_{n\to\infty}\FNR(\oracle(\alpha))=\optFNR(\alpha)$. 
(If $\alpha_1=\alpha_2=\alpha$, then $\oracle(\alpha)\equiv\maroracle(\alpha)$.)
\end{theorem}

Notably, the oracle in \prettyref{thm:ora_FNR_vs_ora_mFNR} that achieves the optimal tradeoff between FDR and FNR is \emph{not} separable if $\alpha_1\neq\alpha_2$. In fact, it randomizes between two separable rules (two NP tests \prettyref{eq:deltaNP} with appropriate parameters). However, as we mentioned before, the suboptimality of separable rules is not due to the allowance for
randomization. It is easy to see\footnote{To see this, note that any randomized rule $\delta^n$ such that $\delta^n=\maroracle(c_1)$ with probability $p$ and $\delta^n=\maroracle(c_2)$ with probability $1-p$ can be simulated by a deterministic rule that thresholds the last observation: If $X_n<c$, apply $\maroracle(c_1)$, and otherwise, apply $\maroracle(c_2)$, where $c$ is a constant such that $\Prob(X_n<c)=p$. Clearly, this rule has the same asymptotic performance as $\delta^n$.} that the characterization of the asymptotic $\optFNR(\alpha)=\lim_{n\to\infty}\optFNR_n(\alpha)$ in \prettyref{thm:ora_FNR_vs_ora_mFNR} holds verbatim even if we restrict to deterministic rules in the definition \prettyref{eq:opt_prob_FDR&FNR} of $\optFNR_n(\alpha)$. 
In this perspective, randomization may be viewed more as a proof technique\footnote{That randomization is useful may appear counterintuitive in view of the classical results on testing simple hypotheses: Although the Neyman-Pearson theory may stipulate randomization for finite sample sizes, in the large-sample limit simple likelihood ratio tests are optimal in terms of the leading-order asymptotics \cite{Hoeffding65,blahut1974hypothesis}. In contrast, the oracle \prettyref{eq:asym_opt_sol_of_FDR&FNR} randomizes for optimal FDR control even when $n$ tends to infinity.}
for establishing the fundamental limit in \prettyref{thm:ora_FNR_vs_ora_mFNR}. Since such randomized rules may cause high variability of the FDP\footnote{As shown in \prettyref{ex:gaussian}, when the signal is very weak, the curve $\optmFNR(\alpha)$ is entirely above the linear line $\pi_1(1-\frac{\alpha}{\pi_0})$, so the oracle \prettyref{eq:asym_opt_sol_of_FDR&FNR} simply randomizes between all rejections and all acceptances, causing the realized FDP to be either $0$ or very close to $\pi_0$ which is usually much larger than the desired level $\alpha$.}, they are less favored in practice where the tests are usually performed once. A natural question is whether the best deterministic rule (which is certainly compound) may remedy the fluctuation of the randomized rule. In fact, 
for the Gaussian mixture model, when the signal is sufficiently, it is shown in \cite[Proposition F.1]{rosset2022optimal} that the best deterministic rule has the same behavior as the randomized oracle, either accepting all hypotheses or none. These results reveal the limitation of the FDR as a criterion, which fails to take into account the variability of the FDP. Motivated by this observation, next we consider the problem of \emph{FDP control with high probability}, where the probability that the FDP exceeds a threshold tends to zero as the number of hypotheses goes to infinity. Under this criterion, the procedure $\maroracle(\alpha)$ turns out to be optimal and the best tradeoff is given by the $\optmFNR(\alpha)$ curve. This is the second main result of this paper.

\begin{theorem}[Optimal FDP control with high probability]
\label{thm:high_prob_formulation}
In the two-group random mixture model, fix a level $\alpha$. 
\begin{itemize}
	\item For any $\beta\geq \optmFNR(\alpha)$, the procedure $\maroracle(\alpha)$ 
	in \prettyref{eq:deltaNP}  satisfies that for any $\epsilon > 0$,
\[
\Prob\pth{\FDP(\maroracle(\alpha)) \leq \alpha + \epsilon} = 1-o(1) \text{ and } \Prob\pth{\FNP(\maroracle(\alpha)) \leq \beta + \epsilon} = 1-o(1).
\]

	\item Conversely, suppose that a sequence of decision rules $\delta^n$ satisfies 
	$\Prob\pth{\FDP(\delta^n) \leq \alpha} = 1-o(1)$ as $n\to\infty$ and $\FNR(\delta^n) \leq \beta$. Then $\beta \geq \optmFNR(\alpha)$.  
\end{itemize}
\end{theorem}

The peril of controlling only the FDR, the expectation of the FDP, 
which does not control its fluctuation, was first noted by Genovese and Wasserman \cite{genovese2006exceedance}. To resolve this issue, they proposed to control instead the probability that the FDP exceeds a prescribed threshold, called the exceedance control of the FDP. (See \cite{chi2008positive,lehmann2012generalizations,sun2015false,dohler2020controlling,basu2021empirical} for more references.) The difference between the setup in \prettyref{thm:high_prob_formulation} and the standard exceedance control of the FDP is that \prettyref{thm:high_prob_formulation} requires that the FDP fulfills the desired level with high probability as opposed to a fixed probability, as the number of hypotheses tends to infinity.

\subsection{Optimality and suboptimality of Benjamini-Hochberg}\label{sec:SuboptBH}

As one of the most commonly used procedures for FDR control in practice, 
the Benjamini-Hochberg (BH) procedure has been extensively studied in previous work \cite{benjamini1995controlling,benjamini2001control,genovese2002operating,ferreira2006benjamini}. 
Assuming that the observations $X^n$ are the $p$-values and $f_0$ is $\Unif(0,1)$, the oracle (i.e., with known $\pi_0$)  BH rule is given by 
\begin{align}\label{eq:BH}
\delta^n_{\rm BH}(\alpha) = \pth{\indc{X_i \leq X_{(i^*)}}}_{i=1,\cdots,n}
\end{align}
where $i^* = \max\{i=0,1,\cdots,n:\,X_{(i)}\leq\frac{i\alpha}{n\pi_0}\}$ and $X_{(0)}\triangleq  0 \leq X_{(1)}\leq\cdots\leq X_{(n)}$ are the order statistics of $X_1,\ldots,X_n$. It is well-known that the (oracle) BH is asymptotically equivalent to rejecting all $p$-values below a fixed threshold \cite{genovese2002operating}; in particular, the limiting procedure for the oracle BH coincides with the procedure $\maroracle$ in \prettyref{eq:deltaNP}. 
The asymptotic properties of the oracle BH are summarized in the following proposition proved in \prettyref{app:asym-BH}; similar results can be found in \cite{genovese2002operating}.
\begin{proposition}\label{prop:asym-BH}
Suppose that $f_0(x)=\indc{0<x<1}$ and $f_1(x)$ is strictly decreasing on $(0,1)$ and the Lfdr $W_i$ defined in \prettyref{eq:Lfdri} is a continuous random variable. For any $\alpha\in(\frac{\pi_0}{f(0+)},\pi_0)$, where $f(0+)=\lim_{x\to 0^+}f(x)$, we have $\mFDR(\delta_{\rm BH}^n(\alpha))\to \alpha$, $\mFNR(\delta_{\rm BH}^n(\alpha))\to \optmFNR(\alpha)$, and $\FDP(\delta_{\rm BH}^n(\alpha)) \toprob \alpha$, $\FNP(\delta_{\rm BH}^n(\alpha))\toprob\optmFNR(\alpha)$, as $n\to\infty$.
\end{proposition}

\prettyref{prop:asym-BH} shows that the oracle BH asymptotically achieves the optimal mFDR-mFNR tradeoff\footnote{This optimality may not hold if the map from raw observations to p-values is not invertible. For an example see the two-sided test in the Gaussian mixture model in \cite{sun2007oracle}.}.
However, \prettyref{thm:ora_FNR_vs_ora_mFNR} shows that it is suboptimal for 
minimizing the FNR subject to the FDR constraint whenever $\optFNR(\alpha)<\optmFNR(\alpha)$.
This happens, for instance, 
in the simple normal location model in \prettyref{ex:gaussian}, 
where the oracle BH is asymptotically dominated by the trivial rule for some FDR level or even any FDR level when the signal is weak, i.e., $\mu$ close to zero.
On the other hand, oracle BH is optimal for controlling the FDP with high probability, in view of \prettyref{thm:high_prob_formulation}.

In addition to BH, another adaptive procedure was proposed by Sun and Cai in \cite[Section 4.2]{sun2007oracle}, which rejects according to the moving average of the ordered approximated Lfdr at the observations. It was proved in \cite[Theorem 5 and 6]{sun2007oracle} that under appropriate assumptions this data-driven procedure asymptotically achieves mFDR and mFNR at level $\alpha$ and the optimal $\optmFNR(\alpha)$, respectively. 
By slightly modifying its proof, its FDR and FNR can be shown to converge to the same. 
Therefore, the preceding conclusion on oracle BH also applies to this procedure, which, in particular, does not optimally control the FDR either. We provide some numerical comparison of these procedures in \prettyref{sec:Simulation}.

\subsection{Organization}\label{sec:Organization}
The rest of the paper is organized as follows. \prettyref{sec:Preliminaries} introduces the needed auxiliary definitions and lemmas for analyzing the optimization problem \eqref{eq:opt_prob_FDR&FNR}, and reduce the problem \eqref{eq:opt_prob_mFDR&mFNR} to a more explicit form. \prettyref{sec:proofs_of_main_thms} gives the proofs of our main theorems. In \prettyref{sec:DataDrivenProcedure}, we propose a data-driven procedure and prove that it asymptotically achieves the oracle performance given in \prettyref{thm:ora_FNR_vs_ora_mFNR}. Numerical experiments on synthetic data are carried out in \prettyref{sec:Simulation} to compare this data-drive procedure with some other adaptive methods such as the BH. We extend our results to the objective of maximizing the expected number of true discoveries in \prettyref{sec:EFN}, to fixed non-null proportion in \prettyref{sec:Fixed-group}, 
and to models with weakly dependent data in \prettyref{sec:dependent}. We conclude the paper with a short discussion in \prettyref{sec:Discussion}. Due to space constraints, omitted proofs are given in the supplementary material.

\section{Preliminaries}\label{sec:Preliminaries}
\subsection{Auxiliary definitions and lemmas}\label{sec:AuxiDefs}

In this subsection, we introduce some definitions and lemmas that will be used later to prove our main theorems. Let $W\triangleq \Lfdr(X)=\pi_0f_0(X)/f(X)$, where $X\sim f$ and $G(t) = \Prob\pth{W \leq t}$ be the cumulative distribution function (CDF) of $W$. Define its inverse as $G^{-1}(u) \triangleq  \inf\{t:\,G(t) \geq u\}$, $u\in[0,1]$. Let $\bar{G}(t) \triangleq  \Prob\pth{W < t}$. 
For any $u\in[0,1]$, we have $G\pth{G^{-1}(u)}\geq u$ and $\bar{G}\pth{G^{-1}(u)}\leq u$ by the definition of $G^{-1}$. Then there exists $p(u)\in[0,1]$ such that
\begin{align}
    \bar{G}\pth{G^{-1}(u)} + p(u)\cdot\Prob\pth{W=G^{-1}(u)} = u.\label{eq:p(u)}
\end{align}
Define a random variable
\begin{align}\label{eq:S_u}
    S_u=
    \begin{cases}
     1 & \text{if } W<G^{-1}(u)\\
     \Bern(p(u)) & \text{if } W = G^{-1}(u)\\
     0 & \text{if } W>G^{-1}(u)
    \end{cases}.
\end{align}
Then $S_u\sim\Bern(u)$ thanks to \eqref{eq:p(u)}. Define two functions $a(\cdot)$ and $b(\cdot)$ on $[0,1]$ as
\begin{align}\label{eq:a(u)}
a(u) = 
\begin{cases}
\Expect[W\,|\,S_u = 1] & \text{if } u > 0\\
0 & \text{if } u = 0
\end{cases}
\end{align}
and
\begin{align}\label{eq:b(u)}
b(u) = 
\begin{cases}
\Expect[1-W\,|\,S_u = 0] & \text{if } u < 1\\
0 & \text{if } u = 1
\end{cases}.
\end{align}
We have the following properties of $a$ and $b$. The proof is given in \prettyref{app:Property-a&b}.

\begin{lemma}\label{lmm:properties_of_a&b}
$a(u)$ is increasing on $[0,1]$ and continuous on $(0,1]$; $b(u)$ is decreasing on $[0,1]$ and continuous on $[0,1)$. 
\end{lemma}

For any $\alpha\in[0,1]$, define
\begin{align}
\begin{split}
    b^*(\alpha) \triangleq  & \inf_{u\in[0,1]} \,\,b(u)\\
    & \,\,\,\,\,\text{s.t.}\quad a(u) \leq \alpha
\end{split}
\label{eq:opt_prob_a&b}    
\end{align}
which satisfies the following properties (see \prettyref{app:ProofLmm-b^*} for a proof):
\begin{lemma}\label{lmm:properties_of_b}
The infimum $b^*(\alpha)$ is achieved at $u^*(\alpha) \triangleq  \sup\{u\in[0,1]: a(u)\leq\alpha\}$. In particular, $b^*(0)=\pi_1$ and $b^*(\alpha)=0$ for $\alpha\in[\pi_0,1]$. Furthermore, $b^*$ is decreasing and right-continuous on $[0,1]$. 
\end{lemma} 

\subsection{Optimal mFDR control}

The optimal procedure which achieves the smallest mFNR subject to a constraint on $\mFDR$ has been investigated in \cite{sun2007oracle,xie2011optimal}. 
As an auxiliary result for proving our main theorems, we formalize this result in the following theorem and give a proof in \prettyref{app:OptRulemFDR} for completeness.
\begin{theorem}[Optimal mFDR control] \label{thm:optimal_sol_of_mFDR&mFNR}
For each $\alpha\in[0,1]$, $\optmFNR(\alpha)=b^*(\alpha)$.
Furthermore, the optimal solution of \eqref{eq:opt_prob_mFDR&mFNR}
is given by $\maroracle(\alpha)= (\NP(X_1;\alpha),\cdots,\NP(X_n;\alpha))$, where 
\begin{align}\label{eq:optimal_sol_of_mFDR&mFNR}
    \NP(X_i;\alpha) \triangleq
    \begin{cases}
     1 & \text{if } W_i < G^{-1}(u^*(\alpha))\\
    \Bern(p(u^*(\alpha))) & \text{if } W_i = G^{-1}(u^*(\alpha))\\
     0 & \text{if } W_i > G^{-1}(u^*(\alpha))
    \end{cases},
\end{align}
$u^*(\alpha) = \sup\{u\in[0,1]: a(u)\leq\alpha\}$ is the optimal solution of \eqref{eq:opt_prob_a&b}, $p(\cdot)$ is defined in \eqref{eq:p(u)}, and $W_i$ is the Lfdr in 
\prettyref{eq:Lfdri}.
\end{theorem}

\section{Proofs of main results}\label{sec:proofs_of_main_thms}
\subsection{Proof of \prettyref{thm:ora_FNR_vs_ora_mFNR}}\label{sec:ProofThm1}

As a first step toward proving \prettyref{thm:ora_FNR_vs_ora_mFNR}, we show that there is no loss of generality to restrict to procedures that sort the Lfdr values and reject the smallest $K$ of them, where $K$ is data-driven and possibly randomized.
Note that such procedures are \emph{not} separable.
\begin{lemma}\label{lmm:opt_prob_FDR&FNR}
Let $W_{(1)}\leq W_{(2)}\leq\cdots\leq W_{(n)}$ be the order statistics of $W_1,\cdots,W_n$. Then the optimization problem \eqref{eq:opt_prob_FDR&FNR} is equivalent to
\begin{align}
\begin{split}
    \optFNR_n(\alpha) =& \inf \,\,\,\Expect\left[\frac{1}{n-K}\sum_{i= K+1}^n(1-W_{(i)})\right]\\
    & \,\text{\rm s.t.}\,\,\,\, \Expect\left[\frac{1}{K}\sum_{i=1}^K W_{(i)}\right] \leq \alpha
\end{split}
\label{eq:opt_prob_FDR&FNR'}
\end{align}
where the infimum is over all Markov kernels $P_{K|X^n}$ from $\calX^n$ to $\{0,1\cdots,n\}$, and we adopt the convention that $\frac{1}{K}\sum_{i=1}^K W_{(i)}=0$ if $K=0$ and $\frac{1}{n-K}\sum_{i= K+1}^n(1-W_{(i)})=0$ if $K=n$. 

Furthermore, it suffices to restrict to binary-valued kernels. Namely, for any $x^n\in\calX^n$, conditioning on $X^n=x^n$, $K$ takes at most two values.
\end{lemma}
\begin{proof}
Suppose $\delta^n= (\delta_1,\cdots,\delta_n)$ is a feasible solution of \eqref{eq:opt_prob_FDR&FNR} and consider $K\triangleq \sum_{i=1}^n\delta_i$, a random variable supported on $\{0,1,\cdots,n\}$. Note that $\delta^n \indep \theta^n \,\,|\, X^n$ and $\Expect[\theta_i|X^n]=\Expect[\theta_i|X_i]=\frac{\pi_1f_1(X_i)}{f(X_i)}=1-W_i$, $\Expect[1-\theta_i|X^n]=\Expect[1-\theta_i|X_i]=\frac{\pi_0f_0(X_i)}{f(X_i)}=W_i$. We have 
\begin{align*}
    &\Expect\left[\frac{\sum_{i=1}^n (1-\delta_i)\theta_i}{1\vee\sum_{i=1}^n (1-\delta_i)}\Big|X^n\right] = \Expect\left[\frac{\sum_{i=1}^n (1-\delta_i)(1-W_i)}{1\vee\sum_{i=1}^n (1-\delta_i)}\Big|X^n\right]\geq \Expect\left[\frac{1}{n-K}\sum_{i=K+1}^n (1-W_{(i)})\Big|X^n\right],\\
    &\Expect\left[\frac{\sum_{i=1}^n \delta_i(1-\theta_i)}{1\vee\sum_{i=1}^n \delta_i}\Big|X^n\right] = \Expect\left[\frac{\sum_{i=1}^n \delta_i W_i}{1\vee\sum_{i=1}^n \delta_i}\Big|X^n\right]\geq \Expect\left[\frac{1}{K}\sum_{i=1}^K W_{(i)}\Big|X^n\right].
\end{align*}
where the expectations on the right side are with respect to the possible randomization of the decision rules. Hence, 
\begin{align}
    \Expect\left[\frac{1}{K}\sum_{i=1}^K W_{(i)}\right]\leq \FDR(\delta^n)\leq\alpha\text{\,\,\, and\,\,\,}\Expect\left[\frac{1}{n-K}\sum_{i=K+1}^n (1-W_{(i)})\right]\leq \FNR(\delta^n).\label{eq:lb_for _FDR_FNR}
\end{align}

On the other hand, suppose $K$ is a feasible solution of \eqref{eq:opt_prob_FDR&FNR'} and consider the decision rule $\delta^n$ such that $\{i\in [n]\triangleq\{1,\cdots,n\}:\,\delta_i=1\}$ consists of the indices corresponding to $W_{(1)},\cdots,W_{(K)}$. We have
\begin{align*}
    \FDR(\delta^n) &= \Expect\left[\Expect\left[\frac{\sum_{i=1}^n \delta_i(1-\theta_i)}{1\vee\sum_{i=1}^n \delta_i}\Big|X^n\right]\right] =  \Expect\left[\Expect\left[\frac{\sum_{i=1}^n \delta_i W_i}{1\vee\sum_{i=1}^n \delta_i}\Big|X^n\right]\right]\\
    &= \Expect\left[\Expect\left[\frac{1}{K}\sum_{i=1}^K W_{(i)}\Big|X^n\right]\right]=\Expect\left[\frac{1}{K}\sum_{i=1}^K W_{(i)}\right]\leq\alpha,\text{ and}
\end{align*}
\begin{align*}
    \FNR(\delta^n) &= \Expect\left[\Expect\left[\frac{\sum_{i=1}^n (1-\delta_i)\theta_i}{1\vee\sum_{i=1}^n (1-\delta_i)}\Big|X^n\right]\right] =  \Expect\left[\Expect\left[\frac{\sum_{i=1}^n (1-\delta_i)(1-W_i)}{1\vee\sum_{i=1}^n (1-\delta_i)}\Big|X^n\right]\right]\\
    &= \Expect\left[\Expect\left[\frac{1}{n-K}\sum_{i=K+1}^n (1-W_{(i)})\Big|X^n\right]\right]=\Expect\left[\frac{1}{n-K}\sum_{i=K+1}^n (1-W_{(i)})\right].
\end{align*}
Therefore, the optimization problem \eqref{eq:opt_prob_FDR&FNR} is equivalent to \eqref{eq:opt_prob_FDR&FNR'}. 

Next, we show that it suffices to restrict to binary-valued kernels. For any fixed $X^n=x^n= (x_1,\cdots,x_n)$, let $w_i\triangleq \pi_0f_0(x_i)/f(x_i)$, $i=1,\cdots,n$ and $w_{(1)}\leq\cdots\leq w_{(n)}$ be their order statistics. Define $h(k)=h(k;x^n)\triangleq \frac{1}{n-k}\sum_{i=k+1}^n(1-w_{(i)})$ and $g(k)=g(k;x^n)\triangleq \frac{1}{k}\sum_{i=1}^k w_{(i)}$, $k=0,1,\cdots,n$. For any feasible $K$ of \eqref{eq:opt_prob_FDR&FNR'}, let $K(x^n)$ denote $K$ conditioned on $X^n=x^n$. Then $(\Expect[g(K(x^n))],\Expect[h(K(x^n))])$ is in the convex hull of the subset $\{(g(k),h(k)):\,k=0,1\cdots,n\}\subset\reals^2$. By the Carath\'eodory theorem \cite[Chapter 2, Theorem 18]{eggleston1958convexity}, there exist $k_1, k_2, k_3\in\{0,1,\cdots,n\}$, $k_1\leq k_2\leq k_3$, and $p_1,p_2,p_3\in[0,1]$, $p_1+p_2+p_3=1$ such that $\Expect[g(K(x^n))] = p_1g(k_1)+p_2g(k_2)+p_3g(k_3)$ and $\Expect[h(K(x^n))] = p_1h(k_1)+p_2h(k_2)+p_3h(k_3)$. To simplify notation, let $g_i=g(k_i)$ and $h_i=h(k_i)$, $i=1,2,3$. Since $g(k)$ is increasing and $h(k)$ is decreasing in $k$, we have $g_1\leq g_2\leq g_3$ and $h_1\geq h_2\geq h_3$. Then there exists $\lambda\in[0,1]$ such that $g_2=\lambda g_1+(1-\lambda)g_3$, and thus $\Expect[g(K(x^n))]=(p_1+\lambda p_2)g_1+(p_3+(1-\lambda)p_2)g_3$. If $h_2\geq\lambda h_1+(1-\lambda)h_3$, then $(p_1+\lambda p_2)h_1+(p_3+(1-\lambda)p_2)h_3= p_1h_1 + p_3h_3 + p_2(\lambda h_1 + (1-\lambda) h_3)\leq p_1h_1+p_2h_2+p_3h_3=\Expect[h(K(x^n))]$. Consider $\tilde{K}(x^n)$ such that $\Prob(\tilde{K}(x^n)=k_1) = p$ and $\Prob(\tilde{K}(x^n)=k_3) = 1-p$, where $p=p_1+\lambda p_2$. Then $\Expect[g(\tilde{K}(x^n))]=\Expect[g(K(x^n))]$ and $\Expect[h(\tilde{K}(x^n))]\leq\Expect[h(K(x^n))]$. Similar arguments also apply when $h_2<\lambda h_1+(1-\lambda)h_3$. Consider $\tilde{K}=\tilde{K}(X^n)$ and further take expectations over $X^n$. We have $\Expect\left[\frac{1}{\tilde{K}}\sum_{i=1}^{\tilde{K}} W_{(i)}\right]=\Expect[g(\tilde{K}(X^n);X^n)]=\Expect[g(K(X^n);X^n)]=\Expect\left[\frac{1}{K}\sum_{i=1}^{K} W_{(i)}\right]$ and $\Expect\left[\frac{1}{n-\tilde{K}}\sum_{i>\tilde{K}} (1-W_{(i)})\right]=\Expect[h(\tilde{K}(X^n);X^n)]\leq\Expect[h(K(X^n);X^n)]=\Expect\left[\frac{1}{n-K}\sum_{i>K} (1-W_{(i)})\right]$, \ie $K$ is dominated by $\tilde{K}$ which is determined by a binary-valued kernel.
\end{proof}

Next we proceed to discuss the optimal solution to \eqref{eq:opt_prob_FDR&FNR'} in the large-$n$ limit. Recall the definitions of $S_u$, $a(\cdot)$ and $b(\cdot)$ in \eqref{eq:S_u}--\eqref{eq:b(u)}. We have the following lemma.

\begin{lemma}\label{lmm:approx_of_obj&constr}
For any $\tau\in (0,1)$ and $n \geq 1/\tau$,
\begin{enumerate}
    \item[(a)] for any $\epsilon>0$
    \begin{align}
    \Prob\pth{\sup_{\tau n\leq k\leq n}\Big|\frac{1}{k}\sum_{i=1}^k W_{(i)} - a(k/n)\Big|>\epsilon} \leq  6 \exp(-n\tau^2\epsilon^2/2);\label{ineq:ApproxConstr}
    \end{align}
    \item[(b)] for any $\epsilon>0$
    \begin{align}
    \Prob\pth{\sup_{0\leq k\leq(1-\tau)n}\Big|\frac{1}{n-k}\sum_{i=k+1}^n (1-W_{(i)}) - b(k/n)\Big|>\epsilon} \leq  6 \exp(-n\tau^2\epsilon^2/2).\label{ineq:ApproxObj}
    \end{align}
\end{enumerate}
\end{lemma}
The proof is given in \prettyref{app:ProofMainLmm}. The following corollary follows directly.
\begin{corollary}\label{cor:UnionBound}
For any $\tau\in(0,1)$ and for any (sequence of) Markov kernels $P_{K_n|X^n}$ from $\calX^n$ to $\{0,1,\cdots,n\}$, 
\begin{align*}
    &\pth{\frac{1}{K_n}\sum_{i=1}^{K_n} W_{(i)} - a(K_n/n)}\cdot\indc{K_n\geq \tau n} \toprob  0\text{, and}\\
    &\pth{\frac{1}{n-K_n}\sum_{i=K_n+1}^n (1-W_{(i)}) - b(K_n/n)}\cdot\indc{K_n\leq (1-\tau)n} \toprob  0\text{ as }n\to\infty.
\end{align*}

\end{corollary}

Now, let us relate \eqref{eq:opt_prob_FDR&FNR}, or equivalently \eqref{eq:opt_prob_FDR&FNR'}, to another optimization problem defined below:
\begin{align}
\begin{split}
    b^{**}(\alpha) \triangleq  & \inf \,\,\,\Expect[b(U)]\\
    & \,\text{s.t.}\,\,\,\, \Expect[a(U)] \leq \alpha
\end{split}
\label{eq:opt_prob_a&b_random}
\end{align}
where the infimum is over all random variables $U$ supported on $[0,1]$. We show in the next lemma that \prettyref{eq:opt_prob_a&b_random} is in fact the convexified version of \prettyref{eq:opt_prob_a&b}. (See \prettyref{app:two-value} for the proof.)

\begin{lemma}\label{lmm:property_of_b_doule_star}
The function $b^{**}$ is the GCM of $b^*$ defined in \prettyref{eq:opt_prob_a&b}. Consequently, $b^{**}$ is convex on $[0,1]$ and thus continuous on $(0,1)$. In particular, $b^{**}(0)=\pi_1$ and $b^{**}(\alpha)=0$ for any $\alpha\in[\pi_0,1]$.

Furthermore, the infimum in \eqref{eq:opt_prob_a&b_random} is achieved by a binary-valued $U^*=U^*(\alpha)$, given by
\begin{align}\label{eq:U*}
U^*  =
\begin{cases}
u^*(\alpha_1)\quad&\text{with probability }\frac{\alpha_2-\alpha}{\alpha_2-\alpha_1}\\
u^*(\alpha_2)\quad&\text{with probability }\frac{\alpha-\alpha_1}{\alpha_2-\alpha_1}
\end{cases}
\end{align}
where $0\leq \alpha_1\leq\alpha\leq \alpha_2\leq 1$ and $u^*(\cdot)$ is the optimal solution of \eqref{eq:opt_prob_a&b} given in \prettyref{lmm:properties_of_b}. (In case that $\alpha_1 = \alpha_2 = \alpha$, we set $U^*\equiv u^*(\alpha)$.)
\end{lemma}

The following proposition, which is the crux of the proof, relates the optimization problem \eqref{eq:opt_prob_FDR&FNR}, or equivalently \eqref{eq:opt_prob_FDR&FNR'}, to the optimization problem \eqref{eq:opt_prob_a&b_random} in the large-$n$ limit.

\begin{proposition}\label{prop:asymp_opt}
For any $\alpha\in[0,1]$, $\optFNR(\alpha)\triangleq \lim
\limits_{n\to\infty}\optFNR_n(\alpha) =  b^{**}(\alpha)$.
\end{proposition}

We start with a simple observation paraphrased from  \cite[Theorem 1]{storey2003positive} (and proved in \prettyref{app:ProofpFDR} for completeness).
\begin{lemma}\label{lmm:eq_for_pFDR&pFNR}
Under the two-group model \eqref{eq:two-group_model}, for any decision rule $\delta^n = (\delta_1,\cdots,\delta_n)\in \{0,1\}^n$ such that $(\theta_1,\delta_1),\cdots,(\theta_n,\delta_n)$ are \iid pairs, we have $\FDR(\delta^n)\leq\Prob\pth{\theta_1 = 0\,|\,\delta_1=1}$ and $\FNR(\delta^n)\leq\Prob\pth{\theta_1 = 1\,|\,\delta_1=0}$.
\end{lemma}

\begin{proof}[Proof of \prettyref{prop:asymp_opt}]
First, we show the positive direction, namely $\optFNR_n(\alpha) \leq  b^{**}(\alpha)$ for any $n\in\mathbb{Z}_+$ and $\alpha\in[0,1]$. For any feasible $U$ of \eqref{eq:opt_prob_a&b_random}, consider a randomized rule $\delta^n = (\delta_1,\delta_2,\cdots,\delta_n)$ based on the Lfdr values \prettyref{eq:Lfdri}:
\begin{align}\label{eq:construction_of_delta}
    \delta_i= 
    \begin{cases}
        1 & \text{if } W_i < G^{-1}(U)\\
        \Bern(p(U)) & \text{if } W_i = G^{-1}(U)\\
        0 & \text{if } W_i > G^{-1}(U)
    \end{cases}
\end{align}
where $p(\cdot)$ is defined in \eqref{eq:p(u)}. For fixed $u\in[0,1]$, let $\delta^{n,u}=(\delta^u_1,\cdots,\delta^u_n)$ denote $\delta^n$ conditioned on $U=u$. Then $(\theta_1,\delta^u_1),\cdots,(\theta_n,\delta^u_n)$ are \iid pairs. By \prettyref{lmm:eq_for_pFDR&pFNR},
\begin{align*}
    \FDR(\delta^{n,u}) &\leq \Prob\pth{\theta_1 = 0\,|\,\delta^u_1=1}= \frac{1}{u}\Expect[\indc{\theta_1=0,\delta^u_1=1}]= \frac{\pi_0}{u}\Expect_0[\indc{\delta^u_1=1}]\\
    &= \frac{1}{u}\Expect\left[\frac{\pi_0f_0(X_1)}{f(X_1)}\indc{\delta^u_1=1}\right]= \frac{1}{u}\Expect\left[W_1\indc{\delta^u_1=1}\right] = a(u)
\end{align*}
where $\Expect_0$ stands for taking expectation under $f_0$. Similarly, we have $\FNR(\delta^{n,u}) \leq b(u)$. Thus,
\[
\FDR(\delta^n) = \Expect[\FDR(\delta^{n,U})] \leq \Expect[a(U)] = \alpha\text{ and }
\FNR(\delta^n) = \Expect[\FNR(\delta^{n,U})] \leq \Expect[b(U)].
\]
It follows that $\optFNR_n(\alpha) \leq \FNR(\delta^n) \leq \Expect[b(U)]$. Optimizing over all feasible $U$, we have $\optFNR_n(\alpha) \leq  b^{**}(\alpha)$.
    
Next, we prove the converse, $\liminf\limits_{n\to\infty} \optFNR_n(\alpha) \geq  b^{**}(\alpha)$ for any $\alpha\in(0,1)$. For any $\epsilon>0$, by the continuity of $a(\cdot)$ and $b(\cdot)$ in \prettyref{lmm:properties_of_a&b}, there exists $\tau >0$ such that $\epsilon > |a(1)-a(1-\tau)| = |\pi_0 - a(1-\tau)|$ and $\epsilon > |b(0)-b(\tau)| = |\pi_1 - b(\tau)|$. Suppose $K_n$ is a feasible solution of \eqref{eq:opt_prob_FDR&FNR'} such that $\Expect\left[\frac{1}{n-K_n}\sum_{i=K_n+1}^{n} (1-W_{(i)})\right] \leq \optFNR_n(\alpha) + \epsilon$. By \prettyref{cor:UnionBound}, we have $\Expect\left[\pth{\frac{1}{K_n}\sum_{i=1}^{K_n} W_{(i)} - a(K_n/n)}\cdot\indc{K_n\geq \tau n}\right] \to 0$, as $n\to\infty$. Therefore, there exists $N$ such that for any $n > N$ 
\begin{align}
   \left| \Expect\left[\pth{\frac{1}{K_n}\sum_{i=1}^{K_n} W_{(i)} - a\pth{K_n/n}}\cdot\indc{K_n\geq \tau n}\right] \right|< \epsilon.		\label{eq:union_approx_a}
\end{align}
Then we have 
\begin{align*}
    \alpha &\geq \Expect\left[\frac{1}{K_n}\sum_{i=1}^{K_n} W_{(i)}\right] \geq \Expect\left[\frac{1}{K_n}\sum_{i=1}^{K_n} W_{(i)}\cdot\indc{K_n\geq \tau n}\right]\\
    &\overset{\eqref{eq:union_approx_a}}{\geq} \Expect\left[a\pth{K_n/n}\cdot\indc{K_n\geq \tau n}\right] - \epsilon\\ 
    &= \Expect\left[a\pth{K_n/n}\cdot\indc{\tau n \leq K_n\leq (1-\tau) n}\right]
    + \Expect\left[a\pth{K_n/n}\cdot\indc{K_n > (1-\tau) n}\right] - \epsilon\\ 
    &\geq \Expect\left[a\pth{K_n/n}\cdot\indc{\tau n \leq K_n\leq (1-\tau) n}\right]
    + a\pth{1-\tau}\Prob\pth{K_n > (1-\tau) n} - \epsilon\\ &\geq \Expect\left[a\pth{K_n/n}\cdot\indc{\tau n \leq K_n \leq (1-\tau) n}\right]
    + \pi_0\Prob\pth{K_n > (1-\tau) n} - 2\epsilon
\end{align*}
where the second last inequality is by noting that $a(y)$ is increasing in $y$ and the last inequality is because of the choice of $\tau$. Similarly, we have
\begin{align*}
    \Expect\left[\frac{1}{n-K_n}\sum_{i=K_n+1}^{n} (1-W_{(i)})\right]\geq \Expect\left[b\pth{K_n/n}\cdot\indc{\tau n \leq K_n\leq (1-\tau) n}\right] + \pi_1\Prob\pth{K_n< \tau n} - 2\epsilon.
\end{align*}
Choose $U = \frac{K_n}{n}\cdot\indc{\tau n \leq K_n\leq (1-\tau)n} + \indc{K_n > (1-\tau) n}$. Then $U\in[0,1]$ almost surely and 
\begin{align*}
    \Expect[a(U)] &= \Expect\left[a\pth{K_n/n}\cdot\indc{\tau n \leq K_n\leq (1-\tau) n}\right] + \pi_0\Prob\pth{K_n>(1-\tau)n} \leq \alpha + 2\epsilon\\
    \Expect[b(U)] &= \Expect\left[b\pth{K_n/n}\cdot\indc{\tau n \leq K_n\leq (1-\tau) n}\right] + \pi_1\Prob\pth{K_n< \tau n}\\
    &\leq \Expect\left[\frac{1}{n-K_n}\sum_{i=K_n+1}^{n} (1-W_{(i)})\right] + 2\epsilon\leq \optFNR_n(\alpha) + 3\epsilon.
\end{align*}
Hence, we have $ b^{**}(\alpha+2\epsilon)\leq \Expect[b(U)] \leq \optFNR_n(\alpha) + 3\epsilon$, for any $n > N$. Letting $n\to\infty$, we get $\liminf\limits_{n\to\infty}\optFNR_n(\alpha) \geq  b^{**}(\alpha+2\epsilon) - 3\epsilon$. By the continuity of $ b^{**}$ in \prettyref{lmm:property_of_b_doule_star}, letting $\epsilon\to 0^+$, we get $\liminf\limits_{n\to\infty}\optFNR_n(\alpha) \geq  b^{**}(\alpha)$ for any $\alpha\in(0,1)$. 
    
Finally for the corner cases of $\alpha=0$ and $\alpha=1$, we have $\optFNR_n(0) =  b^{**}(0) = \pi_1$ and $\optFNR_n(1) =  b^{**}(1) = 0$. This completes the proof.
\end{proof}

To finish proving \prettyref{thm:ora_FNR_vs_ora_mFNR}, note that \prettyref{thm:optimal_sol_of_mFDR&mFNR} and \prettyref{lmm:property_of_b_doule_star} imply that $\optFNR$ is the GCM of $\optmFNR$.
Furthermore, by the construction of $\delta^n$ in \eqref{eq:construction_of_delta}, the following procedure
\begin{align}\label{eq:asymp_opt_sol}
    \delta_i=
    \begin{cases}
        1 & \text{if } W_i < G^{-1}(U^*)\\
        \Bern(p(U^*)) & \text{if } W_i = G^{-1}(U^*)\\
        0 & \text{if } W_i > G^{-1}(U^*)
    \end{cases},\quad i = 1,\cdots,n
\end{align}
where $U^*$ is given by \prettyref{eq:U*}, achieves $\FDR(\oracle(\alpha))\leq\alpha$ and $\lim\limits_{n\to\infty}\FNR(\oracle(\alpha))=b^{**}(\alpha)=\optFNR(\alpha)$. Combining \prettyref{thm:optimal_sol_of_mFDR&mFNR} and \prettyref{eq:U*}, we conclude that the oracle $\oracle(\alpha)$ in \eqref{eq:asym_opt_sol_of_FDR&FNR} is asymptotically optimal, thus completing the proof of \prettyref{thm:ora_FNR_vs_ora_mFNR}.

We remark here on how to decide $\alpha_1$ and $\alpha_2$. For a given $\alpha$, if $\optFNR(\alpha)=\optmFNR(\alpha)$, then $\alpha_1=\alpha_2=\alpha$, and thus $\oracle(\alpha)$ coincides with $\maroracle(\alpha)$. If $\optFNR(\alpha)<\optmFNR(\alpha)$, \ie $b^{**}(\alpha)<b^*(\alpha)$, we have $\alpha_1<\alpha<\alpha_2$. Let $p=(\alpha_2-\alpha)/(\alpha_2-\alpha_1)$. Since $b^{**}(\alpha) = \Expect[b(U^*)] = pb(u^*(\alpha_1)) + (1-p)b(u^*(\alpha_2)) = pb^*(\alpha_1) + (1-p)b^*(\alpha_2) \geq pb^{**}(\alpha_1) + (1-p)b^{**}(\alpha_2) \geq b^{**}(p\alpha_1+(1-p)\alpha_2) = b^{**}(\alpha)$, we have $b^*(\alpha_1)=b^{**}(\alpha_1)$, $b^*(\alpha_2)=b^{**}(\alpha_2)$ and $b^{**}$ is linear on $[\alpha_1,\alpha_2]$. Thus, $\alpha_1 = \sup\{z\in[0,\alpha):\,b^{**}(z)=b^*(z)\}$ and $\alpha_2 = \inf\{z\in(\alpha,1]:\,b^{**}(z)=b^*(z)\}$. A special case is when $\optmFNR(\alpha) \geq \pi_1\pth{1 - \frac{\alpha}{\pi_0}}$ for any $\alpha\in[0,\pi_0]$, for instance, when $\mu$ is sufficiently small in \prettyref{ex:gaussian}. Then $\alpha_1=0$ and $\alpha_2=\pi_0$ for any $\alpha\in(0,\pi_0)$, in other words, $\oracle(\alpha)$ always randomizes between $\maroracle(0)=(0,\cdots,0)$ and $\maroracle(\pi_0)=(1,\cdots,1)$. 

\subsection{Proof of \prettyref{thm:high_prob_formulation}}\label{sec:ProofThm2}

We start with an elementary lemma.
\begin{lemma}\label{lmm:high-prob}
Suppose $X$ and $Y$ are two random variables such that $X\in[0,1]$.
If $\Prob\pth{X \leq \alpha} \geq 1 - \tau$ for some $\tau\in[0,1]$, then $\Prob\pth{\Expect[X|Y] \leq \alpha + \sqrt{\tau}} \geq 1 - \sqrt{\tau}$.
\end{lemma}
\begin{proof}
Define $\mu(y) = \Prob\pth{X > \alpha\,|\,Y=y}$ for $y$ in the support of $Y$. Then
$$\Expect[X|Y=y] = \Expect[X\indc{X>\alpha}|Y=y] + \Expect[X\indc{X \leq\alpha}|Y=y] \leq \mu(y) + \alpha$$ by noting that $X\in[0,1]$. By assumption, $\tau \geq \Prob\pth{X > \alpha} = \Expect[\mu(Y)]$. It follows that
\[
\Prob\pth{\Expect[X|Y] > \alpha + \sqrt{\tau}} \leq \Prob\pth{\mu(Y)> \sqrt{\tau}} \leq \frac{\Expect[\mu(Y)]}{\sqrt{\tau}} \leq \sqrt{\tau}
\]
where the second inequality is by the Markov inequality.
\end{proof}

We first prove the positive part of \prettyref{thm:high_prob_formulation}. Let $\maroracle(\alpha)=(\delta_1,\ldots,\delta_n)$ be the optimal rule for \prettyref{eq:opt_prob_mFDR&mFNR} from  \prettyref{thm:optimal_sol_of_mFDR&mFNR}. Recall the definition of $S_u$ in \prettyref{eq:S_u}, $(W_i,\delta_i)$ are \iid copies of $(W,S_{u^*(\alpha))})$. 
If $u^*(\alpha)=0$, $\delta_i=0$ almost surely for any $i\in[n]$, so we have $\text{FDP}(\maroracle(\alpha))=0=a(u^*(\alpha))$. If $u^*(\alpha)>0$, by the Law of Larger Numbers we have
\begin{align} 
    &\pth{\frac{1}{n}\sum_{i=1}^n\delta_i} \vee \frac{1}{n} \toprob  \Expect[S_{u^*(\alpha))}] \vee 0 = u^*(\alpha)\quad\text{as }n\to\infty,\label{convergence of denominator}\\
    &\frac{1}{n}\sum_{i=1}^n\delta_i(1-\theta_i) \toprob \pi_0\Expect_0[S_{u^*(\alpha))}] = \Expect[W\indc{S_{u^*(\alpha))}=1}]\quad\text{as }n\to\infty.\label{convergence of numerator}
\end{align}
Combining \eqref{convergence of denominator} and \eqref{convergence of numerator} and by the Slutsky's theorem, we have 
\[
\text{FDP}(\maroracle(\alpha)) \toprob  \frac{\Expect[W\indc{S_{u^*(\alpha))}=1}]}{u^*(\alpha)} = a(u^*(\alpha)) \leq \alpha.
\]
Similarly, $\text{FNP}(\maroracle(\alpha)) \toprob  b(u^*(\alpha)) =  b^*(\alpha)=\optmFNR(\alpha)$.

Next, we prove the negative part of \prettyref{thm:high_prob_formulation}. 
Suppose $\Prob\pth{\FDP(\delta^n) \leq \alpha} \geq 1-\tau_n$ for some $\tau_n=o(1)$ and $\FNR(\delta^n) \leq \beta$. We aim to show $\beta \geq \optmFNR(\alpha)=b^*(\alpha)$. The proof is trivial when $\alpha\geq\pi_0$ or $\beta\geq\pi_1$, so we assume $\alpha<\pi_0$ and $\beta<\pi_1$. 
By \prettyref{lmm:high-prob}, we have \[\Prob\pth{\Expect[\FDP(\delta^n)|\delta^n,X^n] \leq \alpha + \sqrt{\tau_n}} \geq 1-\sqrt{\tau_n}.\] Note that 
\begin{align*}
    \Expect[\FDP(\delta^n)|\delta^n,X^n] &= \Expect\left[\frac{\sum_{i=1}^n \delta_i(1-\theta_i)}{1\vee\sum_{i=1}^n \delta_i} \bigg| \delta^n, X^n\right] = \frac{\sum_{i=1}^n \delta_i \Expect[1-\theta_i|\delta^n,X^n]}{1\vee\sum_{i=1}^n \delta_i}\\
    &= \frac{\sum_{i=1}^n \delta_i \Expect[1-\theta_i|X_i]}{1\vee\sum_{i=1}^n \delta_i} = \frac{\sum_{i=1}^n \delta_i W_i}{1\vee\sum_{i=1}^n \delta_i} \geq \frac{1}{K_n} \sum_{i=1}^{K_n} W_{(i)}
\end{align*}
where $K_n\triangleq \sum_{i=1}^n \delta_i$, and $\frac{1}{K_n} \sum_{i=1}^{K_n} W_{(i)} \equiv 0$ if $K_n = 0$.
It follows 
\begin{align}
    \Prob\pth{\frac{1}{K_n} \sum_{i=1}^{K_n} W_{(i)} \leq \alpha + \sqrt{\tau_n}} \geq 1-\sqrt{\tau_n}.\label{eq:ineq1}
\end{align}
For any $\epsilon>0$, we show the following:
\begin{itemize}
    \item [(a)] $\lim_{n\to\infty}\Prob\pth{a(K_n/n)\leq\alpha+\sqrt{\tau_n}+\epsilon/2}=1$.
    \item [(b)] $\Expect[b(K_n/n)]\leq\beta+2\epsilon$ for sufficiently large $n$.
\end{itemize}
We first prove (a). Without loss of generality, assume $\epsilon\in(0,\pi_0-\alpha)$. By the continuity of $a$ on $(0,1]$ (\prettyref{lmm:properties_of_a&b}) and $a(1)=\pi_0$, there exists $\delta\in(0,1)$ such that $a(1-\delta)>\alpha+\epsilon$. Denote $\frac{1}{K_n} \sum_{i=1}^{K_n} W_{(i)}$ by $\overline{W}_{K_n}$. Following \eqref{eq:ineq1}, we have 
\begin{align*}
    1-\sqrt{\tau_n} \leq \Prob\pth{\overline{W}_{K_n}\leq \alpha + \sqrt{\tau_n},\,\frac{K_n}{n}>\delta} + \Prob\pth{\overline{W}_{K_n}\leq \alpha + \sqrt{\tau_n},\,\frac{K_n}{n} \leq \delta}.
\end{align*}
Note that $\overline{W}_{K_n}$ is increasing in $K_n$. Hence, 
\begin{align*}
    &\quad\Prob\pth{\overline{W}_{K_n}\leq \alpha + \sqrt{\tau_n},\,\frac{K_n}{n} \leq \delta} = \Prob\pth{\overline{W}_{\delta n}\leq \alpha + \sqrt{\tau_n},\,\frac{K_n}{n} \leq \delta}\\
    &\leq \Prob\pth{\overline{W}_{\delta n}\leq \alpha + \sqrt{\tau_n},\,\frac{K_n}{n} \leq \delta,\,|\overline{W}_{\delta n}-a(\delta)|\leq\frac{\epsilon}{2}}\\
    &+ \Prob\pth{\overline{W}_{\delta n}\leq \alpha + \sqrt{\tau_n},\,\frac{K_n}{n} \leq \delta,\,|\overline{W}_{\delta n}-a(\delta)|>\frac{\epsilon}{2}}\\
    &\leq \Prob\pth{a(\delta)\leq \alpha+\sqrt{\tau_n}+\frac{\epsilon}{2},\,\frac{K_n}{n} \leq \delta} + \Prob\pth{|\overline{W}_{\delta n}-a(\delta)|>\frac{\epsilon}{2}}\\
    &= \Prob\pth{a(\delta)\leq \alpha+\sqrt{\tau_n}+\frac{\epsilon}{2},\,\frac{K_n}{n} \leq \delta} + o(1)\\
    &= \Prob\pth{a(K_n/n)\leq \alpha+\sqrt{\tau_n}+\frac{\epsilon}{2},\,\frac{K_n}{n} \leq \delta} + o(1),
\end{align*}
where the second last equality is because of \prettyref{cor:UnionBound} and the last equality is by noting that $a$ is a increasing function. Next, 
\begin{align*}
    &\quad\Prob\pth{\overline{W}_{K_n}\leq \alpha + \sqrt{\tau_n},\,\frac{K_n}{n}>\delta}\\
    &\leq \Prob\pth{\overline{W}_{K_n}\leq \alpha + \sqrt{\tau_n},\,\frac{K_n}{n} >\delta,\,|\overline{W}_{K_n}-a(K_n/n)|\leq\frac{\epsilon}{2}}\\
    &+ \Prob\pth{\overline{W}_{K_n}\leq \alpha + \sqrt{\tau_n},\,\frac{K_n}{n} >\delta,\,|\overline{W}_{K_n}-a(K_n/n)|>\frac{\epsilon}{2}}\\
    &\leq \Prob\pth{a(K_n/n)\leq \alpha+\sqrt{\tau_n}+\frac{\epsilon}{2},\,\frac{K_n}{n} >\delta} + \Prob\pth{\frac{K_n}{n} >\delta,\,|\overline{W}_{K_n}-a(K_n/n)|>\frac{\epsilon}{2}}\\
    &= \Prob\pth{a(K_n/n)\leq \alpha+\sqrt{\tau_n}+\frac{\epsilon}{2},\,\frac{K_n}{n} >\delta} + o(1)\quad\text{by \prettyref{cor:UnionBound}.}
\end{align*}
Since $a$ is a increasing function and $a(1-\delta)>\alpha+\epsilon>\alpha+\sqrt{\tau_n}+\frac{\epsilon}{2}$ for sufficiently large $n$, then $\Prob\pth{a(K_n/n)\leq \alpha+\sqrt{\tau_n}+\frac{\epsilon}{2},\,\frac{K_n}{n}>\delta} = \Prob\pth{a(K_n/n)\leq \alpha+\sqrt{\tau_n}+\frac{\epsilon}{2},\,\delta<\frac{K_n}{n}\leq 1-\delta}$. Therefore,
\begin{align*}
    1-\sqrt{\tau_n}&\leq \Prob\pth{a(K_n/n)\leq \alpha+\sqrt{\tau_n}+\frac{\epsilon}{2},\,\frac{K_n}{n} \leq \delta}\\
    &+ \Prob\pth{a(K_n/n)\leq \alpha+\sqrt{\tau_n}+\frac{\epsilon}{2},\,\delta<\frac{K_n}{n}\leq 1-\delta} + o(1)\\
    &= \Prob\pth{a(K_n/n)\leq \alpha+\sqrt{\tau_n}+\frac{\epsilon}{2},\,\frac{K_n}{n}\leq 1-\delta} + o(1).
\end{align*}
Taking the limits for both sides, we have $\lim\limits_{n\to\infty}\Prob(K_n/n>1-\delta)=0$ and $\lim\limits_{n\to\infty}\Prob(a(K_n/n)\leq \alpha+\sqrt{\tau_n}+\frac{\epsilon}{2})=1$.

Next, we prove (b). By the assumption $\beta \geq \FNR(\delta^n)$ and noting that $$\FNR(\delta^n)\geq \Expect\left[\frac{1}{n-K_n}\sum_{i=K_n+1}^{n} (1-W_{(i)})\right]$$ in \prettyref{eq:lb_for _FDR_FNR}, we have $$\beta \geq \Expect\left[\frac{1}{n-K_n}\sum_{i=K_n+1}^{n} (1-W_{(i)})\right] \geq \Expect\left[\frac{1}{n-K_n}\sum_{i=K_n+1}^{n} (1-W_{(i)})\cdot\indc{K_n/n\leq 1-\delta}\right],$$ where $\delta$ is the one chosen in part (a). By \prettyref{cor:UnionBound} and noting that $\Prob(K_n/n>1-\delta)=o(1)$, for sufficiently large $n$, we have
\begin{align*}
  &\quad\Expect\left[\frac{1}{n-K_n}\sum_{i=K_n+1}^{n} (1-W_{(i)})\cdot\indc{K_n/n\leq 1-\delta}\right]\geq\Expect\left[b\pth{K_n/n}\cdot\indc{K_n/n\leq 1-\delta}\right] - \epsilon\\  
  &=\Expect\left[b\pth{K_n/n}\right] - \Expect\left[b\pth{K_n/n}\cdot\indc{K_n/n>1-\delta}\right] - \epsilon\\
  &\geq \Expect\left[b\pth{K_n/n}\right] - \Prob\pth{K_n/n>1-\delta} - \epsilon \geq \Expect\left[b\pth{K_n/n}\right] - 2\epsilon.
\end{align*}
Therefore, $\Expect[b(K_n/n)]\leq\beta+2\epsilon$ for sufficiently large $n$. 

Finally, let $E_1^n=\{b(K_n/n)\leq\beta+3\epsilon\}$ and $E_2^n=\{a(K_n/n)\leq \alpha+\sqrt{\tau_n}+\frac{\epsilon}{2}\}$. By (b), we have $\beta+2\epsilon\geq\Expect[b(K_n/n)\cdot\indc{b(K_n/n)>\beta+3\epsilon}]\geq(\beta+3\epsilon)\Prob((E_1^n)^c)$. Hence, $\Prob((E_1^n)^c)\leq\frac{\beta+2\epsilon}{\beta+3\epsilon}$. By (a), for sufficiently large $n$, $\Prob(E_2^n)\geq\frac{\beta+2.5\epsilon}{\beta+3\epsilon}$. Then, $\Prob(E_1^n\cap E_2^n)\geq\Prob(E_2^n)-\Prob((E_1^n)^c)\geq\frac{\epsilon}{2(\beta+3\epsilon)}>0$. Hence, there exists $k_n$ such that $b(\frac{k_n}{n})\leq\beta+3\epsilon$ and $a(\frac{k_n}{n})\leq\alpha+\sqrt{\tau_n}+\frac{\epsilon}{2}$. By the definition of $b^*$ in \prettyref{eq:opt_prob_a&b}, we have $b^*(\alpha+\sqrt{\tau_n}+\frac{\epsilon}{2})\leq\beta+3\epsilon$. By first letting $n\to\infty$ and then $\epsilon\downarrow 0$ and noting that $b^*$ is right-continuous (\prettyref{lmm:properties_of_b}), we obtain $\beta\geq b^*(\alpha)$.

\section{Data-driven procedure}\label{sec:DataDrivenProcedure}
The preceding Theorem \ref{thm:ora_FNR_vs_ora_mFNR} determines the optimal FDR-FNR tradeoff for every null and non-null densities $f_0$ and $f_1$. 
Complementing this result, in this section we show that the same optimality can be attained without knowing $f_1$, assuming only the knowledge of $f_0$ and $\pi_0$.
Since this adaptive procedure is based on approximating the randomized oracle rule \prettyref{eq:asym_opt_sol_of_FDR&FNR}, it is regarded more as a proof of concept than a practical recommendation that allows us 
to compare its performance with existing approaches (that do not depend on $f_1$ either) and demonstrate the suboptimality of the latter in a wide range of scenarios. (See \prettyref{sec:Simulation}.)

Assume that $\pi_0$ and $f_0$ are known and that $W= \Lfdr(X)=\pi_0f_0(X)/f(X)$ is a continuous random variable with CDF $G$ and pdf $g$. 
 First, for $y\in[0,1]$, define
\begin{align}\label{eq:A(y)&B(y)}
A(y) = 
\begin{cases}
\Expect[W\,|\,W\leq y] & \text{if } \Prob(W \leq y) > 0\\
0 & \text{if } \Prob(W \leq y) = 0
\end{cases},\quad B(y)=
\begin{cases}
\Expect[W\,|\,W> y] & \text{if } \Prob(W > y) > 0\\
0 & \text{if } \Prob(W > y) = 0
\end{cases}.
\end{align}
By replacing $G^{-1}(u)$ with $y$, \eqref{eq:opt_prob_a&b} is equivalent to
\begin{align}
    \begin{split}
     b^*(\alpha) = & \inf_{y\in[0,1]} \,\,\,B(y)\\
    & \,\,\,\,\,\text{s.t.}\quad A(y)\leq \alpha
    \end{split}.
    \label{eq:opt_prob_A&B}
\end{align}
Since $A$ is increasing and left-continuous on $[0,1]$ and $B$ is decreasing on $[0,1]$, \eqref{eq:opt_prob_A&B} has an optimal solution
\begin{align}
    y^*(\alpha) = \sup\{y\in[0,1]:\,A(y)\leq\alpha\}.\label{eq:y*}
\end{align}

Suppose that we are provided with an estimate of the marginal pdf $f$, denoted by $\hat{f}$, based on the observations $X_1,\cdots,X_n$. (For example, under smoothness conditions, $\hat{f}$ can be a kernel density estimator; under monotonicity constraints, $\hat{f}$ can be the Grenander estimator \cite{birge1989grenader}.) We first estimate $ b^*$, the optimal objective of \eqref{eq:opt_prob_A&B}. The idea is the same as the adaptive $z$-value-based procedure proposed in \cite[Section 4.2]{sun2007oracle}. We estimate $W_i$ by $\widehat{W}_i \triangleq  \frac{\pi_0 f_0(X_i)}{\hat{f}(X_i)}\wedge 1$ for $i=1,\cdots,n$. Then by substituting the empirical distribution of $\widehat{W}_1,\cdots,\widehat{W}_n$ for the true distribution of $W$ in the expectations in \eqref{eq:A(y)&B(y)}, $A(y)$ and $B(y)$ are estimated by 
\begin{align}
    \widehat{A}(y)\triangleq  \frac{\sum_{i=1}^n \widehat{W}_i\indc{\widehat{W}_i\leq y}}{1\vee\sum_{i=1}^n \indc{\widehat{W}_i\leq y}}\text{ and }\widehat{B}(y)\triangleq  \frac{\sum_{i=1}^n (1-\widehat{W}_i)\indc{\widehat{W}_i> y}}{1\vee\sum_{i=1}^n \indc{\widehat{W}_i> y}}.\label{eq:A_hat&B_hat}
\end{align}
Then we estimate $ b^*(\alpha)$ by replacing $A$ and $B$ with $\widehat{A}$ and $\widehat{B}$ in \eqref{eq:opt_prob_A&B}, \ie
\begin{align}
\begin{split}
    \widehat{b^*}(\alpha) \triangleq  & \inf_{y\in[0,1]} \,\,\,\widehat{B}(y)\\
    & \,\,\,\,\,\text{s.t.}\quad \widehat{A}(y) \leq \alpha
\end{split}.
\label{eq:OPT_3_hat}    
\end{align}
Since $\widehat{B}$ is decreasing, $\widehat{A}$ is increasing, and both are right-continuous and piecewise constant, the optimal solution of \eqref{eq:OPT_3_hat}, denoted by $\widehat{y^*}(\alpha)$, is any value in $[l,r)$, where $l = \max\{y\in\{0,\widehat{W}_1,\cdots,\widehat{W}_n\}:\widehat{A}(y)\leq \alpha\}$ and $r=\min\{y\in\{0,\widehat{W}_1,\cdots,\widehat{W}_n\}:\widehat{A}(y)> \alpha\}$. Since any choice of $\widehat{y^*}(\alpha)$ gives the same rejecting rule, without loss of generality, we choose
\begin{align}\label{eq:y*_hat}
\widehat{y^*}(\alpha)=\max\{y\in\{0,\widehat{W}_1,\cdots,\widehat{W}_n\}:\widehat{A}(y)\leq \alpha\}.
\end{align}
Thus, $\widehat{b^*}(\alpha) = \widehat{B}(\widehat{y^*}(\alpha))$ is decreasing, piecewise constant and right-continuous in $\alpha\in[0,1]$. It was proved in \cite{sun2007oracle} that the following decision rule (known as the adaptive $z$-value-based procedure),
\begin{align}\label{eq:Sun&Cai}
    \pth{\indc{\widehat{W}_1 \leq \widehat{y^*}(\alpha)},\cdots,\indc{\widehat{W}_n \leq \widehat{y^*}(\alpha)}},
\end{align}
asymptotically controls mFDR at level $\alpha$ and mFNR at level $ b^*(\alpha)$ under appropriate conditions. 
Using similar arguments, it can be shown that the FDR and FNR are also asymptotically controlled at the same levels. However, 
rule \prettyref{eq:Sun&Cai} need not to be optimal, as \prettyref{thm:ora_FNR_vs_ora_mFNR} shows that $\optFNR$ is the GCM of $b^*$. This suggests we should mimic the oracle in \prettyref{eq:asym_opt_sol_of_FDR&FNR} in a data-driven way. 

To this end, let us estimate $\optFNR$ by the GCM of $\widehat{b^*}$, denoted by $\widehat{\optFNR}$, which is piecewise linear and easy to compute.\footnote{For example, this can be done by using the function \texttt{gcmlcm()} in the R package \texttt{fdrtool}.}
Denote its knots by $(s_1,t_1),\cdots,(s_m,t_m)$, where $0=s_1<s_2<\cdots<s_m=1$ and $t_1\geq t_2\geq \cdots\geq t_m=0$. For any $\alpha\in(0,1)$, there exists a unique $k_0\in\{1,\cdots,m\}$ such that $\alpha\in[s_{k_0},s_{k_0+1})$. Let $\hat p = (s_{k_0+1}-\alpha)/(s_{k_0+1}-s_{k_0})$. Then we propose the following adaptive procedure: 
\begin{equation}\label{eq:globally-randomized}
    \begin{split}
    &\text{with probability }\hat p,\qquad\delta_i=\indc{\widehat{W}_i \leq \widehat{y^*}(s_{k_0})},\,i=1,\cdots,n\\
    &\text{with probability }1-\hat p,\,\,\delta_i=\indc{\widehat{W}_i \leq \widehat{y^*}(s_{k_0+1})},\,i=1,\cdots,n
    \end{split}
\end{equation}
where $\widehat{y^*}$ is given in \eqref{eq:y*_hat}. 
We summarize this procedure into \prettyref{algo:globally-randomized}. We show in \prettyref{prop:uniform_consistency} that all the intermediate estimators are uniformly consistent, and in \prettyref{prop:consistency} that under  
monotonicity assumptions on the density, the decision rule \prettyref{eq:globally-randomized} asymptotically achieves the oracle performance.
See  Appendices \ref{app:ProofConsistency1} and \ref{app:ProofConsistency2} for proofs.

\begin{algorithm}
\caption{Data-driven procedure}\label{algo:globally-randomized}
\begin{algorithmic}[1]  
\renewcommand{\algorithmicrequire}{\textbf{Input:}}
\State\textbf{Input:} FDR constraint $\alpha\in(0,1)$, null proportion $\pi_0\in(0,1)$, null pdf $f_0$, observed data $X^n=(X_1,\cdots,X_n)$, marginal density estimator $\hat{f}$ based on $X^n$.

\State Compute $\widehat{W_i} = \frac{\pi_0f_0(X_i)}{\hat{f}(X_i)}\wedge 1$ for $i=1,\cdots,n$.

\State Compute $a_0=\hat{A}(0)$, $b_0=\hat{B}(0)$, and $a_i = \hat{A}(\widehat{W_i})$, $b_i = \hat{B}(\widehat{W_i})$ defined in \eqref{eq:A_hat&B_hat}, for $i=1,\cdots,n$.

\State Compute the GCM of the right-continuous and piecewise constant function generated by points $(a_0,b_0),\cdots,(a_n,b_n)$ and $(1,0)$, and represent the GCM by its knots $(s_1,t_1),\cdots,(s_m,t_m)$, where $0=s_1<s_2<\cdots<s_m=1$ and $t_1\geq t_2\geq \cdots\geq t_m=0$. 

\State Choose the unique $k_0\in\{1,\cdots,m\}$ such that $\alpha\in[s_{k_0},s_{k_0+1})$ and compute $\hat p=\frac{s_{k_0+1}-\alpha}{s_{k_0+1}-s_{k_0}}$.

\State Compute $\widehat{y^*}(s_{k_0})$ and $\widehat{y^*}(s_{k_0+1})$ defined in \eqref{eq:y*_hat}.

\State Apply the decision rule \eqref{eq:globally-randomized}.

\end{algorithmic} 
\end{algorithm}

\begin{proposition}\label{prop:uniform_consistency}
Suppose $f_0(x)=\indc{0<x<1}$, $f_1(x)$ is decreasing and continuous on $(0,1)$, and $W$ has CDF $G$ and positive pdf $g$ on $(0,1)$. Let $\hat{f}$ be the Grenander estimator \cite{birge1989grenader} of $f$, namely, the slope of the least concave majorant of the empirical CDF of the data $X_1,\cdots,X_n$.
Let $\widehat{W}_i = \frac{\pi_0}{\hat{f}(X_i)}\wedge 1$, for $i=1,\cdots,n$. Let $\widehat{A}$, $\widehat{B}$ and $\widehat{b^*}$ be defined in \eqref{eq:A_hat&B_hat}-\eqref{eq:OPT_3_hat}. Let $\widehat{\optFNR}$ be the GCM of $\widehat{b^*}$. Then we have $\|\widehat{A}-A\|_\infty\toprob 0$, $\|\widehat{B}-B\|_\infty\toprob 0$, $\|\widehat{b^*}- b^*\|_\infty\toprob 0$, and $\|\widehat{\optFNR}-\optFNR\|_\infty\toprob 0$, as $n\to\infty$, where $A$, $B$ and $ b^*$ are defined in \eqref{eq:A(y)&B(y)}-\eqref{eq:opt_prob_A&B}.
\end{proposition}

\begin{proposition}\label{prop:consistency}
Under the same conditions of \prettyref{prop:uniform_consistency}, the data-driven procedure $\delta^n$ in \eqref{eq:globally-randomized} is asymptotically optimal,  satisfying $\limsup\limits_{n\to\infty}\sup\limits_{\alpha\in[0,1]}\pth{\FDR(\delta^n)-\alpha}\leq 0$ and $\lim\limits_{n\to\infty}\sup\limits_{\alpha\in[0,1]}|\FNR(\delta^n)- \optFNR(\alpha)|= 0$.
\end{proposition}

\section{Simulation studies}\label{sec:Simulation}
In this section, we carry out experiments based on the data generated from the Gaussian location model, \ie, $X_i \sim \calN(0,1)$ under the null and $\calN(\mu,1)$ for some $\mu>0$ under the alternative, to compare the performance of three data-driven procedures: The oracle BH procedure given in \eqref{eq:BH}, Sun and Cai's procedure given in \eqref{eq:Sun&Cai}, and \prettyref{algo:globally-randomized}. In these experiments, we assume that the null proportion $\pi_0$ is known and we convert the raw observations $X_i$ to their p-values $1-\Phi(X_i)$ so that they follow a uniform distribution under the null. By direct calculation, the density of $1-\Phi(X_i)$ under the alternative is monotone, so we can use the Grenander estimator \cite{birge1989grenader} to estimate the marginal density.\footnote{This was done by the function \texttt{grenander()} in the R package \texttt{fdrtool}.}
We choose $\pi_0=0.75,0.99$ and the sample size $n=1000$. We use Monte-Carlo method (10000 identical and independent trials) to approximate the FDR and FNR.

The top panels of \prettyref{fig:BH_vs_SC_vs_GR_normal} correspond to the case of $\pi_0=0.75$. For $\mu=1$ (left), we have $\optmFNR(\alpha) > \optFNR(\alpha)$ for any $\alpha\in(0,\pi_0)$ as shown in \prettyref{ex:gaussian}, so \prettyref{algo:globally-randomized} which is designed for optimal FDR control attains smaller FNR than the oracle BH and Sun\&Cai. 
For $\mu=2$ (right), $\optmFNR$ appears almost convex, so $\optFNR$ nearly coincides with $\optmFNR$. (Nevertheless, as shown in \prettyref{ex:gaussian}, $\optFNR(\alpha)$ is strictly smaller than $\optmFNR(\alpha)$ for small $\alpha$.)
In this case, all three procedures exhibit similar performance. The bottom panels of \prettyref{fig:BH_vs_SC_vs_GR_normal} display the case of $\pi_0=0.99$, \ie, the signals are very sparse. For $\mu=2$ (left), \prettyref{algo:globally-randomized} which approximates the optimal FDR-FNR tradeoff outperforms the oracle BH and Sun\&Cai. For $\mu=3$ (right), in which case $\optFNR$ nearly coincides with $\optmFNR$, \prettyref{algo:globally-randomized} performs worse than the others due to the density estimation errors for the given sample size.  

\begin{figure}[thpb]
	\centering
	\includegraphics[scale=0.17]{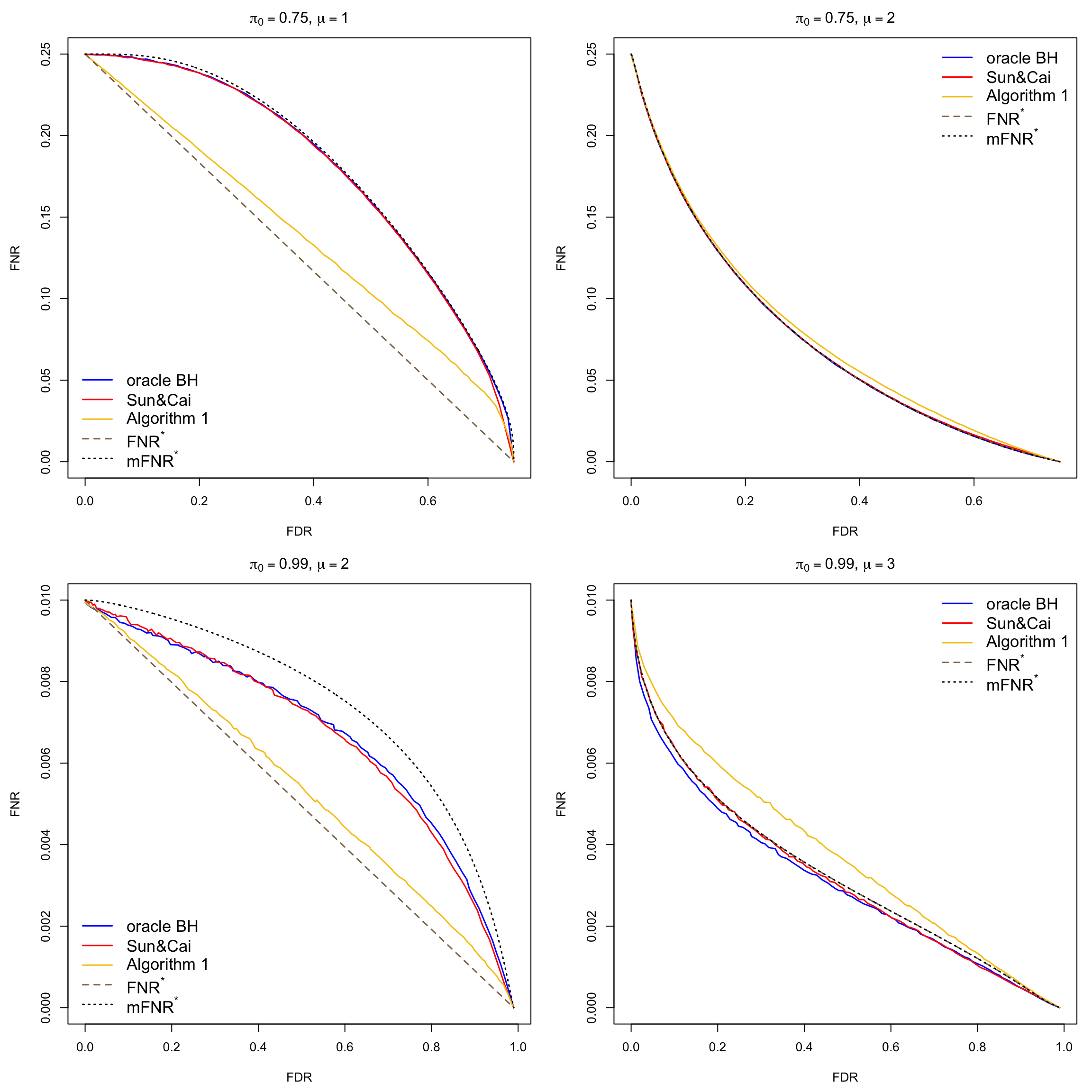}
	\caption{Comparisons between the oracle BH, Sun\&Cai and \prettyref{algo:globally-randomized} in the normal mean model, with $\pi_0=0.75$, $\mu=1,2$ (top) and $\pi_0=0.75$, $\mu=2,3$ (bottom).}
	\label{fig:BH_vs_SC_vs_GR_normal}
\end{figure}

Overall, both experiments verify the conclusions in \prettyref{sec:SuboptBH} that the oracle BH and Sun\&Cai achieve the optimal mFDR-mFNR tradeoff but not 
the FDR-FNR tradeoff. In particular, the sparser and rarer the signals are, these procedures are more likely to be suboptimal. For example, with the same $\mu=2$, the oracle BH and Sun\&Cai are optimal for $\pi_0=0.75$ while suboptimal for $\pi_0=0.99$; with the same $\pi_0=0.75$, they are optimal for $\mu=2$ but suboptimal for $\mu=1$. For a more comprehensive comparison, we provide in \prettyref{app:ComparisonsFixedLevel} a table of the estimated EFN, FNR, FDR, as well as the variance of FDP, for the oracle BH, Sun\&Cai, and \prettyref{algo:globally-randomized} for a fixed FDR level.

\section{Extensions}\label{sec:Extensions}
\subsection{Maximizing the expected number of true discoveries}\label{sec:EFN}

Our main results in \prettyref{sec:MainThm} focused on the optimal FDR-FNR tradeoff as well as their high-probability version, showing that separable rules are in general suboptimal for the former but optimal for the latter. In this section we extend these results to a different objective previously mentioned in 
\prettyref{sec:mFDR}, namely, minimizing the number of false non-discoveries divided by $n$ in expectation (EFN) and with high probability, whiling controlling the FDR. 
The following two theorems, extending Theorems \ref{thm:ora_FNR_vs_ora_mFNR} and 
\ref{thm:high_prob_formulation}, 
show that similar conclusions, both qualitative and quantitative, continue to apply to this different objective function.
 
\begin{theorem}\label{thm:FDR-ETD}
Recall that $\optEFNnorm_n(\alpha)$ defined in \prettyref{eq:opt_prob_FDR&EFN} denotes the the optimal EFN 
subject to that FDR is at most $\alpha$ in the two-group random mixture model with $n$ hypotheses.
The limit $\optEFNnorm(\alpha)\triangleq\lim\limits_{n\to\infty}\optEFNnorm_n(\alpha)$ exists, which, as a function of $\alpha$, is the GCM of the function $\optmEFNnorm$. 
\end{theorem}

\begin{theorem}\label{thm:high_prob_formulation_1}
Let $\FN(\delta^n)\triangleq\frac{1}{n}\sum_{i=1}^n\theta_i(1-\delta_i)$ be the number of false non-discoveries divided by $n$. In the two-group random mixture model, fix a level $\alpha$.
\begin{itemize}
	\item For any $\beta\geq \optmEFNnorm(\alpha)$, the procedure $\maroracle(\alpha)$ 
	in \prettyref{eq:deltaNP}  satisfies that for any $\epsilon > 0$,
\[
\Prob\pth{\FDP(\maroracle(\alpha)) \leq \alpha + \epsilon} = 1-o(1) \text{ and } \Prob(\FN(\maroracle(\alpha)) \leq \beta + \epsilon) = 1-o(1).
\]

	\item Conversely, suppose that a sequence of decision rules $\delta^n$ satisfies 
	$\Prob\pth{\FDP(\delta^n) \leq \alpha} = 1-o(1)$ as $n\to\infty$ and $\EFN(\delta^n) \leq \beta$. Then $\beta \geq \optmEFNnorm(\alpha)$. 
\end{itemize}
\end{theorem}

The proof of these results are similar to those of 
Theorems \ref{thm:ora_FNR_vs_ora_mFNR} and  \ref{thm:high_prob_formulation}. Here we only pointed out the main difference. Recall the definitions of $a$ and $b$ in \prettyref{eq:a(u)} and \prettyref{eq:b(u)}. Define $\tilde{b}(u)=(1-u)b(u)$ and
\begin{align}
\begin{split}
    \tilde{b}^*(\alpha) \triangleq  & \inf_{u\in[0,1]} \,\,\tilde{b}(u)=(1-u)b(u)\\
    & \,\,\,\,\,\text{s.t.}\quad a(u) \leq \alpha.
\end{split}
\label{eq:opt_prob_a&b'}    
\end{align}
Then, \prettyref{lmm:properties_of_b} holds verbatim for $\tilde{b}^*$ by applying the same arguments in \prettyref{app:ProofLmm-b^*}. Furthermore, we have the following lemma. (See the proof in \prettyref{app:OptRulemFDR}.)
\begin{lemma}\label{lmm:optimal_sol_of_mFDR&EFN}
For each $\alpha\in[0,1]$, we have $\optmEFNnorm(\alpha)=\tilde{b}^*(\alpha)$, and the optimal solution of \eqref{eq:opt_prob_mFDR&EFN}
is given by \prettyref{eq:optimal_sol_of_mFDR&mFNR}, the same as that of the problem \eqref{eq:opt_prob_mFDR&mFNR}. 
\end{lemma}

Next, we follow the same argument as the proofs of Theorems \ref{thm:ora_FNR_vs_ora_mFNR} and  \ref{thm:high_prob_formulation}. Akin to \prettyref{lmm:opt_prob_FDR&FNR}, the optimization problem \prettyref{eq:opt_prob_FDR&EFN} is equivalent to 
\begin{align}
\begin{split}
    \optEFNnorm_n(\alpha) =& \inf \,\,\,\Expect\left[\frac{1}{n}\sum_{i= K+1}^n(1-W_{(i)})\right]\\
    & \,\,\text{s.t.}\,\,\,\,\, \Expect\left[\frac{1}{K}\sum_{i=1}^K W_{(i)}\right] \leq \alpha,
\end{split}
\label{eq:opt_prob_FDR&EFN'}
\end{align}
where the infimum is over all Markov kernels $P_{K|X^n}$ from $\calX^n$ to $\{0,1\cdots,n\}$. Then, we relate \prettyref{eq:opt_prob_FDR&EFN'} 
to the following optimization problem
\begin{align}
\begin{split}
    \tilde{b}^{**}(\alpha) \triangleq  & \inf \,\,\,\Expect[\tilde{b}(U)]\\
    & \,\text{s.t.}\,\,\,\, \Expect[a(U)] \leq \alpha,
\end{split}
\label{eq:opt_prob_a&b'_random}
\end{align}
where $U$ is any random variable supported on $[0,1]$, by showing that $\lim\limits_{n\to\infty}\optEFNnorm_n(\alpha)=\tilde{b}^{**}(\alpha)$. The key step is to notice that $\frac{1}{n}\sum_{i= k+1}^n(1-W_{(i)})-\tilde{b}(\frac{k}{n})=(1-\frac{k}{n})(\frac{1}{n-k}\sum_{i= k+1}^n(1-W_{(i)})-b(\frac{k}{n}))$. Combined with \prettyref{cor:UnionBound}, for any $\tau\in(0,1)$ and for any (sequence of) Markov kernels $P_{K_n|X^n}$ from $\calX^n$ to $\{0,1,\cdots,n\}$, we have
\begin{align}
    \pth{\frac{1}{n}\sum_{i=K_n+1}^n (1-W_{(i)}) - \tilde{b}(K_n/n)}\cdot\indc{K_n\leq (1-\tau)n} \toprob  0\text{ as }n\to\infty.\label{eq:concentrate_b_tilde}
\end{align} 
The rest of proofs are provided in \prettyref{app:PoofsFDR-EFN}.

\subsection{Fixed non-null proportion}\label{sec:Fixed-group}
Our results so far are based on the two-group random mixture model in which the true labels of the hypotheses are independently drawn from $\Bern(\pi_1)$. 
This independence assumption is convenient for the technical analysis, in particular, the concentration result in \prettyref{lmm:approx_of_obj&constr}, to go through.
It is of interest to consider the model with a fixed number $n_1$ of non-nulls, which was originally considered in \cite{genovese2002operating}. We show in this section that the main results (\prettyref{thm:ora_FNR_vs_ora_mFNR} and \prettyref{thm:high_prob_formulation}) also hold for the \emph{fixed-proportion model}, 
in which the true labels $\theta^n$ are drawn uniformly at random from $\{\theta^n\in\{0,1\}^n:\sum_{i=1}^n\theta_i=n_1\}$, where $\frac{n_1}{n}\to\pi_1$ as $n\to\infty$, and given $\theta^n$, $X^n|\theta^n\sim\prod_{i=1}^nf_0^{1-\theta_i}f_1^{\theta_i}$. 

\begin{theorem}
\label{thm:ora_FNR_vs_ora_mFNR_fixed}
For the optimization problem \prettyref{eq:opt_prob_FDR&FNR} in the fixed-proportion model, the limit $\lim\limits_{n\to\infty}\optFNR_n(\alpha)$ exists and is given by $b^{**}(\alpha)$ defined in \prettyref{eq:opt_prob_a&b_random}, which is the GCM of the function $b^*$ defined in \prettyref{eq:opt_prob_a&b}. 
\end{theorem}

\begin{theorem}
\label{thm:high_prob_formulation_fixed}
In the fixed-proportion model, fix a level $\alpha$.
\begin{itemize}
	\item For any $\beta\geq b^*(\alpha)$, the procedure $\maroracle(\alpha)$ 
	in \prettyref{eq:deltaNP}  satisfies, for any $\epsilon > 0$,
\[
\Prob\pth{\FDP(\maroracle(\alpha)) \leq \alpha + \epsilon} = 1-o(1) \text{ and } \Prob\pth{\FNP(\maroracle(\alpha)) \leq \beta + \epsilon} = 1-o(1).
\]

	\item Conversely, suppose that a sequence of decision rules $\delta^n$ satisfies 
	$\Prob\pth{\FDP(\delta^n) \leq \alpha} = 1-o(1)$ as $n\to\infty$ and $\FNR(\delta^n) \leq \beta$. Then $\beta \geq b^*(\alpha)$. 
\end{itemize}
\end{theorem}

One of the major technical difficulties in proving these theorems is that in \prettyref{lmm:opt_prob_FDR&FNR} 
the local FDR $\Expect[1-\theta_i|X^n]$ now depends on the entire dataset $X^n$. Nevertheless, the following result shows that the same deterministic approximation as in 
\prettyref{cor:UnionBound} continues to hold. 
\begin{proposition}\label{prop:UnionBound_fixed}
In the fixed-proportion model, let $\widetilde{W}_i\triangleq\Expect[1-\theta_i|X^n]$ and $\widetilde{W}_{(1)},\cdots,\widetilde{W}_{(n)}$ be the order statistics. For any $\tau\in(0,1)$ and for any (sequence of) Markov kernels $P_{K_n|X^n}$ from $\calX^n$ to $\{0,1,\cdots,n\}$, we have
\begin{align}
    &\pth{\frac{1}{K_n}\sum_{i=1}^{K_n} \widetilde{W}_{(i)} - a(K_n/n)}\cdot\indc{K_n\geq \tau n} \toprob  0\text{, and}\label{eq:converge1}\\
    &\pth{\frac{1}{n-K_n}\sum_{i=K_n+1}^n (1-\widetilde{W}_{(i)}) - b(K_n/n)}\cdot\indc{K_n\leq (1-\tau)n} \toprob  0\text{ as }n\to\infty.\label{eq:converge2}
\end{align}
\end{proposition}
\begin{proof}
    We use the same notation as before in \prettyref{eq:Lfdri}. Let $W_i=\frac{\pi_0f_0(X_i)}{\pi_0f_0(X_i)+\pi_1f_1(X_i)}$ and $W_i'=\Expect[1-\theta_i|X_i]=\frac{\frac{n_0}{n}f_0(X_i)}{\frac{n_0}{n}f_0(X_i)+\frac{n_1}{n}f_1(X_i)}$, where $n_0=n-n_1$. We first show that for any $\tau\in(0,1)$
    \begin{align}
        \sup_{k\geq\tau n}\bigg|\frac{1}{k}\sum_{i=1}^kW_{(i)}-\frac{1}{k}\sum_{i=1}^k W_{(i)}'\bigg|\leq\frac{1}{\tau}\bigg|1-\frac{n_0\pi_1}{n_1\pi_0}\bigg|=o(1),\label{eq:tail_bound3}
    \end{align}
    where $W_{(i)}$, $W_{(i)}'$ are the order statistics of $W_i$ and $W_i'$ respectively. We start with an elementary lemma. (For a proof see \prettyref{app:ProofBestMatching}.)
    \begin{lemma}\label{lmm:best_matching}
        Suppose $a_{(1)}\leq a_{(2)}\cdots\leq a_{(n)}$ and $b_{(1)}\leq b_{(2)}\cdots\leq b_{(n)}$ are the sorted versions of two sequences $a_1,\cdots,a_n$ and $b_1,\cdots,b_n$. Then, $\sup_{0\leq k\leq n}|\sum_{i=1}^ka_{(i)}-\sum_{i=1}^kb_{(i)}|\leq\sum_{i=1}^n|a_i-b_i|$.
    \end{lemma}
    Direct calculation shows that $|W_i-W_i'|=W_i(1-W_i')\big|1-\frac{n_0\pi_1}{n_1\pi_0}\big|\leq \big|1-\frac{n_0\pi_1}{n_1\pi_0}\big|$ for any $i$. Then, combined with \prettyref{lmm:best_matching}, $\sup_{k\geq\tau n}\bigg|\frac{1}{k}\sum_{i=1}^kW_{(i)}-\frac{1}{k}\sum_{i=1}^k W_{(i)}'\bigg|\leq\frac{1}{\tau}\frac{1}{n}\sum_{i=1}^n|W_i-W_i'|\leq\frac{1}{\tau}\big|1-\frac{n_0\pi_1}{n_1\pi_0}\big|=o(1)$.
    
    Next, we show that the partial average $\frac{1}{k}\sum_{i=1}^k\widetilde{W}_{(i)}$ is uniformly close to $\frac{1}{k}\sum_{i=1}^kW_{(i)}'$ in the sense that for any $\epsilon>0$ and $\tau\in(0,1)$,
    \begin{align}\label{eq:tail_bound2}
        \Prob\pth{\sup_{k\geq\tau n}\bigg|\frac{1}{k}\sum_{i=1}^kW_{(i)}'-\frac{1}{k}\sum_{i=1}^k\widetilde{W}_{(i)}\bigg|>\epsilon}\leq\frac{\log (n+1)}{\tau^2\epsilon^2n}.
    \end{align}
		Note that since the number of non-nulls are fixed, $\theta^n$ no longer has iid coordinates and hence 
		$\widetilde{W}_i=\Expect[1-\theta_i|X^n]$ depends on the whole sample $X^n$. Intuitively, given the direct observation $X_i$ 
		the dependency between $\theta_i$ and the other indirect observations $\{X_j: j \neq i\}$ is rather weak and hence 
		we anticipate $\widetilde{W}_i$ to be close to $W_i'=\Expect[1-\theta_i|X_i]$. The following lemma quantifies this intuition and may be of independent interest. 
		The proof is based on an information-theoretic argument and given in \prettyref{app:ProofWW_tilde}.
    \begin{lemma}\label{lmm:W&W_tilde}
    In the fixed-proportion model, we have $\max\limits_{1\leq i\leq n}\Expect[(W_i'-\widetilde{W}_i)^2]\leq\frac{\log (n+1)}{n}$, for any finite $n$.
    \end{lemma}
    By \prettyref{lmm:best_matching}, we have
    \begin{align*}
        \sup_{k\geq\tau n}\bigg|\frac{1}{k}\sum_{i=1}^kW_{(i)}'-\frac{1}{k}\sum_{i=1}^k\widetilde{W}_{(i)}\bigg|\leq\frac{1}{\tau}\frac{1}{n}\sum_{i=1}^n|W_i'-\widetilde{W}_i|\leq\frac{1}{\tau}\sqrt{\frac{1}{n}\sum_{i=1}^n(W_i'-\widetilde{W}_i)^2}.
    \end{align*}
By Markov's inequality and \prettyref{lmm:W&W_tilde}, for any $\epsilon>0$,
\begin{align*}
    \Prob\pth{\sup_{k\geq\tau n}\bigg|\frac{1}{k}\sum_{i=1}^kW_{(i)}'-\frac{1}{k}\sum_{i=1}^k\widetilde{W}_{(i)}\bigg|>\epsilon}\leq\Prob\pth{\frac{1}{n}\sum_{i=1}^n(W_i'-\widetilde{W}_i)^2>\tau^2\epsilon^2}\leq\frac{\log (n+1)}{\tau^2\epsilon^2n}.
\end{align*}

Last, we show that for any fixed $\tau\in(0,1)$ and $\epsilon>0$, the following inequality holds for sufficiently large $n$:
\begin{align}\label{eq:tail_bound1}
    \Prob\pth{\sup_{k\geq\tau n}\Big|\frac{1}{k}\sum_{i=1}^k W_{(i)} - a(k/n)\Big|>\epsilon}\leq12\sqrt{n}\exp(-n\tau^2\epsilon^2/4).
\end{align}
Consider the following process, which is an alternative way to generate the two-group random mixture model. First, sample $N_1\sim\Binom(n,\pi_1)$. Then, conditioning on $N_1$, sample true labels $\theta^n$ uniformly at random from $\{\theta^n\in\{0,1\}^n:\sum_{i=1}^n\theta_i=N_1\}$. Finally, conditioning on $\theta^n$, sample $X^n|\theta^n\sim\prod_{i=1}^nf_0^{1-\theta_i}f_1^{\theta_i}$. Denote the event $\{\sup_{k\geq\tau n}\big|\frac{1}{k}\sum_{i=1}^k W_{(i)} - a(\frac{k}{n})\big|>\epsilon\}$ by $E$. We proved in \prettyref{lmm:approx_of_obj&constr} that $\Prob\pth{E} \leq  6 \exp(-n\tau^2\epsilon^2/2)$ in the random mixture model. Then, $\Prob(E)=\sum_{i=0}^n\Prob(N_1=i)\Prob(E|N_1=i)\geq\Prob(N_1=n_1)\Prob(E|N_1=n_1)=\binom{n}{n_1}\pi_1^{n_1}\pi_0^{n_0}\Prob(E|N_1=n_1)$. By the 
non-asymptotic Stirling's approximation \cite{robbins1955remark},
 we have $\binom{n}{n_1}\pi_1^{n_1}\pi_0^{n_0}\geq\frac{1}{2\sqrt{n}}(\frac{n}{n_1})^{n_1}(\frac{n}{n_0})^{n_0}\cdot\pi_1^{n_1}\pi_0^{n_0}=\frac{1}{2\sqrt{n}}(\frac{n\pi_1}{n_1})^{n_1}(\frac{n\pi_0}{n_0})^{n_0}=\frac{1}{2\sqrt{n}}\pth{(\frac{n\pi_1}{n_1})^{\frac{n_1}{n}}(\frac{n\pi_0}{n_0})^{\frac{n_0}{n}}}^n$. Note that $(\frac{n\pi_1}{n_1})^{\frac{n_1}{n}}(\frac{n\pi_0}{n_0})^{\frac{n_0}{n}}=1+o(1)$ as $n\to\infty$, then $\pth{(\frac{n\pi_1}{n_1})^{\frac{n_1}{n}}(\frac{n\pi_0}{n_0})^{\frac{n_0}{n}}}^n=e^{n\cdot o(1)}\geq e^{-n\tau^2\epsilon^2/4}$ for sufficiently large $n$. Hence, $\binom{n}{n_1}\pi_1^{n_1}\pi_0^{n_0}\geq\frac{1}{2\sqrt{n}}e^{-n\tau^2\epsilon^2/4}$ for sufficiently large $n$. It follows that $\Prob(E|N_1=n_1)\leq12\sqrt{n}\exp(-n\tau^2\epsilon^2/4)$. Namely, \prettyref{eq:tail_bound1} holds for the fixed-proportion model. Combined with \prettyref{eq:tail_bound3} and \prettyref{eq:tail_bound2}, for sufficiently large $n$, we have $$\Prob\pth{\sup_{k\geq\tau n}\Big|\frac{1}{k}\sum_{i=1}^k \widetilde{W}_{(i)} - a(k/n)\Big|>3\epsilon}\leq12\sqrt{n}\exp(-n\tau^2\epsilon^2/4)+\frac{\log (n+1)}{\tau^2\epsilon^2n},$$
and \prettyref{eq:converge1} follows directly. The proof of \prettyref{eq:converge2} is similar.
\end{proof}

With \prettyref{prop:UnionBound_fixed}, the proofs of Theorems \ref{thm:ora_FNR_vs_ora_mFNR_fixed} and \ref{thm:high_prob_formulation_fixed} follow from the same program as those of Theorems \ref{thm:ora_FNR_vs_ora_mFNR} and \ref{thm:high_prob_formulation}. 
\begin{proof}[proof of \prettyref{thm:ora_FNR_vs_ora_mFNR_fixed}]
    First, we show $\limsup_{n\to\infty}\optFNR_n(\alpha)\leq b^{**}(\alpha)$. For any fixed $u\in[0,1]$, consider the separable rule $\delta^{n,u}=(\delta_1,\cdots,\delta_n)$ of the form
    \begin{align*}
    \delta_i= 
    \begin{cases}
        1 & \text{if } W_i < G^{-1}(u)\\
        \Bern(p(u)) & \text{if } W_i = G^{-1}(u)\\
        0 & \text{if } W_i > G^{-1}(u)
    \end{cases}
    \end{align*}
    where $W_i=\pi_0f_0(X_i)/(\pi_0f_0(X_i)+\pi_1f_1(X_i))$ and $p(\cdot)$ is defined in \eqref{eq:p(u)}. Then, $\FDR(\delta^{n,u})=\Expect\left[\frac{\sum_{i=1}^n(1-\theta_i)\delta_i}{1\vee\sum_{i=1}^n\delta_i}\right]=\Expect\left[\Expect\left[\frac{\sum_{i=1}^n(1-\theta_i)\delta_i}{1\vee\sum_{i=1}^n\delta_i}\big|\theta^n\right]\right]$. By exchangeability, $\Expect\left[\frac{\sum_{i=1}^n(1-\theta_i)\delta_i}{1\vee\sum_{i=1}^n\delta_i}\big|\theta^n\right]$ is determined by the number of 1s in $\theta^n$ but does not depend on which coordinates are equal to 1, so we can assume $\theta_i=0$ for $i=1,\cdots,n_0$ and $\theta_i=1$ for $i=n_0+1,\cdots,n$, where $n_0=n-n_1$. Then, by the Law of Large Numbers,
    \begin{align}
        &\frac{1}{n}\sum_{i=1}^n(1-\theta_i)\delta_i=\frac{n_0}{n}\frac{1}{n_0}\sum_{i=1}^{n_0}\delta_i~\toprob~\pi_0\Expect_0[\delta_1]=\Expect\left[\pi_0f_0(X_1)/f(X_1)\cdot\delta_1\right]=\Expect[W_1\delta_1],\label{eq:numerator}\\
        &\frac{1}{n}\sum_{i=1}^n\delta_i=\frac{n_0}{n}\frac{1}{n_0}\sum_{i\leq n_0}\delta_i+\frac{n_1}{n}\frac{1}{n_1}\sum_{i>n_0}\delta_i~\toprob~\pi_0\Expect_0[\delta_1]+\pi_1\Expect_1[\delta_1]=\Expect[\delta_1].\label{eq:denominator}
    \end{align}
    Then, $\FDR(\delta^{n,u})\to\frac{\Expect[W_1\delta_1]}{\Expect[\delta_1]}=\Expect[W_1|\delta_1=1]=a(u)$. Similarly, $\FNR(\delta^{n,u})\to b(u)$. For any $\epsilon>0$, suppose $U^*$ is the optimal solution to \prettyref{eq:opt_prob_a&b_random} with level $\alpha-\epsilon$, \ie $\Expect[a(U^*)]\leq\alpha-\epsilon$ and $\Expect[b(U^*)]=b^{**}(\alpha-\epsilon)$. Consider the randomized rule $\delta^{n,U^*}$. By the bounded convergence theorem, $\FDR(\delta^{n,U^*})\to \Expect[a(U^*)]\leq\alpha-\epsilon$ and $\FNR(\delta^{n,U^*})\to \Expect[b(U^*)]=b^{**}(\alpha-\epsilon)$. Thus, there exists $N$ such that $\FDR(\delta^{n,U^*})\leq\alpha$ and $\FNR(\delta^{n,U^*})\leq b^{**}(\alpha-\epsilon)+\epsilon$ for any $n>N$. It follows that $\optFNR_n(\alpha)\leq b^{**}(\alpha-\epsilon)+\epsilon$ for $n>N$. By first letting $n\to\infty$ and then $\epsilon\to 0$ and noting that $b^{**}$ is continuous on $(0,1)$, we have $\limsup_{n\to\infty}\optFNR_n(\alpha)\leq b^{**}(\alpha)$.

    Next, we show $\liminf_{n\to\infty}\optFNR_n(\alpha)\geq b^{**}(\alpha)$. Note that \prettyref{lmm:opt_prob_FDR&FNR} also apply to the fixed-proportion model if we replace $W_i$ with $\widetilde{W}_i=\Expect[\theta_i|X^n]$. Then, thanks to \prettyref{prop:UnionBound_fixed}, we can entirely adopt the proof of the negative part of \prettyref{prop:asymp_opt} by replacing all $W_i$ with $\widetilde{W}_i$. The proof of \prettyref{thm:ora_FNR_vs_ora_mFNR_fixed} is completed.
\end{proof}

\begin{proof}[proof of \prettyref{thm:high_prob_formulation_fixed}]
    The proof for the negative part is the same as that in \prettyref{sec:ProofThm2} by replacing all $W_i$ with $\widetilde{W}_i$ and applying \prettyref{prop:UnionBound_fixed}.
    For the positive part, fix $\theta^n$. By \prettyref{eq:numerator} and \prettyref{eq:denominator}, we have $\FDP(\NP^n(\alpha))\toprob a(u^*(\alpha))\leq\alpha$ and $\FNP(\NP^n(\alpha))\toprob b(u^*(\alpha))=b^*(\alpha)$, where $u^*(\alpha)$ is the optimal solution of \prettyref{eq:opt_prob_a&b}. For any $\epsilon>0$, by exchangeability, $\Prob\pth{\FDP(\NP^n)\leq\alpha+\epsilon\,|\,\theta^n}$ and $\Prob\pth{\FNP(\NP^n)\leq\beta+\epsilon\,|\,\theta^n}$ remain the same for all $\theta^n$ such that $\sum_{i=1}^n\theta_i=n_1$. Therefore, for any $\beta\geq b^*(\alpha)$,
    \begin{align*}
    \Prob\pth{\FDP(\NP^n)\leq\alpha+\epsilon}=\Expect[\Prob\pth{\FDP(\NP^n)\leq\alpha+\epsilon\,|\,\theta^n}]=1-o(1),\\
    \Prob\pth{\FNP(\NP^n)\leq\beta+\epsilon}=\Expect[\Prob\pth{\FNP(\NP^n)\leq\beta+\epsilon\,|\,\theta^n}]=1-o(1).
    \end{align*}
    The proof is completed.
\end{proof}

\subsection{Weakly dependent data}\label{sec:dependent}

Recently there has been progress in understanding the fundamental limits of multiple testing with dependent data, either by allowing the labels to be dependent or the observations to be dependent conditioned on the labels \cite{sun2009large,xie2011optimal,heller2021optimal}. 
The preceding \prettyref{thm:high_prob_formulation_fixed} can be viewed as an extension of our main results to the case of weakly dependent labels. Next we provide another extension for weakly dependent data. 

Following \cite{xie2011optimal,heller2021optimal,rosset2022optimal}, we consider the two-group normal mean model with correlated observations. Let $X^n \sim N(\mu \cdot \theta^n, \Sigma_n)$, where $\mu>0$, $\theta^n \iiddistr \Bern(\pi_1)$, and 
$\Sigma_n$ is an $n\times n$ covariance matrix with diagonals all equal to $\sigma^2$.
(Note that \prettyref{ex:gaussian} corresponds to $\Sigma_n = \sigma^2 I_n$.) 
The following result, proved in \prettyref{app:dependent}, shows that the optimal FDR-FNR tradeoff in \prettyref{thm:ora_FNR_vs_ora_mFNR} and \prettyref{thm:high_prob_formulation} remains the same under a weak dependency assumption. 
Characterizing the optimal FDR-FNR tradeoff for models with strongly dependent data is
an open problem; see the next section for discussion.
\begin{theorem}
\label{thm:dependent}
Under the preceding two-group normal mean model, \prettyref{thm:ora_FNR_vs_ora_mFNR} and \prettyref{thm:high_prob_formulation} continue to hold provided that 
	$\|\Sigma_n-\sigma^2  I_n\|_{\rm sp} \to 0$ in the spectral norm.	
\end{theorem}

\section{Conclusions and Discussions}\label{sec:Discussion}

In this paper, we resolve the fundamental limit of multiple testing in the two-group random mixture model by characterizing the minimum achievable FNR subject to a FDR constraint in the large-$n$ limit. In contrast to the optimal mFDR control (that minimizes the mFNR), which is separable and consists of $n$ independent and identical Neyman-Pearson tests, 
the optimal FDR control is in general not separable even in the large-$n$ limit and can be achieved by randomizing between two sets of mFDR controls. 
This phenomenon brings into light the peril of controlling only the FDR, which fails to control the variability of the FDP. To remedy this limitation, we also showed that controlling the FDP with high probability results in the separable rules being optimal. Finally, we extended our results to a different objective of maximizing the expected number of true discoveries and also to models with fixed (as opposed to on average) non-null proportion. We proved that the same conclusion also applies to these extensions.

	We have focused on models where the observations are independent conditioned on the labels and the labels are either iid Bernoulli or uniform at random with a fixed non-null proportion, with the exception of 
	\prettyref{thm:ora_FNR_vs_ora_mFNR_fixed} and \ref{thm:dependent} that allow some weak dependence in the labels and the observations, respectively. 
	It is of interest to extend the results to models with stronger dependency. 
	
	For a simple example, suppose that the true labels $\theta_1,\cdots,\theta_n$ form a stationary and $m$-dependent sequence \cite{Hoeffding65}, in the sense that $(\theta_i)_{i\in\calI} \indep (\theta_j)_{j\in\calJ}$ if $\min\{j:j\in\calJ\}-\max\{i:i\in\calI\}>m$ for any subset $\calI,\calJ\subset[n]$. Our main results can be shown to hold in this model with suitable modification of the definitions. To see this, note that the Lfdr sequence $W_i = \Expect[1-\theta_i|X^n] = \Expect[1-\theta_i|X_{i-m},\cdots,X_{i+m}]$ is stationary and $3m$-dependent. Similar to the independent model, one can show that $\frac{1}{\lceil\lambda n\rceil}\sum_{i=1}^{\lceil\lambda n\rceil} W_{(i)}$ concentrates for any $\lambda\in(0,1)$, by similar arguments as in \prettyref{app:ProofMainLmm}. Therefore, 
	$\lim_{n\to\infty}\optFNR_n(\alpha)=b^{**}(\alpha)$ continues to hold provided that $\Lfdr(X_i)$ is replaced by $\Expect[1-\theta_i|X_{i-m},\cdots,X_{i+m}]$ in the definition of $b^{**}(\alpha)$. 
	
	More sophisticated models have been considered in the literature of multiple testing, for example, the hidden Markov model (HMM) in which the true labels $\theta_1,\cdots,\theta_n$ follows a stationary Markov chain 
and the observations $X_i$'s are independent conditioned on $\theta_i$'s	\cite{sun2009large,abraham2022multiple}. It was shown in \cite{sun2009large} that the optimal procedure to control the mFDR is given by thresholding $\Lfdr(X_i)$ with a fixed cutoff; however, the optimal FDR control is unknown.
	Note that a common proof ingredient for both Theorems \ref{thm:high_prob_formulation_fixed} and \ref{thm:dependent} is to approximate the full likelihood ratio $\Expect[\theta_i|X^n]$ by the marginal one $\Expect[\theta_i|X_i]$. For HMM, we do not expect such approximation to hold. It is thus unclear whether optimal FDR-FNR frontier is still given by the greatest convex minorant of the mFDR-mFNR tradeoff.

\newpage
\appendix

\section{Example of Gaussian mixture model}\label{app:GaussianModel}

Under the two-group model \eqref{eq:two-group_model}, suppose $X_i \sim \calN(0,1)$ if $\theta_i=0$, and $X_i \sim \calN(\mu,1)$ if $\theta_i=1$, for $i=1,\cdots,n$, where $\mu>0$.
Denoting by $\Phi$ the standard normal CDF, direct calculation shows that the optimal tradeoff curve $(\alpha,\optmFNR(\alpha))$, $\alpha\in(0,\pi_0)$ has the following parametric form
\begin{align*}
\begin{cases}
    x &=\frac{\pi_0\pth{1-\Phi(z)}}{1-\pth{\pi_0\Phi(z)+\pi_1\Phi(z-\mu)}}\\
    y &=\frac{\pi_1\Phi(z-\mu)}{\pi_0\Phi(z)+\pi_1\Phi(z-\mu)}
\end{cases}
,\quad z\in(-\infty,+\infty)
\end{align*}
and the optimal tradeoff curve $(\alpha,\optmEFNnorm(\alpha))$ is given by
\begin{align*}
\begin{cases}
    x &=\frac{\pi_0\pth{1-\Phi(z)}}{1-\pth{\pi_0\Phi(z)+\pi_1\Phi(z-\mu)}}\\
    y &=\pi_1\Phi(z-\mu)
\end{cases}
,\quad z\in(-\infty,+\infty).
\end{align*}

\begin{suppprop}
Assume the Gaussian mixture model described above.
\begin{enumerate}
    \item[(a)] There exists $\mu_0>0$ such that for any $\mu\in(0,\mu_0)$, it holds that $\optmFNR(\alpha)>\pi_1\pth{1 - \frac{\alpha}{\pi_0}}$ for any $\alpha\in(0,\pi_0)$ and $\pi_0\in(0,1)$. 
    
    \item[(b)] For any $\mu>0$ and $\pi_0\in(0,1)$, there exists $\alpha_0=\alpha_0(\mu,\pi_0)\in(0,\pi_0)$ such that $\optmFNR(\alpha)>\pi_1\pth{1 - \frac{\alpha}{\pi_0}}$ for any $\alpha\in(0,\alpha_0)$.

    \item[(c)] For any $\mu>0$ and $\pi_0\in(0,1)$, there exists $\alpha_1=\alpha_1(\mu,\pi_0)\in(0,\pi_0)$ such that $\optmEFNnorm(\alpha)>\pi_1\pth{1 - \frac{\alpha}{\pi_0}}$ for any $\alpha\in(0,\alpha_1)$.

    \item[(d)] For any $\mu>0$ and $\pi_0\in(0,1)$, there exists $\alpha_2=\alpha_2(\mu,\pi_0)\in(0,\pi_0)$ such that $\optmEFNnorm(\alpha)<\pi_1\pth{1 - \frac{\alpha}{\pi_0}}$ for any $\alpha\in(\alpha_2,1)$.
\end{enumerate}
\end{suppprop}

\begin{proof}
We first prove (a) and (b). Since 
\begin{align*}
    &\qquad\,\,\,\optmFNR(\alpha)>\pi_1\pth{1 - \frac{\alpha}{\pi_0}}\\ &\Longleftrightarrow \frac{\pi_1\Phi(z-\mu)}{\pi_0\Phi(z)+\pi_1\Phi(z-\mu)} > \pi_1 - \frac{\pi_1}{\pi_0}\frac{\pi_0\pth{1-\Phi(z)}}{1-\pth{\pi_0\Phi(z)+\pi_1\Phi(z-\mu)}}\\
    &\Longleftrightarrow \frac{\Phi(z-\mu)}{\pi_0\Phi(z)+\pi_1\Phi(z-\mu)} + \frac{1-\Phi(z)}{1-\pth{\pi_0\Phi(z)+\pi_1\Phi(z-\mu)}} > 1\\
    &\Longleftrightarrow \Phi(z-\mu) > \pi_0\pi_1\pth{\Phi(z)-\Phi(z-\mu)}^2+\Phi(z)\Phi(z-\mu),
\end{align*}
it suffices to consider the case that $\pi_0=\pi_1=\frac{1}{2}$ and show $\Phi(z-\mu) > \frac{1}{4}\pth{\Phi(z)+\Phi(z-\mu)}^2$. Define $h_\mu(z)=\Phi(z-\mu)-\frac{1}{2}\pth{\Phi(z)^2+\Phi(z-\mu)^2}$. 

For (a), since $\phi(2\mu)-\mu\Phi(2\mu)$ is continuous in $\mu$ and it is positive when $\mu=0$ and is negative when $\mu\to+\infty$, there exists $\mu_1>0$ such that $\phi(2\mu_1)-\mu_1\Phi(2\mu_1)=0$. Since $\frac{1}{2}e^{\mu^2}-\Phi(2\mu)$ is continuous in $\mu$ and it is negative when $\mu=1/2$ and positive when $\mu\to+\infty$, there exists $\mu_2>0$ such that $\frac{1}{2}e^{\mu_2^2}-\Phi(2\mu_2)=0$. Let $\mu_0=\min(\mu_1,\mu_2)$ and we show $h_\mu(z)>0$ for any $\mu\in(0,\mu_0)$ and $z\in\reals$. By calculation, $h'_\mu(z) = \phi(z-\mu)\pth{1-\Phi(z-\mu)-\exp\pth{\frac{1}{2}\mu^2-\mu z}\Phi(z)}$, where $\phi$ is the pdf of $\calN(0,1)$. Let $g_\mu(z)=1-\Phi(z-\mu)-\exp\pth{\frac{1}{2}\mu^2-\mu z}\Phi(z)$, so we have $g'_\mu(z)=-\pth{\phi(z-\mu)+e^{\frac{1}{2}\mu^2-\mu z}\pth{\phi(z)-\mu\Phi(z)}}$. If $z\leq 2\mu$, since $\frac{d}{dz}\pth{\phi(z)-\mu\Phi(z)} = -\phi(z)(\mu+z)$, we have $$\phi(z)-\mu\Phi(z)\geq \min\pth{\lim\limits_{z\to-\infty}\pth{\phi(z)-\mu\Phi(z)},\phi(2\mu)-\mu\Phi(2\mu)} = \min\pth{0,\phi(2\mu)-\mu\Phi(2\mu)}.$$ Note that $\phi(2\mu)-\mu\Phi(2\mu)$ is decreasing on $\mu\in(0,\mu_0)$, so $\phi(2\mu)-\mu\Phi(2\mu)\geq \phi(2\mu_0)-\mu_0\cdot\Phi(2\mu_0)\geq 0$. Hence, we have $g'_\mu(z)<0$, for $z\leq 2\mu$, \ie $g_\mu(z)$ is strictly decreasing on $(-\infty,2\mu)$. Note that $g_\mu(\mu/2)=1-\Phi(-\mu/2)-\Phi(\mu/2)=0$, so $g_\mu(z)>0$ if $z<\mu/2$, and $g_\mu(z)<0$ if $z\in(\mu/2,2\mu)$. If $z>2\mu$, we have 
\begin{align*}
    g_\mu(z)&=1-\Phi(z-\mu)-\exp\pth{\frac{1}{2}\mu^2-\mu z}\Phi(z)\\
    &\leq \frac{1}{2}\exp\pth{-\frac{1}{2}(z-\mu)^2}-\exp\pth{\frac{1}{2}\mu^2-\mu z}\Phi(z)\\
    &=\exp\pth{\frac{1}{2}\mu^2-\mu z}\pth{\frac{1}{2}\exp\pth{-\frac{1}{2}(z-2\mu)^2+\mu^2}-\Phi(z)}
\end{align*}
where the first inequality is by $1-\Phi(x)\leq\frac{1}{2}e^{-x^2/2}$, $\forall\,x\geq 0$. Since $\frac{1}{2}\exp\pth{-\frac{1}{2}(z-2\mu)^2+\mu^2}-\Phi(z)$ is decreasing on $z\in(2\mu,+\infty)$, we have $\frac{1}{2}\exp\pth{-\frac{1}{2}(z-2\mu)^2+\mu^2}-\Phi(z)\leq \frac{1}{2}e^{\mu^2}-\Phi(2\mu)$. Since $\frac{d}{d\mu}\pth{\frac{1}{2}e^{\mu^2}-\Phi(2\mu)}=e^{-2\mu^2}\pth{\mu e^{3\mu^2}-\frac{2}{\sqrt{2\pi}}}$, for $0<\mu<\mu_0$, we have $\frac{1}{2}e^{\mu^2}-\Phi(2\mu)<\max\pth{0,\frac{1}{2}e^{\mu_0^2}-\Phi(2\mu_0)}=0$, so $g_\mu(z)<0$ if $z>2\mu$. Noting that $h'_\mu(z)=\phi(z-\mu)g_\mu(z)$, we have $h_\mu(z)$ is strictly increasing on $(-\infty,\mu/2)$ and strictly decreasing on $(\mu/2,+\infty)$. Thus, $h_\mu(z)>\min\pth{\lim\limits_{z\to-\infty}h_\mu(z),\lim\limits_{z\to+\infty}h_\mu(z)}=0$. We complete the proof for part (a).

Next, we prove part (b). Note that
\[
\frac{d}{dz}\pth{\frac{\pi_0\pth{1-\Phi(z)}}{1-\pth{\pi_0\Phi(z)+\pi_1\Phi(z-\mu)}}}=\frac{\pi_0\pi_1\pth{\phi(z-\mu)(1-\Phi(z))-\phi(z)(1-\Phi(z-\mu))}}{\pth{1-\pth{\pi_0\Phi(z)+\pi_1\Phi(z-\mu)}}^2}
\]
and
\begin{align*}
    &\quad\,\phi(z)(1-\Phi(z-\mu))=\phi(z)\int_{z-\mu}^{+\infty}\phi(s)ds = \phi(z)\int_z^{+\infty}\phi(s-\mu)ds\\
    &= \phi(z)\int_z^{+\infty}e^{-\frac{1}{2}\mu^2+\mu s}\phi(s)ds \geq \phi(z)e^{-\frac{1}{2}\mu^2+\mu z}\int_z^{+\infty}\phi(s)ds = \phi(z-\mu)(1-\Phi(z)).
\end{align*}
Thus $\frac{\pi_0\pth{1-\Phi(z)}}{1-\pth{\pi_0\Phi(z)+\pi_1\Phi(z-\mu)}}$ is decreasing in $z$. Hence, it suffices to show that for any $\mu>0$, there exists $z_0(\mu)$ such that $h_\mu(z)>0$ for any $z>z_0(\mu)$. Since $g_\mu(z)\exp\pth{\mu z-\frac{1}{2}\mu^2}=(1-\Phi(z-\mu))\exp\pth{\mu z-\frac{1}{2}\mu^2}-\Phi(z)\to-1$ as $z\to\infty$, there exists $z_0(\mu)$ such that $g_\mu(z)<0$ for any $z>z_0(\mu)$. Hence, $h'_\mu(z)=\phi(z-\mu)g_\mu(z)<0$ when $z>z_0(\mu)$. It follows that $h_\mu(z)>\lim\limits_{z\to\infty}h_\mu(z)=0$ for any $z>z_0(\mu)$. The proof of (b) is completed.

For (c), by direct calculation, 
\begin{align*}
    &\qquad\,\,\optmEFNnorm(\alpha)>\pi_1\pth{1 - \frac{\alpha}{\pi_0}} \\
    &\Longleftrightarrow 2\Phi(z-\mu)-\Phi(z)-\Phi(z-\mu)^2+\pi_0(\Phi(z)-\Phi(z-\mu))(1-\Phi(z-\mu))>0,
\end{align*}
so it suffices to show that $f_\mu(z)\triangleq2\Phi(z-\mu)-\Phi(z)-\Phi(z-\mu)^2>0$ for sufficiently large $z$. Taking the derivative, $f_\mu'(z) = 2\phi(z-\mu)\pth{1-\frac{1}{2}\exp(-\mu z+\mu^2/2)-\Phi(z-\mu)}$. Let $l_\mu(z)=1-\frac{1}{2}\exp(-\mu z+\mu^2/2)-\Phi(z-\mu)$, then $l_\mu'(z)=\frac{\mu}{2}\exp(-\mu z+\mu^2/2)-\phi(z-\mu)>0$ for sufficiently large $z$. Hence, $l_\mu(z)<l_\mu(\infty)=0$ for large $z$, and thus $f_\mu(z)$ is decreasing when $z$ is large. Note that $f_\mu(\infty)=0$, so we have $f_\mu(z)=2\Phi(z-\mu)-\Phi(z)-\Phi(z-\mu)^2>0$ for large $z$, or equivalently, $\optmEFNnorm(\alpha)>\pi_1(1 - \frac{\alpha}{\pi_0})$ for small $\alpha$. 

Finally, we prove (d). Direct calculation shows that  
\begin{align*}
    &\qquad\,\,\optmEFNnorm(\alpha)<\pi_1\pth{1 - \frac{\alpha}{\pi_0}} \\
    &\Longleftrightarrow \Phi(z-\mu)-\pi_0\Phi(z)\Phi(z-\mu)-\pi_1\Phi(z-\mu)^2-\pi_1\Phi(z)+\pi_1\Phi(z-\mu)<0.
\end{align*}
Hence, it suffices to show that $f_{\mu,\pi_1}(z)\triangleq\Phi(z-\mu)-\pi_1\Phi(z)+\pi_1\Phi(z-\mu)<0$ for sufficiently small $z$.
Note that $f_{\mu,\pi_1}'(z)=\phi(z-\mu)\pth{1+\pi_1-\pi_1\exp(-\mu z+\frac{1}{2}\mu^2)}<0$ for sufficiently small $z$. Then, $f_{\mu,\pi_1}(z)<f_{\mu,\pi_1}(-\infty)=0$ for small $z$.
\end{proof}

\section{Optimal mFDR control (proof of \prettyref{thm:optimal_sol_of_mFDR&mFNR} and \prettyref{lmm:optimal_sol_of_mFDR&EFN})}\label{app:OptRulemFDR}

First, consider a single test: $\theta\sim\Bern(\pi_1)$, $X\,|\,\theta \sim (1-\theta)f_0+\theta f_1$. Given a decision rule $\delta\in\{0,1\}$, define $\Pi_{i|j}(\delta)\triangleq \Prob(\delta=i\,|\,\theta=j)$, $i,j\in\{0,1\}$. By the Neyman-Pearson lemma \cite{neyman1933ix}, the smallest $\Pi_{0|1}(\delta)$ subject to $\Pi_{1|0}(\delta)\leq\alpha$ is achieved by
\begin{align*}
    \delta=
    \begin{cases}
        1 & \text{if } W < c\\
        \Bern(p) & \text{if } W = c\\
        0 & \text{if } W > c
    \end{cases}
\end{align*}
for some $c,p\in[0,1]$ such that $\Pi_{1|0}(\delta)=\alpha$. Let $u=\Prob(\delta=1) = \Prob(W<c) + p\cdot\Prob(W=c)$, and consider the decision rule $\tilde{\delta}\equiv S_u$ defined in \eqref{eq:S_u}. We claim that $\Pi_{0|1}(\delta)=\Pi_{0|1}(\tilde{\delta})$ and $\Pi_{1|0}(\delta)=\Pi_{1|0}(\tilde{\delta})$. As a result, we can assume $c=G^{-1}(u)$ and $p=p(u)$ without loss of generality. To show this, we first notice that $G(c)=\Prob(W\leq c)\geq u$, so $c\geq G^{-1}(u)$. If $c = G^{-1}(u)$, then $\delta=\tilde{\delta}$ almost surely. If $c> G^{-1}(u)$, then $\delta\geq\tilde{\delta}$ almost surely. It follows that $\Pi_{0|1}(\delta)=\Prob(\delta=0\,|\,\theta=1)\leq\Prob(\tilde{\delta}=0\,|\,\theta=1)=\Pi_{0|1}(\tilde{\delta})$ and $\Pi_{1|0}(\delta)=\Prob(\delta=1\,|\,\theta=0)\geq\Prob(\tilde{\delta}=1\,|\,\theta=0)=\Pi_{1|0}(\tilde{\delta})$. On the other hand, note that 
\begin{align*}
    \pi_0\Pi_{1|0}(\delta) - \pi_1\Pi_{0|1}(\delta) &= \Prob(\delta=1,\theta=0) - \Prob(\delta=0,\theta=1)= \Prob(\delta=1) - \Prob(\theta=1)\\
    &= \Prob(\tilde{\delta}=1) - \Prob(\theta=1)= \pi_0\Pi_{1|0}(\tilde{\delta}) - \pi_1\Pi_{0|1}(\tilde{\delta}).
\end{align*}
We have
\[0=\pi_0(\Pi_{1|0}(\delta)-\Pi_{1|0}(\tilde{\delta})) - \pi_1(\Pi_{0|1}(\delta)-\Pi_{0|1}(\tilde{\delta}))\geq 0,\]
which implies that $\Pi_{0|1}(\delta)=\Pi_{0|1}(\tilde{\delta})$ and $\Pi_{1|0}(\delta)=\Pi_{1|0}(\tilde{\delta})$. Therefore, it suffices to consider the Neyman-Pearson tests of the form \eqref{eq:S_u}.

Given a pair of null and alternative pdfs $(f_0,f_1)$, define the region \[\calR(f_0,f_1)\triangleq \{(\Pi_{0|0}(\delta),\Pi_{0|1}(\delta)):\,\delta\in\{0,1\}\text{ is a decision rule}\}\subset [0,1]^2.\]
\begin{supplemma}
$\calR(f_0,f_1)$ is convex and its lower boundary, denoted by $\beta^*(s), s\in[0,1]$, is achieved by Neyman-Pearson tests.
\end{supplemma}
\begin{proof}
For any two points $(\Pi_{0|0}(\delta_1), \Pi_{0|1}(\delta_1))$, $(\Pi_{0|0}(\delta_2), \Pi_{0|1}(\delta_2))$ in $\calR(f_0,f_1)$, and for any $\lambda\in[0,1]$, consider $\delta$ such that $\delta=\delta_1$ with probability $\lambda$ and $\delta=\delta_2$ with probability $1-\lambda$. Then $\Pi_{0|0}(\delta) = \lambda\Pi_{0|0}(\delta_1) + (1-\lambda)\Pi_{0|0}(\delta_2)$ and $\Pi_{0|1}(\delta) = \lambda\Pi_{0|1}(\delta_1) + (1-\lambda)\Pi_{0|1}(\delta_2)$. Since $(\Pi_{0|0}(\delta), \Pi_{0|1}(\delta))\in\calR(f_0,f_1)$, we have $\lambda(\Pi_{0|0}(\delta_1), \Pi_{0|1}(\delta_1))+(1-\lambda)(\Pi_{0|0}(\delta_2), \Pi_{0|1}(\delta_2))\in\calR(f_0,f_1)$. Therefore, $\calR(f_0,f_1)$ is convex. By Neyman-Pearson lemma, for a fixed $\Pi_{0|0}(\delta)=s$, the minimum $\Pi_{0|1}(\delta)$ is obtained by Neyman-Pearson tests.
\end{proof}
Now we get back to multiple testing problems and prove \prettyref{thm:optimal_sol_of_mFDR&mFNR}. For any decision rule $\delta^n=(\delta_1,\cdots,\delta_n)\in\{0,1\}^n$, let $\Pi^i_{0|0} = \Prob(\delta_i=0\,|\,\theta_i=0)$ and $\Pi^i_{0|1} = \Prob(\delta_i=0\,|\,\theta_i=1)$, $i=1,\cdots,n$. Thinking of $\delta_i$ as a single test on the observation $X_i$, we have $(\Pi^i_{0|0},\Pi^i_{0|1})\in\calR(f_0,f_1)$ for all $i$. Let $\bar{\Pi}_{0|0}=\frac{1}{n}\sum_{i=1}^n\Pi^i_{0|0}$ and $\bar{\Pi}_{0|1}=\frac{1}{n}\sum_{i=1}^n\Pi^i_{0|1}$. Then
\begin{align*}
    &\mFDR(\delta^n) =\frac{\Expect\left[\sum_{i=1}^n \delta_i(1-\theta_i)\right]}{\Expect\left[\sum_{i=1}^n \delta_i\right]} = \frac{\pi_0(1-\bar{\Pi}_{0|0})}{\pi_0(1-\bar{\Pi}_{0|0})+\pi_1(1-\bar{\Pi}_{0|1})}\quad\text{and}\\
    &\mFNR(\delta^n) =\frac{\Expect\left[\sum_{i=1}^n (1-\delta_i)\theta_i\right]}{\Expect\left[\sum_{i=1}^n (1-\delta_i)\right]} = \frac{\pi_1\bar{\Pi}_{0|1}}{\pi_0\bar{\Pi}_{0|0}+\pi_1\bar{\Pi}_{0|1}},\,\,\EFN(\delta^n)=\pi_1\bar{\Pi}_{0|1}.
\end{align*}
Since $\calR(f_0,f_1)$ is convex, its lower boundary $\beta^*$ is a convex function. Then 
\begin{align}\label{ineq:beta*}
    \beta^*(\bar{\Pi}_{0|0})\leq \frac{1}{n}\sum_{i=1}^n \beta^*(\Pi^i_{0|0})\leq \frac{1}{n}\sum_{i=1}^n \Pi^i_{0|1} = \bar{\Pi}_{0|1}
\end{align}
where the first inequality is by Jensen inequality and the second is by noting that $(\Pi^i_{0|0},\Pi^i_{0|1})\in\calR(f_0,f_1)$. Consider the decision rule $\tilde{\delta}^n\triangleq (\tilde{\delta}_1,\cdots,\tilde{\delta}_n)$ comprising $n$ independent and identical Neyman-Pearson tests such that $\Prob(\tilde{\delta}_i=0\,|\,\theta_i=0)= \bar{\Pi}_{0|0}$ and $\Prob(\tilde{\delta}_i=0\,|\,\theta_i=1)= \beta^*(\bar{\Pi}_{0|0})$. Then we have
\[
\mFDR(\tilde{\delta}^n) = \frac{\pi_0(1-\bar{\Pi}_{0|0})}{\pi_0(1-\bar{\Pi}_{0|0})+\pi_1(1-\beta^*(\bar{\Pi}_{0|0}))}\overset{\eqref{ineq:beta*}}{\leq} \frac{\pi_0(1-\bar{\Pi}_{0|0})}{\pi_0(1-\bar{\Pi}_{0|0})+\pi_1(1-\bar{\Pi}_{0|1})} = \mFDR(\delta^n),
\]
\[
\mFNR(\tilde{\delta}^n) = \frac{\pi_1\beta^*(\bar{\Pi}_{0|0})}{\pi_0\bar{\Pi}_{0|0}+\pi_1\beta^*(\bar{\Pi}_{0|0})}\overset{\eqref{ineq:beta*}}{\leq}\frac{\pi_1\bar{\Pi}_{0|1}}{\pi_0\bar{\Pi}_{0|0}+\pi_1\bar{\Pi}_{0|1}}= \mFNR(\delta^n),
\]
and
\[
\EFN(\tilde{\delta}^n)=\pi_1\beta^*(\bar{\Pi}_{0|0})\leq\pi_1\bar{\Pi}_{0|1}\leq\EFN(\delta^n).
\]
Thus, $\delta^n$ is dominated by $\tilde{\delta}^n$, so it suffices to consider the decision rules $\delta^n(u)$ consisting of $n$ \iid copies of $S_u$. Note that 
\[
\mFDR(\delta^n(u)) = \frac{\Expect[\delta_1(1-\theta_1)]}{\Expect[\delta_1]}=\frac{\pi_0\Expect_0[\delta_1]}{u}= \frac{1}{u}\Expect\left[\frac{\pi_0f_0(X_1)}{f(X_1)}\delta_1\right] = \frac{\Expect[W_1\indc{\delta_1=1}]}{u} = a(u),
\]
\begin{align*}
    \mFNR(\delta^n(u)) &= \frac{\Expect[(1-\delta_1)\theta_1]}{\Expect[1-\delta_1]}=\frac{\pi_1\Expect_1[1-\delta_1]}{1-u} =\frac{1}{1-u}\Expect\left[\frac{\pi_1f_1(X_1)}{f(X_1)}(1-\delta_1)\right]\\
    &= \frac{1}{1-u}\Expect[(1-W_1)\indc{\delta_1=0}] = b(u)
\end{align*}
and
\[
\EFN(\delta^n(u)) = \Expect[(1-\delta_1)\theta_1]=(1-u)\mFNR(\delta^n(u))=(1-u)b(u).
\]
Hence, the optimization problems \eqref{eq:opt_prob_mFDR&mFNR} and \eqref{eq:opt_prob_a&b} are equivalent, and \eqref{eq:opt_prob_mFDR&EFN} and \eqref{eq:opt_prob_a&b'} are equivalent. Since \eqref{eq:opt_prob_a&b} and \eqref{eq:opt_prob_a&b'} have an optimal solution $u^*(\alpha)$ given in \prettyref{lmm:properties_of_b}, the optimal solution of \eqref{eq:opt_prob_mFDR&mFNR} and \eqref{eq:opt_prob_mFDR&EFN} also exists and has the form \eqref{eq:optimal_sol_of_mFDR&mFNR}. This completes the proof of \prettyref{thm:optimal_sol_of_mFDR&mFNR} and \prettyref{lmm:optimal_sol_of_mFDR&EFN}.

\section{Proofs of auxiliary results}
\subsection{Proof of \prettyref{prop:asym-BH}}\label{app:asym-BH}

Recall the definition of $a(u)$ and $b(u)$ in 
\prettyref{eq:a(u)} and \prettyref{eq:b(u)}. By assumption, $W=\frac{\pi_0}{f(X)}=\frac{\pi_0}{\pi_0+\pi_1 f_1(X)}$ is a continuous random variable and is strictly increasing in $X$.
Furthermore, $a(u)=\Expect[W|W<G^{-1}(u)]$ and $b(u)=\Expect[1-W|W>G^{-1}(u)]$, where $G$ is the CDF of $W$.
 Thus, \eqref{eq:opt_prob_a&b} is equivalent to 
\begin{align*}
     b^*(\alpha) = & \inf_{t\in[0,1]} \,\,\,\Expect[1-W|X>t]\\
    & \,\,\,\,\,\text{s.t.}\quad \Expect[W|X\leq t]\leq \alpha
\end{align*}
where we agree upon that $\Expect[W|X\leq 0]=0$ and $\Expect[1-W|X>1]=0$. 

Denote by $F$ the CDF associated with the PDF $f$, which by assumption is concave on $[0,1]$.
Thus $\Expect[W|X\leq t]=\frac{\pi_0\Expect_0[\indc{X\leq t}]}{\Prob(X\leq t)}=\frac{\pi_0 t}{F(t)}$ is increasing in $t$ and $\Expect[1-W|X>t]=\frac{\pi_1\Expect_1[\indc{X> t}]}{\Prob(X> t)}=\frac{\pi_1 \Prob_1(X> t)}{1-F(t)} = 1-\frac{\pi_0(1-t)}{1-F(t)}$ is decreasing in $t$, so $b^*(\alpha)$ is achieved by $t^*(\alpha)\triangleq \sup\{t\in[0,1]:\,\pi_0 t/F(t)\leq\alpha\}$. In fact, since $\pi_0t/F(t)$ is continuous on $(0,1]$ and $\lim_{t\to 0^+}\pi_0t/F(t)=\pi_0/f(0+) < \alpha < \pi_0=\pi_0/F(1)$, we have $\pi_0t^*(\alpha)/F(t^*(\alpha))=\alpha$. 
Furthermore, using similar arguments as in \cite[Theorem 1]{genovese2002operating}, we can show $\frac{i^*\alpha}{n\pi_0}\toprob t^*(\alpha)$, as $n\to\infty$. Let $\tilde{F}(t,s)=\Prob(X_1\leq t,\theta_1\leq s)$ be the CDF of $(X_1,\theta_1)$ and $\tilde{F}_{n}(t,s)=\frac{1}{n}\sum_{i=1}^n\indc{X_i\leq t, \theta_i\leq s}$ be the empirical CDF of $(X_1,\theta_1),\cdots,(X_n,\theta_n)$. Then 
\[
\Bigg|\frac{1}{n}\sum_{i=1}^n\indc{X_i\leq \frac{i^*\alpha}{n\pi_0},\theta_i=0}-\tilde{F}\pth{\frac{i^*\alpha}{n\pi_0},0}\Bigg| = \Bigg|\tilde{F}_n\pth{\frac{i^*\alpha}{n\pi_0},0}-\tilde{F}\pth{\frac{i^*\alpha}{n\pi_0},0}\Bigg| \leq \|\tilde{F}_n-\tilde{F}\|_\infty\toprob 0
\] by the multivariate Dvoretzky-Kiefer-Wolfowitz (DKW) inequality \cite[Theorem 3.2]{naaman2021tight}.
Note that $\tilde{F}\pth{\frac{i^*\alpha}{n\pi_0},0} = \pi_0\frac{i^*\alpha}{n\pi_0}\toprob\pi_0t^*(\alpha)$, we have 
\[
\frac{1}{n}\sum_{i=1}^n\indc{X_i\leq \frac{i^*\alpha}{n\pi_0},\theta_i=0}\toprob\pi_0t^*(\alpha),\text{ as }n\to\infty.
\]
By the definition of $i^*$, the oracle BH is equivalent to $\delta^n_{\rm BH}(\alpha) = \pth{\indc{X_i \leq \frac{i^*\alpha}{n\pi_0}}}_{i=1,\cdots,n}$. Therefore, 
\begin{align*}
    \FDP(\delta_{\rm BH}^n(\alpha)) = \frac{\frac{1}{n}\sum_{i=1}^n \indc{X_i \leq \frac{i^*\alpha}{n\pi_0},\theta_i=0}}{i^*/n} \toprob \frac{\pi_0t^*(\alpha)}{\pi_0t^*(\alpha)/\alpha}=\alpha,
\end{align*}
and
\begin{align*}
    \mFDR(\delta_{\rm BH}^n(\alpha)) = \frac{\Expect\left[\frac{1}{n}\sum_{i=1}^n \indc{X_i \leq \frac{i^*\alpha}{n\pi_0},\theta_i=0}\right]}{\Expect\left[i^*/n\right]} \to \frac{\pi_0t^*(\alpha)}{\pi_0t^*(\alpha)/\alpha}=\alpha.
\end{align*}

Similarly,
\begin{align*}
    \FNP(\delta_{\rm BH}^n(\alpha)) = \frac{\frac{1}{n}\sum_{i=1}^n \indc{X_i > \frac{i^*\alpha}{n\pi_0},\theta_i=1}}{1-i^*/n} \toprob \frac{\pi_1\pth{1-F_1\pth{t^*(\alpha)}}}{1-\pi_0t^*(\alpha)/\alpha}=\frac{\pi_1\pth{1-F_1\pth{t^*(\alpha)}}}{1-F(t^*(\alpha))}=b^*(\alpha),
\end{align*}
and
\begin{align*}
    \mFNR(\delta_{\rm BH}^n(\alpha)) = \frac{\Expect\left[\frac{1}{n}\sum_{i=1}^n \indc{X_i > \frac{i^*\alpha}{n\pi_0},\theta_i=1}\right]}{\Expect\left[1-i^*/n\right]}
    &\to \frac{\pi_1\pth{1-F_1\pth{t^*(\alpha)}}}{1-\pi_0t^*(\alpha)/\alpha}=b^*(\alpha).
\end{align*}

By \prettyref{thm:optimal_sol_of_mFDR&mFNR}, we have $\optmFNR(\alpha)=b^*(\alpha)$, which completes the proof.

\subsection{Proof of \prettyref{lmm:properties_of_a&b}}\label{app:Property-a&b}
For any $u\in(0,1]$, we have
\begin{align*}
    a(u)&=\Expect\left[W\,|\,S_u = 1\right] = \frac{1}{u}\Expect[W\indc{S_u = 1}]\\
        &= \frac{1}{u}\pth{\Expect\left[W\indc{W < G^{-1}(u)}\right] + p(u)G^{-1}(u)\Prob(W = G^{-1}(u))}\\
        &= \frac{1}{u}\pth{\Expect\left[W\indc{W < G^{-1}(u)}\right] + G^{-1}(u)\pth{u-\Prob\pth{W < G^{-1}(u)}}}\\
        &= G^{-1}(u) - \frac{1}{u}\pth{\Expect\left[(G^{-1}(u) -W)\indc{W < G^{-1}(u)}\right]}.
\end{align*}
If $u' \geq u > 0$, then $G^{-1}(u')\geq G^{-1}(u)$, and 
\begin{align*}
    a(u')-a(u) &= G^{-1}(u') - G^{-1}(u)- \frac{1}{u'}\pth{\Expect\left[(G^{-1}(u') -W)\indc{W < G^{-1}(u')}\right]}\\
    &+ \frac{1}{u}\pth{\Expect\left[(G^{-1}(u) -W)\indc{W < G^{-1}(u)}\right]}\\
    &\geq G^{-1}(u') - G^{-1}(u)- \frac{1}{u'}\pth{\Expect\left[(G^{-1}(u') -W)\indc{W < G^{-1}(u')}\right]}\\
    &+ \frac{1}{u'}\pth{\Expect\left[(G^{-1}(u) -W)\indc{W < G^{-1}(u)}\right]}\\
    &= G^{-1}(u') - G^{-1}(u)- \frac{1}{u'}\pth{\Expect\left[(G^{-1}(u') -W)\indc{G^{-1}(u) \leq W < G^{-1}(u')}\right]}\\
    &- \frac{1}{u'}\pth{\Expect\left[(G^{-1}(u') - G^{-1}(u))\indc{W < G^{-1}(u)}\right]}\\
    &\geq G^{-1}(u') - G^{-1}(u)- \frac{1}{u'}\pth{\Expect\left[(G^{-1}(u') -G^{-1}(u))\indc{G^{-1}(u) \leq W < G^{-1}(u')}\right]}\\
    &- \frac{1}{u'}\pth{\Expect\left[(G^{-1}(u') - G^{-1}(u))\indc{W < G^{-1}(u)}\right]}\\
    &= G^{-1}(u') - G^{-1}(u) - \frac{1}{u'}(G^{-1}(u') - G^{-1}(u))\Prob\pth{W < G^{-1}(u')}\\
    &= \frac{1}{u'}(G^{-1}(u') - G^{-1}(u))\pth{u'-\bar{G}(G^{-1}(u')}) \geq 0
\end{align*}
where the last inequality is by noting that $\bar{G}(G^{-1}(u')\leq u'$. Combined with the fact that $a(u)\geq 0=a(0)$, we conclude $a$ is increasing on $[0,1]$. To show the continuity, it suffices to show $\Expect\left[W\indc{S_u = 1}\right]$ is continuous in $u\in(0,1]$ by noting that $a(u) = \Expect\left[W\indc{S_u = 1}\right]/u$. In fact, for any $0< u\leq u'\leq 1$, we have $G^{-1}(u) \leq G^{-1}(u')$. If $G^{-1}(u) = G^{-1}(u') \triangleq  t_0 \in (0,1]$, then \[\big|\Expect\left[W\indc{S_{u'} = 1}\right]-\Expect\left[W\indc{S_u = 1}\right]\big| = (u'-u)t_0 \leq u'-u.\] If $G^{-1}(u) < G^{-1}(u')$, then
\begin{align*}
&\quad\,\indc{S_{u'} = 1}-\indc{S_u = 1}\\
&= \pth{\indc{S_{u'} = 1}-\indc{S_u = 1}}\indc{W < G^{-1}(u')}+ \pth{\indc{S_{u'} = 1}-\indc{S_u = 1}}\indc{W \geq G^{-1}(u')}\\
&= \pth{1-\indc{S_u = 1}}\indc{W < G^{-1}(u')}+ \pth{\indc{S_{u'} = 1}-0}\indc{W \geq G^{-1}(u')}\geq 0.
\end{align*}
It follows that 
\begin{align*}
    \big|\Expect\left[W\indc{S_u = 1}\right]-\Expect\left[W\indc{S_{u'} = 1}\right]\big|&= \Expect[W(\indc{S_{u'} = 1}-\indc{S_u = 1})]\\
    &\leq \Expect[\indc{S_{u'} = 1}-\indc{S_u = 1}]=u'-u
\end{align*}
Hence, $\Expect\left[W\indc{S_u = 1}\right]$ is continuous on $u\in(0,1]$. We complete the proof of the first part. Similarly, we can prove $b$ is decreasing on $[0,1]$ and continuous on $[0,1)$.

\subsection{Proof of \prettyref{lmm:properties_of_b}}\label{app:ProofLmm-b^*}
By \prettyref{lmm:properties_of_a&b}, $b$ is decreasing and $a$ is increasing on $[0,1]$, and $a$ is continuous on $(0,1]$, so the infimum $b^*(\alpha)$ is achieved at $u^*(\alpha) = \sup\{u\in[0,1]: a(u)\leq\alpha\}$. Then $b^*(0)=b(0)=\Expect[1-W\,|\,S_0=0]=\Expect[1-W]=\pi_1$. Since $a(1)=\Expect[W\,|\,S_1=1]=\Expect[W]=\pi_0$, we have $b^*(\alpha)=b(1)=0$ for $\alpha\in[\pi_0,1]$.  By the definition of $b^*(\alpha)$, it is obviously decreasing in $\alpha$. 
Next, we prove the right-continuity. Let $a_0\triangleq \lim_{u\to 0^+}a(u)$. If $\alpha\in[0,a_0)$, then $u^*(\alpha)=0$ and thus $b^*(\alpha) = b(u^*(\alpha))=b(0)=\Expect[1-W]=\pi_1$. If $\alpha\in[\pi_0,1)$, then $u^*(\alpha)=1$ by noting that $a(1)=\Expect[W]=\pi_0$, so $b^*(\alpha) =b(u^*(\alpha))=b(1)=0$. Thus, $b^*$ is right-continuous on $[0,a_0)$ and $[\pi_0,1)$. For $\alpha\in[a_0,\pi_0)$, if $u^*(\alpha)=1$, then for any $\alpha'>\alpha$, we have $u^*(\alpha')=1$ and thus $b^*(\alpha') =0=b^*(\alpha)$. If $u^*(\alpha)<1$, then by definition $a(u)>\alpha$ for $u>u^*(\alpha)$. For any positive sequence $\{\delta_n\}_{n=1}^\infty$ such that $\delta_n\to 0$ as $n\to\infty$, without loss of generality assume $\alpha+\delta_n\in(a_0,\pi_0)$ for all $n$. Then by the continuity of $a$ we have $a(u^*(\alpha+\delta_n))=\alpha+\delta_n$. If $u^*(\alpha+\delta_n)\nrightarrow u^*(\alpha)$, then there exists $\epsilon_0>0$ and subsequence $\{\delta_{n_k}\}_{k=1}^\infty$ such that $u^*(\alpha+\delta_{n_k})\geq u^*(\alpha)+\epsilon_0$ for any $k$. Then
\[
\alpha < a(u^*(\alpha)+\epsilon_0) \leq \liminf_{k\to\infty} a(u^*(\alpha+\delta_{n_k})) = \liminf_{k\to\infty}(\alpha+\delta_{n_k}) = \alpha
\]
which is a contradiction. Thus, $u^*(\alpha+\delta_n)\to u^*(\alpha)$ as $n\to\infty$. Noting that $b$ is continuous in $[0,1)$ and $u^*(\alpha)<1$, we get $b^*(\alpha+\delta_n) = b(u^*(\alpha+\delta_n))\to b(u^*(\alpha)) = b^*(\alpha)$, so $b^*$ is also right-continuous on $[a_0,\pi_0)$.

\subsection{Proof of \prettyref{lmm:approx_of_obj&constr}}\label{app:ProofMainLmm}
We first describe some properties of $G$ and $\bar{G}$, which are defined in \prettyref{sec:AuxiDefs}.
\begin{supplemma}
For any $u\in[0,1]$, we have $\Prob\pth{G(W)<u}\leq u$ and $\Prob\pth{\bar{G}(W)\leq u}\geq u$.
\end{supplemma}
\begin{proof}
\[
\Prob\pth{G(W)<u} = \Prob\pth{W< G^{-1}(u)} = \bar{G}\pth{G^{-1}(u)}\leq u,
\]
where the first equality is because $G(W)\geq u \iff W\geq G^{-1}(u)$. Let $\bar{G}^{-1}(u) \triangleq  \sup\{t:\,\bar{G}(t) \leq u\}$, $u\in[0,1]$. Then
\[
\Prob\pth{\bar{G}(W)\leq u} = \Prob\pth{W\leq \bar{G}^{-1}(u)} = G\pth{\bar{G}^{-1}(u)}\geq u,
\]
where the first equality is because $\bar{G}(W)\leq u \iff W\leq \bar{G}^{-1}(u)$. 
\end{proof}

Now we prove \eqref{ineq:ApproxConstr} in the following steps.
\begin{enumerate}
    \item Show 
    \begin{align}\label{eq:step1}
        \Prob\pth{\sup\limits_{0\leq t\leq 1}\Big|\frac{1}{n}\sum_{i=1}^n (W_i-t)\indc{W_i< t} - \Expect[(W-t)\indc{W< t}]\Big|>\epsilon} \leq 2 \exp(-2n\epsilon^2).
    \end{align}
    First, we have
    \begin{align*}
        &\quad\,\Expect[(t-W)\indc{W< t}] =\Expect[(t-W)\indc{W\leq t}]=\Expect[\int_W^t ds\cdot\indc{W\leq t}]\\
        &= \Expect[\int_0^1 \indc{W\leq s\leq t, W\leq t}\,ds]= \Expect[\int_0^t \indc{W\leq s}\,ds]= \int_0^t \Prob\pth{W\leq s}ds
    \end{align*}
    Similarly, $\frac{1}{n}\sum_{i=1}^n (t-W_i) \indc{W_i\leq t} = \Expect_{G_n}[(t-W)\indc{W\leq t}] = \int_0^t \Prob_n\pth{W\leq s}ds$, where $G_n$ is the empirical CDF of $W_1,\cdots,W_n$. Then for any $t\in [0,1]$, we have
    \[
    \Big|\frac{1}{n}\sum_{i=1}^n (W_i-t)\indc{W_i< t} - \Expect[(W-t)\indc{W< t}]\Big| \leq \int_0^t \Big|G_n(s)-G(s)\Big|ds \leq \|G-G_n\|_\infty.
    \]
    By the DKW inequality \cite{massart1990tight}, $\Prob\pth{\|G-G_n\|_\infty > \epsilon} \leq 2\exp(-2n\epsilon^2)$. We obtain \prettyref{eq:step1}.
    
    \item Show 
    \begin{align}
        &\quad\Prob\pth{\sup\limits_{k\geq\tau n}\Big|\frac{1}{n}\sum_{i=1}^n (W_i-W_{(k)})\indc{W_i< W_{(k)}} - \Expect[(W-W_{(k)})\indc{W< W_{(k)}}]\Big|>\epsilon}\notag\\
        &\leq 2 \exp(-2n\epsilon^2).\label{eq:step2}
    \end{align}    
    In fact,
    \begin{align*}
        &\quad\,\Prob\pth{\sup\limits_{k\geq\tau n}\Big|\frac{1}{n}\sum_{i=1}^n (W_i-W_{(k)})\indc{W_i< W_{(k)}} - \Expect[(W-W_{(k)})\indc{W< W_{(k)}}]\Big|>\epsilon}\notag\\
        &\leq \Prob\pth{\sup\limits_{0\leq t\leq 1}\Big|\frac{1}{n}\sum_{i=1}^n (W_i-t)\indc{W_i< t} - \Expect[(W-t)\indc{W< t}]\Big|>\epsilon}\overset{\eqref{eq:step1}}{\leq} 2 \exp(-2n\epsilon^2).
    \end{align*}
    
    \item Show 
    \begin{align}
        &\quad\Prob\pth{\sup\limits_{k\geq\tau n}\Big|\frac{1}{k}\sum_{i=1}^k W_{(i)} - \frac{n}{k}\pth{\Expect[(W-W_{(k)})\indc{W< W_{(k)}}]+\frac{k}{n}W_{(k)}}\Big|>\epsilon}\notag\\
        &\leq 2 \exp(-2n\tau^2\epsilon^2).\label{eq:step3}
    \end{align}    
    Since
    \begin{align*}
        \frac{1}{k}\sum_{i=1}^k W_{(i)} &= \frac{1}{k}\pth{\sum_{i=1}^n W_i\cdot\indc{W_i < W_{(k)}} + W_{(k)}\pth{k-\sum_{i=1}^n \indc{W_i < W_{(k)}}}}\\
        &= \frac{1}{k}\pth{\sum_{i=1}^n (W_i-W_{(k)})\indc{W_i < W_{(k)}}}+ W_{(k)}\\
        &= \frac{n}{k}\pth{\frac{1}{n}\sum_{i=1}^n (W_i-W_{(k)})\indc{W_i < W_{(k)}}}+ W_{(k)}
    \end{align*}
    and
    \begin{align*}
        \frac{n}{k}\pth{\Expect[(W-W_{(k)})\indc{W< W_{(k)}}]+\frac{k}{n}W_{(k)}}= \frac{n}{k}\cdot\Expect[(W-W_{(k)})\indc{W< W_{(k)}}]+W_{(k)}
    \end{align*}
    we have
    \begin{align*}
        &\quad\,\Prob\pth{\sup\limits_{k\geq\tau n}\Big|\frac{1}{k}\sum_{i=1}^k W_{(i)} - \frac{n}{k}\pth{\Expect[(W-W_{(k)})\indc{W< W_{(k)}}]+\frac{k}{n}W_{(k)}}\Big|>\epsilon}\notag\\
        &\leq \Prob\pth{\sup\limits_{k\geq\tau n}\Big|\frac{1}{n}\sum_{i=1}^n (W_i-W_{(k)})\indc{W_i< W_{(k)}} - \Expect[(W-W_{(k)})\indc{W< W_{(k)}}]\Big|>\tau\epsilon}\notag\\
        &\leq \Prob\pth{\sup\limits_{k\geq\tau n}\Big|\frac{1}{n}\sum_{i=1}^n (W_i-W_{(k)})\indc{W_i< W_{(k)}} - \Expect[(W-W_{(k)})\indc{W< W_{(k)}}]\Big|>\tau\epsilon}\notag\\
        &\overset{\eqref{eq:step2}}{\leq} 2 \exp(-2n\tau^2\epsilon^2).
    \end{align*}
    
    \item Define $\phi_k(t)\triangleq  \Expect[(W-t)\indc{W< t}] + \frac{k}{n}t$, for $t\in[0,1]$. Show
    \begin{align}
        \Prob\pth{\sup\limits_{k\geq\tau n}\frac{n}{k}\,\Big|\phi_k(W_{(k)}) - \phi_k(G^{-1}(k/n)) \Big|>\epsilon} \leq 4 \exp(-2n\tau^2\epsilon^2).\label{eq:step4}
    \end{align}    
    For any $t_1>t_2$, by direct calculation, we have
    \begin{align*}
        \phi_k(t_2) - \phi_k(t_1) = (t_1-t_2)\pth{\bar{G}(t_1)-\frac{k}{n}} - \Expect[(W-t_2)\indc{t_2 < W < t_1}].
    \end{align*}
    Let $t_1 = W_{(k)}$ and $t_2 = G^{-1}(k/n)$. If $t_1 > t_2$, then $\phi_k(t_2) - \phi_k(t_1) \geq (t_1-t_2)\pth{\bar{G}(t_1)-\frac{k}{n}} - \Expect[(t_1-t_2)\indc{t_2 < W < t_1}]= (t_1-t_2)\pth{G(t_2)-\frac{k}{n}} \geq 0$, and $\phi_k(t_2) - \phi_k(t_1) \leq(t_1-t_2)\pth{\bar{G}(t_1)-\frac{k}{n}} \leq \bar{G}(W_{(k)}) - \frac{k}{n}$. If $t_2 > t_1$, then $\phi_k(t_2) - \phi_k(t_1) = (t_2-t_1)\pth{\frac{k}{n}-\bar{G}(t_2)} + \Expect[(W-t_1)\indc{t_1 < W < t_2}] \geq(t_2-t_1)\pth{\frac{k}{n}-\bar{G}(t_2)}\geq 0$, and $\phi_k(t_2) - \phi_k(t_1) \leq(t_2-t_1)\pth{\frac{k}{n}-\bar{G}(t_2)+\Prob\pth{t_1<W<t_2}} \leq \frac{k}{n} - G(W_{(k)})$. Combining the above two cases, we obtain
    \begin{align*}
    \big|\phi_k(W_{(k)}) - \phi_k(G^{-1}(k/n)) \big| \leq \max\pth{\frac{k}{n} - G(W_{(k)}), \bar{G}(W_{(k)}) - \frac{k}{n}}.
    \end{align*}
    Let $\bar{G}_n(t)\triangleq \frac{1}{n}\sum_{i=1}^n\indc{W_i<t}$. Note that $\bar{G}_n(W_{(k)})\leq\frac{k}{n}\leq G_n(W_{(k)})$, then we have
    \begin{align*}
    \big|\phi(W_{(k)}) - \phi(G^{-1}(k/n)) \big| \leq \max\pth{\|G_n-G\|_\infty, \|\bar{G}_n-\bar{G}\|_\infty}.
    \end{align*}
    Combined with the DKW inequality, we obtain \prettyref{eq:step4}.
    
    \item Show $\Prob\pth{\sup\limits_{k\geq\tau n}\Big|\frac{1}{k}\sum_{i=1}^k W_{(i)} - a(k/n)\Big|>\epsilon} \leq  6 \exp(-n\tau^2\epsilon^2/2)$, the desired \eqref{ineq:ApproxConstr}.    
    Since 
    \begin{align*}
        a(k/n)&=\Expect\left[W\,|\,S_{k/n} = 1\right] = \frac{n}{k}\Expect[W\indc{S_{k/n} = 1}]\\
        &= \frac{n}{k}\pth{\Expect\left[W\indc{W < G^{-1}(k/n)}\right] + p(u)G^{-1}(k/n)\Prob(W = G^{-1}(k/n))}\\
        &= \frac{n}{k}\pth{\Expect\left[W\indc{W < G^{-1}(k/n)}\right] + G^{-1}(k/n)\pth{\frac{k}{n}-\Prob\pth{W < G^{-1}(k/n)}}}\\
        &= \frac{n}{k}\phi(G^{-1}(k/n)),
    \end{align*}
    we have
    \begin{align*}
        &\quad\,\Prob\pth{\sup\limits_{k\geq\tau n}\Big|\frac{1}{k}\sum_{i=1}^k W_{(i)} - a(k/n)\Big|>\epsilon}\\
        &\leq \Prob\pth{\sup\limits_{k\geq\tau n}\Big|\frac{1}{k}\sum_{i=1}^k W_{(i)} - \frac{n}{k}\phi(W_{(k)})\Big|>\frac{\epsilon}{2}} + \Prob\pth{\sup\limits_{k\geq\tau n}\frac{n}{k}\,\Big|\phi(W_{(k)}) - \phi(G^{-1}(k/n)) \Big|>\frac{\epsilon}{2}}\\
        &\overset{\eqref{eq:step3}\eqref{eq:step4}}{\leq} 2\exp(-n\tau^2\epsilon^2/2) + 4\exp(-n\tau^2\epsilon^2/2)\leq 6 \exp(-n\tau^2\epsilon^2/2).
    \end{align*}
\end{enumerate}

Using similar arguments, we can show \eqref{ineq:ApproxObj}. The proof is completed.

\subsection{Proof of \prettyref{lmm:property_of_b_doule_star}}\label{app:two-value}

To start, note that if $\Expect[a(U)] \leq 0$, we have $U=0$ almost surely, so $b^{**}(0)=b(0)=\pi_1$. Since $\Expect[a(U)]\leq a(1)=\pi_0$ for any random variable $U$ supported on $[0,1]$, we have $b^{**}(\alpha)=b(1)=0$ for any $\alpha\in[\pi_0,1]$. Next, we show $b^{**}$ is convex. For any $\lambda\in[0,1]$, and $\alpha_1$, $\alpha_2\in[0,1]$, if random variables $U_1$, $U_2$ are supported on $[0,1]$ and satisfy $\Expect[a(U_1)] \leq \alpha_1$, $\Expect[a(U_2)] \leq \alpha_2$, consider the random variable $U$ such that $U= U_1$ with probability $\lambda$ and $U= U_2$ with probability $1-\lambda$. Then $U\in[0,1]$ almost surely and
\[
\Expect[a(U)] = \lambda\Expect[a(U_1)] + (1-\lambda)\Expect[a(U_2)]\leq \lambda\alpha_1 + (1-\lambda)\alpha_2.
\]
Thus,
\[
b^{**}(\lambda\alpha_1 + (1-\lambda)\alpha_2) \leq \Expect[b(U)] = \lambda\Expect[b(U_1)] + (1-\lambda)\Expect[b(U_2)].
\]
Optimizing over $U_1$ and $U_2$, we obtain $b^{**}(\lambda\alpha_1 + (1-\lambda)\alpha_2) \leq \lambda b^{**}(\alpha_1)+(1-\lambda)b^{**}(\alpha_2)$.

Next, we show $b^{**}$ is the GCM of $b^{*}$, namely, 
\[
b^{**}(\alpha) = \sup \{C(\alpha): C(t) \leq b^*(t),\forall\,t\in[0,1],\,C(\cdot) \text{ is convex on }[0,1]\}.
\]
If there exists a convex function $\tilde{C}: [0,1]\to\reals$ such that $\tilde{C}(\cdot) \leq  b^*(\cdot)$ and $\tilde{C}(\alpha) >  b^{**}(\alpha)$ for some $\alpha$, let $U^*$ denote the optimal solution\footnote{To be precise, at this point, we have not shown that the optimal solution exists. Nevertheless, we can choose a sequence of feasible solutions $\{U_k\}_{k=1}^\infty$ such that $\lim_{k\to\infty}\Expect[b(U_k)] = b^{**}(\alpha)$, and the proof also applies.} of \eqref{eq:opt_prob_a&b_random}. Since $\tilde{C}(\alpha) >b^{**}(\alpha)\geq 0$, we have $\alpha\in[0,\pi_0)$, so we can assume $\Expect[a(U^*)]= \alpha$ without loss of generality. In fact, if $\alpha' \triangleq  \Expect[a(U^*)] < \alpha$, consider $\tilde{U}^*$ such that $\tilde{U}^* = U^*$ with probability $p$ and $\tilde{U}^* = 1$ with probability $1-p$, where $p = (\pi_0-\alpha)/(\pi_0-\alpha')\in(0,1)$. Then $\Expect[a(\tilde{U}^*)] = p\Expect[a(U^*)] + (1-p)a(1) = p\alpha' + (1-p)\pi_0=\alpha$ and $\Expect[b(\tilde{U}^*)] = p\Expect[b(U^*)] + (1-p)b(1) =pb^{**}(\alpha)\leq b^{**}(\alpha)$. Then we have
\begin{align*}
\tilde{C}(\alpha) &>  b^{**}(\alpha) = \Expect[b(U^*)] \geq \Expect[b^*(a(U^*))]\geq \Expect[\tilde{C}(a(U^*))] \geq \tilde{C}(\Expect[a(U^*)]) = \tilde{C}(\alpha)
\end{align*}
where the second inequality is because $b^*(a(u))\leq b(u)$ for any $u\in[0,1]$ by the definition of $b^*$. This leads to a contradiction.

Finally, we show \eqref{eq:opt_prob_a&b_random} admits a binary-valued solution.
\begin{supplemma}
The optimization problem \eqref{eq:opt_prob_a&b_random} is equivalent to
\begin{align}
\begin{split}
    b^{**}(\alpha)= &\inf \,\,\,\Expect[ b^*(Z)]\\
    & \,{\rm s.t.}\,\,\,\,\Expect[Z] \leq \alpha\
\end{split}
\label{eq:opt_prob_Z}
\end{align}
where the infimum is over all random variables $Z$ supported on $[0,1]$.
\end{supplemma}
\begin{proof}
For any feasible solution $Z$ of \eqref{eq:opt_prob_Z}, choose $U = u^*(Z)$, where $u^*(\cdot)$ is the optimal solution of \eqref{eq:opt_prob_a&b} given in \prettyref{lmm:properties_of_b}. Then $U\in[0,1]$ almost surely and $\Expect[a(U)] = \Expect[a(u^*(Z))]\leq\Expect[Z]\leq\alpha$, and $\Expect[b(U)] = \Expect[b(u^*(Z))] = \Expect[ b^*(Z)]$.

On the other hand, for any feasible solution $U$ of \eqref{eq:opt_prob_a&b_random}, choose $Z=a(U)$. Then $Z\in[0,1]$ almost surely and $\Expect[Z]\leq\alpha$, and $\Expect[ b^*(Z)] = \Expect[ b^*(a(U))]\leq \Expect[b(U)]$. Therefore, \eqref{eq:opt_prob_a&b_random} and \eqref{eq:opt_prob_Z} are equivalent.
\end{proof}

\begin{supplemma}\label{supplmm:two_vals}
The problem \eqref{eq:opt_prob_Z} admits an optimal solution $Z^*=Z^*(\alpha)$ which takes at most two values. Thus 
$U^*=u^*(Z^*)$ is an optimal solution of 
\eqref{eq:opt_prob_a&b_random} and takes at most two values.
\end{supplemma}
\begin{proof}
Define $R\triangleq \{(z,t): z\in[0,1], t =  b^*(z)\} \subset \reals^2$. By \prettyref{lmm:properties_of_b}, $  b^*$ is decreasing and right-continuous, so $\lim_{z'\to z^-} b^*(z')$ exists for any $z\in(0,1]$ and $ b^*(z-)\triangleq \lim_{z'\to z^-} b^*(z') \geq b^*(z)$. Let $S\triangleq \{z\in(0,1]:\, b^*(z-)> b^*(z)\}$ and \[
\bar{R} \triangleq  R\,\cup\,\pth{\cup_{z\in S}\{(z,t):\, b^*(z-)\geq t> b^*(z)\}}.
\]
Then $\bar{R}$ is a compact and connected\footnote{In fact, by its construction, $\bar{R}$ is a continuous curve from $(0,\pi_1)$ to $(1,0)$ in $\reals^2$.} subset of $\reals^2$, and thus its convex hull, denoted by $\text{conv}(\bar{R})$, is also compact. For any feasible solution $Z$ of \eqref{eq:opt_prob_Z}, since $(Z, b^*(Z))\in R \subset \bar{R}$ almost surely, we have \[(\Expect[Z], \Expect[ b^*(Z)])\in \text{cl}\pth{\text{conv}(\bar{R})}=\text{conv}(\bar{R}),\] where $\text{cl}(A)$ denotes the closure of a set $A$. By the Fenchel-Eggleston-Carath\'eodory theorem \cite[Chapter 2, Theorem 18]{eggleston1958convexity}, there exists $(z_i, t_i)\in \bar{R},\, i=1,2$, and $p\in [0,1]$, such that $\Expect[Z] = pz_1 + (1-p)z_2$, and $\Expect[ b^*(Z)] = pt_1 + (1-p)t_2$. By the construction of $\bar{R}$, we have $t_i\geq  b^*(z_i),\,i=1,2$. Consider the random variable $\tilde{Z}$ such that $\Prob(\tilde{Z}=z_1)=p$ and $\Prob(\tilde{Z}=z_2)=1-p$. Then $\tilde{Z}\in[0,1]$ almost surely and $\Expect[\tilde{Z}]= pz_1+(1-p)z_2=\Expect[Z]\leq\alpha$, $\Expect[ b^*(\tilde{Z})] = p b^*(z_1)+(1-p) b^*(z_2) \leq pt_1 + (1-p)t_2=\Expect[ b^*(Z)]$. Therefore, it suffices to consider such $Z$ that takes at most two values.

Next we prove the existence of the optimizer. Let $\{Z_n\}_{n=1}^\infty$ be a sequence of feasible solutions of \eqref{eq:opt_prob_Z} such that $\lim_{n\to\infty}\Expect[ b^*(Z_n)] =  b^{**}(\alpha)$, and $\Prob\pth{Z_n = z_{n,1}} = p_n$, $\Prob\pth{Z_n = z_{n,2}} = 1-p_n$, where $p_n,z_{n,1},z_{n,2}\in[0,1]$. 
By passing to a subsequence if necessary, we can assume $\lim_{n\to\infty}p_n = p^*\in[0,1]$, $\lim_{n\to\infty}z_{n,1} = z_1^*\in[0,1]$ and $\lim_{n\to\infty}z_{n,2} = z_2^*\in[0,1]$. 
Let $Z^*$ be the random variable such that $\Prob\pth{Z^*=z^*_1}=p^*$ and $\Prob\pth{Z^*=z^*_2}=1-p^*$. Then clearly, $\Expect[Z^*] = \lim_{n\to\infty} \Expect[Z_n] \leq\alpha$  
and 
\begin{align*}
\Expect[ b^*(Z^*)] &= p^*  b^*(z^*_1) + (1-p^*)  b^*(z^*_2)\\
&= p^* b^*\pth{\lim_{n\to\infty} z_{n,1}} + (1-p^*)  b^*\pth{\lim_{n\to\infty} z_{n,2}}\\
&\leq p^* \liminf_{n\to\infty}  b^*(z_{n,1}) + (1-p^*) \liminf_{n\to\infty}  b^*(z_{n,2})\\
&= \liminf_{n\to\infty}\pth{p_n  b^*(z_{n,1}) + (1-p_n)  b^*(z_{n,2})} \\
&= \liminf_{n\to\infty} \Expect[ b^*(Z_n)]= b^{**}(\alpha)
\end{align*}
where the inequality applies by the lower semi-continuity of $ b^*$ since it is decreasing and right-continuous (\prettyref{lmm:properties_of_b}). Therefore, the optimal objective is achieved by $Z^*$.
\end{proof}

By \prettyref{supplmm:two_vals}, there exist $(\alpha_1,\alpha_2,p)$, $0\leq \alpha_1\leq\alpha\leq \alpha_2\leq 1$ and $p\in[0,1]$, such that $p\alpha_1+(1-p)\alpha_2\leq\alpha$ and
\begin{align*}
Z^* = 
\begin{cases}
\alpha_1\quad&\text{with probability }p\\
\alpha_2\quad&\text{with probability }1-p
\end{cases}
\end{align*}
Since $b^*$ is a decreasing function, we can assume $p\alpha_1+(1-p)\alpha_2=\alpha$ without loss of generality. Therefore, we have 
\begin{align*}
U^* = 
\begin{cases}
u^*(\alpha_1)\quad&\text{with probability }p\\
u^*(\alpha_2)\quad&\text{with probability }1-p
\end{cases}
\end{align*}
where $p=(\alpha_2-\alpha)/(\alpha_2-\alpha_1)$ if $\alpha_1<\alpha_2$ and $p=1$ if $\alpha_1=\alpha_2=\alpha$. The proof of \prettyref{lmm:property_of_b_doule_star} is completed.

\subsection{Proof of \prettyref{lmm:eq_for_pFDR&pFNR}}\label{app:ProofpFDR}
We show that 
\begin{align} 
    &\Expect\left[\frac{\sum_{i=1}^n\delta_i(1-\theta_i)}{\sum_{i=1}^n\delta_i}\,\Bigg |\,\sum_{i=1}^n\delta_i >0\right]=\Prob\pth{\theta_1 = 0\,|\,\delta_1=1}\text{ and }\label{eq:pFDR}\\
    &\Expect\left[\frac{\sum_{i=1}^n(1-\delta_i)\theta_i}{\sum_{i=1}^n (1-\delta_i)}\,\Bigg |\,\sum_{i=1}^n\delta_i <n\right]=\Prob\pth{\theta_1 = 1\,|\,\delta_1=0}.\label{eq:pFNR}
\end{align} 
As a result, 
\begin{align*}
    &\FDR(\delta^n)=\Expect\left[\frac{\sum_{i=1}^n\delta_i(1-\theta_i)}{\sum_{i=1}^n\delta_i}\,\Bigg |\,\sum_{i=1}^n\delta_i >0\right]\Prob\pth{\sum_{i=1}^n\delta_i >0}\leq\Prob\pth{\theta_1 = 0\,|\,\delta_1=1}\text{ and }\\
    &\FNR(\delta^n)=\Expect\left[\frac{\sum_{i=1}^n(1-\delta_i)\theta_i}{\sum_{i=1}^n(1-\delta_i)}\,\Bigg |\,\sum_{i=1}^n\delta_i <n\right]\Prob\pth{\sum_{i=1}^n\delta_i <n}\leq\Prob\pth{\theta_1 = 1\,|\,\delta_1=0}.
\end{align*}
In fact, for any non-empty $\calI \subset[n]$,
\begin{align*}
    &\quad\,\Expect\left[\frac{\sum_{i=1}^n\delta_i(1-\theta_i)}{\sum_{i=1}^n\delta_i}\,\Bigg |\,\delta_i=\indc{i\in\calI}, i\in[n]\right]= \frac{1}{|\calI|}\Expect\left[\sum_{i\in\calI}(1-\theta_i)\,\Big|\,\delta_i=1,i\in\calI\right]\\
    &= \frac{1}{|\calI|}\sum_{i\in\calI}\Expect\left[(1-\theta_i)\,|\,\delta_i=1\right]= \Prob\pth{\theta_1 = 0\,|\,\delta_1=1}
\end{align*}
by noting that $(\theta_1,\delta_1),\cdots,(\theta_n,\delta_n)$ are \iid pairs. Since this equality holds for any realization of $\delta^n$ such that $\sum_{i=1}^n\delta_i >0$, we conclude \eqref{eq:pFDR}, and similarly \eqref{eq:pFNR}, hold.

\subsection{Proof of \prettyref{lmm:best_matching}}\label{app:ProofBestMatching}
Assume $a_{(i)}=a_{s_i}$. Then, 
\begin{align*}
    \sum_{i=1}^kb_{(i)}-\sum_{i=1}^ka_{(i)}&=\sum_{i=1}^kb_{(i)}-\sum_{i=1}^ka_{s_i}\leq\sum_{i=1}^kb_{s_i}-\sum_{i=1}^ka_{s_i}\leq\sum_{i=1}^k|b_{s_i}-a_{s_i}|\leq\sum_{i=1}^n|b_i-a_i|.
\end{align*}
Similarly, $\sum_{i=1}^ka_{(i)}-\sum_{i=1}^kb_{(i)}\leq\sum_{i=1}^n|b_i-a_i|$. Therefore, $|\sum_{i=1}^kb_{(i)}-\sum_{i=1}^ka_{(i)}|\leq\sum_{i=1}^n|b_i-a_i|$ for any $k$.

\subsection{Proof of \prettyref{lmm:W&W_tilde}}\label{app:ProofWW_tilde}
We begin by recalling some standard notation from information theory:
Let $H(X)\equiv H(P_X) \triangleq\Expect[-\log(P_X(X))]$ denote the entropy of a discrete random variable $X$, where $P_X(\cdot)$ is the probability mass function of $X$. 
Similarly, the conditional entropy of $X$ given $Z$ is $H(X|Z) = \Expect_Z[H(P_{X|Z})]$.
The Kullback-Leibler divergence  between distributions $P$ and $Q$ is $D(P\|Q)\triangleq \Expect_{X\sim P}[\log \frac{dP}{dQ}(X)]$ if 
$P \ll Q$ and $\infty$ otherwise. The mutual information between random variables $X$ and $Y$ 
is denoted by $I(X;Y)\triangleq D(P_{XY}\|P_XP_Y)$, where $P_{XY}$ is their joint distribution and $P_X$, $P_Y$ are their marginals. 
Similarly, the conditional mutual information between random variables $X$ and $Y$  given $Z$ is 
$I(X;Y|Z)\triangleq \Expect_{Z}[D(P_{XY|Z}\|P_{X|Z}P_{Y|Z})]$.
For discrete $X$, one has $I(X;Y)=H(X)-H(X|Y)$ and $I(X;Y|Z)=H(X|Z)-H(X|Y,Z)$.

We now proceed to the proof of \prettyref{lmm:W&W_tilde}. 
By symmetry and using 
Tao's inequality (a consequence of Pinsker's inequality cf.~e.g.~\cite[Corollary 7.17]{PW-it}), we have 
\[
\Expect[(W_i'-\widetilde{W}_i)^2]=\Expect[(\Expect[\theta_1|X_1] - \Expect[\theta_1|X^n])^2] 
\leq I(\theta_1;X_2,\ldots,X_n|X_1).
\]
Next, 
\begin{align*}
I(\theta_1;X_2,\ldots,X_n|X_1)= & I(\theta_1;X^n) - I(\theta_1;X_1) \\
= & H(\theta_1)-H(\theta_1|X^n) - I(\theta_1;X_1) \\
\overset{(a)}{\leq} & H(\theta_1)- \frac{1}{n} H(\theta^n|X^n) - I(\theta_1;X_1) \\
= & H(\theta_1)- \frac{1}{n} H(\theta^n) + \frac{1}{n}[ I(\theta^n;X^n) - n I(\theta_1;X_1)] \\
\overset{(b)}{\leq} & H(\theta_1)- \frac{1}{n} H(\theta^n)\\
=&\frac{n_1}{n}\log\frac{n}{n_1}+\frac{n_0}{n}\log\frac{n}{n_0}-\frac{1}{n}\log\binom{n}{n_1}\\
\overset{(c)}{\leq} & \frac{\log(2\sqrt{n})}{n} \leq\frac{\log (n+1)}{n}.
\end{align*}	
where 
\begin{itemize}
	\item (a): We apply the fact that joint entropy is at most sum of marginal entropies (\cite[Theorem 1.4]{PW-it}). Furthermore, by exchangeability, $H(\theta_i|X^n)$ does not depend on $i$.
	\item (b) Since $P_{X^n|\theta^n} = \prod_{i=1}^n P_{X_i|\theta_i}$, we have $I(\theta^n;X^n) \leq \sum I(\theta_i;X_i)$ (\cite[Theorem 6.1]{PW-it}). Again by exchangeability, $I(\theta_i;X_i)$ does not depend on $i$.
	\item (c): By the non-asymptotic Stirling's approximation \cite{robbins1955remark}, we have $\binom{n}{n_1}\geq\frac{1}{2\sqrt{n}}(\frac{n}{n_1})^{n_1}(\frac{n}{n_0})^{n_0}$.
\end{itemize}

\subsection{Proof of \prettyref{thm:dependent}}
\label{app:dependent}

For the sake of conciseness, in the proof we assume $\sigma^2=1$ and abbreviate $\Sigma_n \equiv \Sigma$.
We first provide the following key ingredients:
a) approximation of $\tilde W_i = \Expect[\theta_i|X^n]$ by $W_i = \Expect[\theta_i|X_i]$  in an average $L_2$ sense;
b) concentration of partial sums of Lfdr statistics.

\begin{supplemma}
\label{supplmm:dependent-wapprox}	
	\[
	\frac{1}{n}\sum_{i=1}^n \Expect[(\tilde W_i -W_i)^2] \leq C \|\Sigma^{-1/2} - I\|_{\rm sp}.
	\]
	where $C$ is a constant only depending on $\mu$ and $\sigma^2$.
\end{supplemma}
\begin{proof}
	We again apply the information-theoretic argument in the proof 
	of \prettyref{lmm:W&W_tilde} (see \prettyref{app:ProofWW_tilde}), in order to quantify the approximate conditional independence of $\theta_i$  and $X_{\backslash i}$ given the direct observation $X_i$; the difference is that $X^n$ given $\theta^n$ are 
not independent and we do not have exchangeability.
	
	Applying Tao's inequality, we have $\Expect[(W_i-\widetilde{W}_i)^2]\leq I(\theta_i;X_{\backslash i}|X_i)
	= I(\theta_i;X^n) - I(\theta_i;X_i) = H(\theta_i) - H(\theta_i|X^n) - I(\theta_i;X_i)$. So
\begin{align*}
\sum_{i=1}^n \Expect[(W_i-\widetilde{W}_i)^2]
\leq & ~ 	 \sum_{i=1}^n H(\theta_i) - \sum_{i=1}^n H(\theta_i|X^n) - \sum_{i=1}^n I(\theta_i;X_i) \\
\leq & ~ 	\sum_{i=1}^n H(\theta_i) - H(\theta^n|X^n) - \sum_{i=1}^n I(\theta_i;X_i) \\
= & ~ 	I(\theta^n;X^n) - \sum_{i=1}^n I(\theta_i;X_i),
\end{align*}	
where the first inequality follows from joint entropy is at most the sum of the marginal entropies, and the last equality follows from the independence of $\theta_i$'s.
Write $X^n = \mu \theta^n + \Sigma^{1/2} Z^n$, where $Z^n \sim N(0,I_n)$.
Then $\sum_{i=1}^n I(\theta_i;X_i)= I(\theta^n;\mu \theta^n + Z^n)$ and 
$I(\theta^n;X^n) = I(\tilde \theta^n; \mu \tilde \theta^n + Z^n) $, where $\tilde \theta^n \equiv \Sigma^{-1/2} \theta^n$.
To bound the difference of these two mutual informations with additive Gaussian noise, we apply the transportation distance-based bound in \cite[Corollary 4]{PW15}:
\[
|I(\theta^n;\mu \theta^n + Z^n) - I(\theta^n;\mu \tilde \theta^n + Z^n)|
\leq C_0(\mu^2 + 1) \sqrt{n} W_2(\Law(\mu \theta^n+ Z^n), \Law(\mu \tilde \theta^n + Z^n))
\]
for some universal constant $C_0$, where $W_2$ denotes the 2-Wasserstein distance.
Finally, this Wasserstein distance between the two Gaussian convolutions can be bounded using the trivial coupling:
\begin{align*}
W_2(\Law(\mu \theta^n+ Z^n), \Law(\mu \tilde \theta^n + Z^n)) 
\leq & ~ W_2(\Law(\mu \theta^n), \Law(\mu \tilde \theta^n)) \\
\leq & ~ \mu \max_{\theta^n \in \{0,1\}^n} \|(\Sigma^{-1/2} - I)\theta^n\| \\
\leq & ~ \mu \sqrt{n} \|\Sigma^{-1/2} - I\|_{\rm sp}.
\end{align*}
\end{proof}

\begin{supplemma}
\label{supplmm:dependent-concentration}	
	Let $G_n(t) = \frac{1}{n} \sum_{i=1}^n \indc{W_i \leq t}$ the empirical CDF and $G(t)$ the CDF of $W_i$. 
	Then $\|G_n-G\|_\infty \to 0$ in probability provided that $\|\Sigma-I\|_{\rm sp} = o(\sqrt{n})$.	
\end{supplemma}
\begin{proof}
	Recall from \prettyref{ex:gaussian} that 
	$W_i = \frac{\pi_0 \varphi(X_i)}{\pi_0 \varphi(X_i) + \pi_1 \varphi(X_i-\mu)} 
= \frac{\pi_0}{\pi_0 + \pi_1\exp(\mu X_i - \mu^2/2)}$ is a monotone transformation of $X_i$.
Thus $\|F_n-F\|_\infty = \|G_n-G\|_\infty$, where $F_n(t) = \frac{1}{n} \sum_{i=1}^n \indc{X_i \leq t}$ and $F(t)$ are the empirical and population CDF of the $X_i$'s. (In particular, $F$ is standard normal CDF.)	
	We show that for any $\epsilon>0$,
	\[
	\prob{\|F_n-F\|_\infty \geq 2\epsilon} \leq \frac{C \|\Sigma-I\|_{\rm sp}}{\sqrt{n}\epsilon^3}.
	\]
	
	First of all, there exist $t_1<\ldots<t_m$ with $m \leq \frac{1}{\epsilon}+1$, such that 
	$\|F_n-F\|_\infty  \leq \max_{i=1,\ldots,m}|F_n(t_i)-F(t_i)| + \epsilon$.
	This follows by choosing $t_i$'s as the appropriate quantiles of $F$ such that $F(t_i)-F(t_{i-1}) \leq \epsilon, i=1,\ldots,m+1$ with $t_0 \equiv -\infty$ and $t_{m+1} \equiv \infty$, so that for any $t \in (t_{i-1},t_i)$, 
	$F_n(t)-F(t) \leq F_n(t_i)-F(t_i) + \epsilon$ and 
	$F_n(t)-F(t) \geq F_n(t_{i-1})-F(t_{i-1}) - \epsilon$.
	So by a union bound, it suffices to bound $\prob{|F_n(t)-F(t)| \geq \epsilon}$ for a fixed $t$.
	
	Next, $F_n(t) - F(t)  = \frac{1}{n} \sum_{i=1}^n [\indc{X_i \leq t} - F(t)]$. By Chebyshev inequality,
	\[
	\prob{|F_n(t)-F(t)| \geq \epsilon} \leq \frac{1}{n^2\epsilon^2} \pth{n + \sum_{i\neq j} \Cov(\indc{X_i \leq t},\indc{X_j \leq t})}.
	\]
	Note that 
	$\Cov(\indc{X_i \leq t},\indc{X_j \leq t}) \leq \rho(\indc{X_i \leq t},\indc{X_j \leq t}) \leq 
	\rho_{\max}(X_i,X_j) = |\Sigma_{ij}|$,
	where 
	$\rho_{\max}$ denotes the maximal correlation\footnote{The (Hirschfeld-Gebelein-R\'enyi) maximal correlation between random variables $X$ and $Y$ are defined as $\rho_{\max}(X,Y) = \sup\rho(f(X),g(Y))$, where the supremum is taken over all non-zero functions $f$ and $g$.} and for Gaussians
	equals to the magnitude of the usual correlation (see, e.g., \cite[Theorem 33.6]{PW-it}).
	Finally, $\sum_{i \neq j} |\Sigma_{ij}| \leq n \Fnorm{\Sigma - I} \leq n^{3/2} \|\Sigma - I\|_{\rm sp}$.	
\end{proof}
By replacing the DKW inequality with \prettyref{supplmm:dependent-concentration} in \prettyref{app:ProofMainLmm}, we can show the concentration of $\frac{1}{k}\sum_{i=1}^k W_{(i)}$ as stated in \prettyref{lmm:approx_of_obj&constr}. Combined with \prettyref{supplmm:dependent-wapprox}, we have the concentration of $\frac{1}{k}\sum_{i=1}^k \widetilde{W}_{(i)}$ as stated in \prettyref{prop:UnionBound_fixed}, which establishes the negative parts of \prettyref{thm:ora_FNR_vs_ora_mFNR} and \prettyref{thm:high_prob_formulation} in view of the proofs of \prettyref{thm:ora_FNR_vs_ora_mFNR_fixed} and \prettyref{thm:high_prob_formulation_fixed}. For the positive parts, note that the Law of Large Numbers continues to hold by applying Chebyshev's inequality and the concentration assumption in the spectral norm, then the proofs of \prettyref{thm:ora_FNR_vs_ora_mFNR} and \prettyref{thm:high_prob_formulation} in the independent two-group model are also applied here.

\section{Consistency of data-driven procedure}
\subsection{Proof of \prettyref{prop:uniform_consistency}}\label{app:ProofConsistency1}

\begin{supplemma}\label{supplmm:uniform_convergence}
$\calI=[c,d]$, where $-\infty<c<d<+\infty$, is a closed interval on the real line. Let $\{\hat{f}_n\}_{n=1}^\infty$ be a sequence of decreasing random functions on $\calI$, and $f$ be a decreasing and continuous function on $\calI$. Suppose that for any $x\in\calI$, $\hat{f}_n(x)\toprob f(x)$, as $n\to\infty$. Then $\sup_{x\in\calI}|\hat{f}_n(x)-f(x)|\toprob 0$, as $n\to\infty$.
\end{supplemma}
\begin{proof}
If $\sup_{x\in\calI}|\hat{f}_n(x)-f(x)|\overset{p}{\nrightarrow}0$, then there exists $\delta>0$, $\eta>0$, $\{n_k\}_{k=1}^\infty$ and $\{x_k\}_{k=1}^\infty\subset\calI$, such that $\lim\limits_{k\to\infty} n_k = \infty$ and $\Prob\pth{|\hat{f}_{n_k}(x_k)-f(x_k)|>\delta} \geq 3\eta$ for any $k\geq 1$. Since $\{x_k\}_{k=1}^\infty\subset[c,d]$, it has a limiting point. Without loss of generality, assume $\lim\limits_{k\to\infty}x_k=x_0\in[c,d]$. If $x_0\in(c,d)$, since $f$ is continuous and decreasing, there exists $x^-_0,\,x^+_0$ such that $c\leq x^-_0 < x_0 < x^+_0 \leq d$ and $f(x^-_0)-\frac{\delta}{4} < f(x_0) < f(x^+_0) + \frac{\delta}{4}$. If $x_0=c$, let $x^-_0=c$ and $x^+_0$ is chosen as above. If $x_0=d$, let $x^+_0=d$ and $x^-_0$ is chosen as above. By noting that $\lim\limits_{k\to\infty}x_k=x_0$, $\hat{f}_{n_k}(x^-_0)\toprob f(x^-_0)$ and $\hat{f}_{n_k}(x^+_0)\toprob f(x^+_0)$, as $k\to\infty$, there exists an enough large $k$ such that $x^-_0\leq x_k \leq x^+_0$, $\Prob\pth{|\hat{f}_{n_k}(x^-_0)-f(x^-_0)|\leq\frac{\delta}{2}} \geq 1-\eta$ and $\Prob\pth{|\hat{f}_{n_k}(x^+_0)-f(x^+_0)|\leq\frac{\delta}{2}} \geq 1-\eta$. Consider the events $E_1^k=\{|\hat{f}_{n_k}(x^-_0)-f(x^-_0)|\leq\frac{\delta}{2}\}$, $E_2^k=\{|\hat{f}_{n_k}(x^+_0)-f(x^+_0)|\leq\frac{\delta}{2}\}$ and $E_3^k=\{|\hat{f}_{n_k}(x_k)-f(x_k)|>\delta\}$. Then we have $\Prob\pth{E_1^k\cap E_2^k} \geq \Prob\pth{E_1^k} + \Prob\pth{E_1^k} - 1 \geq 1-2\eta$ and $\Prob\pth{E_1^k\cap E_2^k\cap E_3^k}\geq \Prob\pth{E_1^k\cap E_2^k} + \Prob\pth{E_3^k} - 1 \geq \eta > 0$. However, on the event $E_1^k\cap E_2^k$, we have 
\begin{align*}
    &\hat{f}_{n_k}(x_k)-f(x_k) \leq \hat{f}_{n_k}(x^-_0)-f(x^+_0) \leq |\hat{f}_{n_k}(x^-_0)-f(x^-_0)| + f(x^-_0)-f(x^+_0)<\frac{\delta}{2} + \frac{\delta}{2}=\delta,\\
    &\hat{f}_{n_k}(x_k)-f(x_k) \geq \hat{f}_{n_k}(x^+_0)-f(x^-_0) \geq -|\hat{f}_{n_k}(x^+_0)-f(x^+_0)| + f(x^+_0)-f(x^-_0)>-\frac{\delta}{2} - \frac{\delta}{2}=-\delta
\end{align*} 
by noting that both $\hat{f}_{n_k}$ and $f$ are decreasing. Hence, $|\hat{f}_{n_k}(x_k)-f(x_k)|<\delta$ always holds on the event $E_1^k\cap E_2^k$, which contradicts that $\Prob\pth{E_1^k\cap E_2^k\cap E^k_3}>0$. Therefore, $\sup_{x\in\calI}|\hat{f}_n(x)-f(x)|\toprob 0$, as $n\to\infty$.
\end{proof}

Now we give the proof of \prettyref{prop:uniform_consistency} in the following steps.
\begin{enumerate}
    \item Show $\omega_n\triangleq \max_{1\leq i\leq n}|\widehat{W}_{(i)}-W_{(i)}|\toprob 0$, as $n\to\infty$: 
    
    $g$ is positive on $(0,1)$ shows that $\lim\limits_{x\to 0^+}f(x)=+\infty$ and $\lim\limits_{x\to 1^-}f(x)=\pi_0$. Hence, for any $\delta\in(0,1)$, there exists $0<c_1<c_2<1$ such that $\frac{\pi_0}{f(c_1)}<\frac{\delta}{16}$ and $\frac{\pi_0}{f(c_2)}>1-\frac{\delta}{16}$. Then
    \begin{align}
        &\quad\,\max_{1\leq i\leq n}|\widehat{W}_i-W_i| \leq \sup_{x\in(0,1)} \bigg|\frac{\pi_0}{f(x)}-\frac{\pi_0}{\hat{f}(x)}\wedge 1\bigg|\notag\\
        &\leq \sup_{x\in(0,c_1)} \bigg|\frac{\pi_0}{f(x)}-\frac{\pi_0}{\hat{f}(x)}\wedge 1\bigg| + \sup_{x\in[c_1,c_2]} \bigg|\frac{\pi_0}{f(x)}-\frac{\pi_0}{\hat{f}(x)}\wedge 1\bigg| + \sup_{x\in(c_2,1)} \bigg|\frac{\pi_0}{f(x)}-\frac{\pi_0}{\hat{f}(x)}\wedge 1\bigg|\notag\\
        &\leq \frac{\pi_0}{f(c_1)} + \frac{\pi_0}{\hat{f}(c_1)} + \sup_{x\in[c_1,c_2]} \bigg|\frac{\pi_0}{f(x)}-\frac{\pi_0}{\hat{f}(x)}\wedge 1\bigg|\sup_{x\in(c_2,1)} \bigg|\frac{\pi_0}{f(x)}-\frac{\pi_0}{\hat{f}(x)}\wedge 1\bigg|\label{ineq:bd}
    \end{align}
    by noting that both $f$ and $\hat{f}$ are decreasing. Since $\frac{\pi_0}{f(c_1)}<\frac{\delta}{16}$, we have
    \begin{align}
        \frac{\pi_0}{f(c_1)} + \frac{\pi_0}{\hat{f}(c_1)}< \frac{\pi_0}{\hat{f}(c_1)}- \frac{\pi_0}{f(c_1)}+\frac{\delta}{8}.\label{ineq:bd1}
    \end{align}
    For any $x\in[c_1,c_2]$, if $\hat{f}(x)\geq \pi_0$, then $\bigg|\frac{\pi_0}{f(x)}-\frac{\pi_0}{\hat{f}(x)}\wedge 1\bigg| = \bigg|\frac{\pi_0}{f(x)}-\frac{\pi_0}{\hat{f}(x)}\bigg|=\frac{\pi_0 |f(x)-\hat{f}(x)|}{f(x)\hat{f}(x)}\leq\frac{|f(x)-\hat{f}(x)|}{\pi_0}$. If $\hat{f}(x)< \pi_0$, then $\bigg|\frac{\pi_0}{f(x)}-\frac{\pi_0}{\hat{f}(x)}\wedge 1\bigg| = 1-\frac{\pi_0}{f(x)} =\frac{f(x)-\pi_0}{f(x)}\leq\frac{f(x)-\hat{f}(x)}{\pi_0}$. Hence, 
    \begin{align*}
        \sup_{x\in[c_1,c_2]} \bigg|\frac{\pi_0}{f(x)}-\frac{\pi_0}{\hat{f}(x)}\wedge 1\bigg|\leq \frac{1}{\pi_0}\sup_{x\in[c_1,c_2]}|f(x)-\hat{f}(x)|.
    \end{align*}
    For any $x\in(c_2,1)$, we have $1\geq \frac{\pi_0}{f(x)}\geq \frac{\pi_0}{f(c_2)}>1-\frac{\delta}{16}$. If $\frac{\pi_0}{\hat{f}(x)} > 1-\frac{\delta}{16}$, then $\bigg|\frac{\pi_0}{f(x)}-\frac{\pi_0}{\hat{f}(x)}\wedge 1\bigg|\leq \frac{\delta}{16}$. If $\frac{\pi_0}{\hat{f}(x)} \leq 1-\frac{\delta}{16}$, then
    \begin{align*}
        &\quad\,\bigg|\frac{\pi_0}{f(x)}-\frac{\pi_0}{\hat{f}(x)}\wedge 1\bigg| = \frac{\pi_0}{f(x)}-\frac{\pi_0}{\hat{f}(x)} = \frac{\pi_0 (\hat{f}(x)-f(x))}{f(x)\hat{f}(x)}\leq \frac{\hat{f}(x)-f(x)}{\pi_0}\\
        &\leq \frac{\hat{f}(c_2)-f(x)}{\pi_0} = \frac{\hat{f}(c_2)-f(c_2)}{\pi_0} + \frac{f(c_2)-f(x)}{\pi_0} \leq \frac{|\hat{f}(c_2)-f(c_2)|}{\pi_0} + \frac{\delta}{8}
    \end{align*}
    where the second inequality is by noting that $\hat{f}$ is decreasing, and the last inequality is because $\frac{f(c_2)-f(x)}{\pi_0} \leq 1/(1-\frac{\delta}{16})-1 = \frac{\delta}{16-\delta}<\frac{\delta}{8}$. Hence, we have 
    \begin{align}
        \sup\limits_{x\in(c_2,1)} \bigg|\frac{\pi_0}{f(x)}-\frac{\pi_0}{\hat{f}(x)}\wedge 1\bigg|\leq \frac{|\hat{f}(c_2)-f(c_2)|}{\pi_0} + \frac{\delta}{8}.\label{ineq:bd3}
    \end{align}
    Combining $\eqref{ineq:bd}-\eqref{ineq:bd3}$, we obtain
    \begin{align*}
        \max_{1\leq i\leq n}|\widehat{W}_i-W_i| \leq \frac{\pi_0}{\hat{f}(c_1)}  - \frac{\pi_0}{f(c_1)} + \frac{1}{\pi_0}\sup\limits_{x\in[c_1,c_2]}|f(x)-\hat{f}(x)| + \frac{|\hat{f}(c_2)-f(c_2)|}{\pi_0} + \frac{\delta}{4}.
    \end{align*}
    It follows that $\Prob\pth{\omega_n > \delta} \leq \Prob\pth{\sup_{x\in[c_1,c_2]}|f(x)-\hat{f}(x)|>\frac{\pi_0\delta}{4}}
        + \Prob\pth{\frac{\pi_0}{\hat{f}(c_1)}  - \frac{\pi_0}{f(c_1)} > \frac{\delta}{4}} + \Prob\pth{|\hat{f}(c_2)-f(c_2)|>\frac{\pi_0\delta}{4}}$. By Corollary 3.1 of \cite{groeneboom2014nonparametric}, we have $\hat{f}(x)\toprob f(x)$ for any $x\in(0,1)$, so $\frac{\pi_0}{\hat{f}(c_1)}\toprob \frac{\pi_0}{f(c_1)}$. Since $\hat{f}$ is decreasing and $f$ is decreasing and continuous on $[c_1,c_2]$, we have $\sup_{x\in[c_1,c_2]}|f(x)-\hat{f}(x)|\toprob 0$ by \prettyref{supplmm:uniform_convergence}. Thus, $\Prob\pth{\omega_n > \delta}\to 0$, as $n\to\infty$. We complete the proof.

    \item Show $\frac{1}{n}\sum_{i=1}^n \widehat{W}_i\indc{\widehat{W}_i\leq y} \toprob  \Expect[W\indc{W\leq y}]$ for any $y\in[0,1]$:
    
    Fix $y\in[0,1]$. Let $\nu_n(t)=\bigg|\frac{1}{n}\sum_{i=1}^n W_i \indc{W_i\leq t}-\Expect[W\indc{W\leq t}]\bigg|$. By the Weak Law of Large Numbers, $\nu_n(t)\toprob 0$, as $n\to\infty$, for any $t\in\reals$. Since $G$ is continuous at $y$, for any $\delta>0$, there exists $\epsilon\in(0,\frac{\delta}{3})$, such that $|G(\tilde{y})-G(y)|\leq \frac{\delta}{3}$ for any $\tilde{y}\in[y-\epsilon,y+\epsilon]$. Then
    \begin{align}
        &\quad\,\Prob\pth{\bigg|\frac{1}{n}\sum_{i=1}^n \widehat{W}_i\indc{\widehat{W}_i\leq y} - \Expect[W\indc{W\leq y}]\bigg|>\delta}\notag\\
        &\leq \Prob\pth{\bigg|\frac{1}{n}\sum_{i=1}^n \widehat{W}_i\indc{\widehat{W}_i\leq y} - \Expect[W\indc{W\leq y}]\bigg|>\delta,\,\omega_n\leq \epsilon} + \Prob\pth{\omega_n > \epsilon}.\label{ineq:bdd_part2}
    \end{align}
    If $\omega_n\leq\epsilon$, we have 
    \begin{align*}
        \frac{1}{n}\sum_{i=1}^n \widehat{W}_i\indc{\widehat{W}_i\leq y} &\leq \frac{1}{n}\sum_{i=1}^n (W_i+\epsilon)\indc{W_i\leq y+\epsilon}\leq \frac{1}{n}\sum_{i=1}^n W_i\indc{W_i\leq y+\epsilon} + \epsilon\\
        &\leq \Expect[W\indc{W\leq y+\epsilon}] + \nu_n(y+\epsilon) + \frac{\delta}{3}\\
        &= \Expect[W\indc{W\leq  y}] + \Expect[W\indc{y< W \leq y+\epsilon}] + \nu_n(y+\epsilon) + \frac{\delta}{3}\\
        &\leq \Expect[W\indc{W\leq y}] + G(y+\epsilon) - G(y) + \nu_n(y+\epsilon) + \frac{\delta}{3}\\
        &\leq \Expect[W\indc{W\leq y}]  + \nu_n(y+\epsilon) + \frac{2}{3}\delta.
    \end{align*}
    Similarly, we have $\frac{1}{n}\sum_{i=1}^n \widehat{W}_i\indc{\widehat{W}_i\leq y}\geq \Expect[W\indc{W\leq y}]  - \nu_n(y+\epsilon) - \frac{2}{3}\delta$. Thus, 
    \[
    \bigg|\frac{1}{n}\sum_{i=1}^n \widehat{W}_i\indc{\widehat{W}_i\leq y} - \Expect[W\indc{W\leq y}]\bigg| \leq \nu_n(y+\epsilon) + \frac{2}{3}\delta.
    \]
    Continuing with \eqref{ineq:bdd_part2}, 
    \begin{align*}
        &\quad\,\Prob\pth{\bigg|\frac{1}{n}\sum_{i=1}^n \widehat{W}_i\indc{\widehat{W}_i\leq y} - \Expect[W\indc{W\leq y}]\bigg|>\delta}\\
        &\leq \Prob\pth{\nu_n(y+\epsilon)>\frac{\delta}{3}} + \Prob\pth{\omega_n > \epsilon}\to 0.
    \end{align*} This completes the proof for part 2.
    
    \item Similarly, $\frac{1}{n}\sum_{i=1}^n \indc{\widehat{W}_i\leq y} \toprob  \Expect[\indc{W\leq y}]=G(y)$, as $n\to\infty$, for any $y\in[0,1]$.
    
    \item Show $\sup\limits_{y\in[0,1]}\big|\widehat{A}(y)- A(y)\big|\toprob  0$, as $n\to\infty$: 
    
    Since $g$ is positive on $(0,1)$, $G(y)>0$ for any $y>0$. By part 2, 3 and the Slutsky's theorem,
    \[
    \widehat{A}(y) = \frac{\frac{1}{n}\sum_{i=1}^n \widehat{W}_i\indc{\widehat{W}_i\leq  y}}{\frac{1}{n}\vee\frac{1}{n}\sum_{i=1}^n \indc{\widehat{W}_i \leq y}} \toprob  \frac{\Expect[W\indc{W\leq y}]}{G(y)} = A(y)\] for any $y\in(0,1]$. Note that $\widehat{A}(0)=A(0)=0$, so $\widehat{A}(y)\toprob A(y)$ for any $y\in[0,1]$. Since $A(y)=\frac{\int_0^y sg(s)\,ds}{G(y)}$ is increasing and continuous on $(0,1]$, and $\lim_{y\to 0^+}A(y)=\lim_{y\to 0^+}\frac{\int_0^y sg(s)\,ds}{G(y)} = \lim_{y\to 0^+}\frac{yg(y)}{g(y)}=0=A(0)$, we have $A$ is increasing and continuous on $[0,1]$. Note that $\widehat{A}$ is increasing on $[0,1]$, so by \prettyref{supplmm:uniform_convergence} we have $\sup_{y\in[0,1]}\big|\widehat{A}(y)- A(y)\big|\toprob  0$, as $n\to\infty$.
    
    \item Similarly, $\sup\limits_{y\in[0,1]}\big|\widehat{B}(y)- B(y)\big|\toprob  0$, as $n\to\infty$.
    
    \item Show that there exists one choice of $\widehat{y^*}$ such that $\sup\limits_{\alpha\in[0,1]}\big|\widehat{y^*}(\alpha) - y^*(\alpha)|\toprob 0$, as $n\to\infty$, where $y^*(\alpha)$ is given in \eqref{eq:y*}:
    
    The proof is based on \cite[Lemma A.5]{sun2007oracle}. Define two continuous functions $\widehat{A}^-(y)$ and $\widehat{A}^+(y)$ on $[0,1]$ such that for $y\in[\widehat{W}_{(k)},\widehat{W}_{(k+1)}]$, $k=0,1,\cdots,n$,
    \[\widehat{A}^-(y) = \frac{\widehat{W}_{(k+1)}-y}{\widehat{W}_{(k+1)}-\widehat{W}_{(k)}}\,\widehat{A}(\widehat{W}_{(k-1)}) + \frac{y-\widehat{W}_{(k)}}{\widehat{W}_{(k+1)}-\widehat{W}_{(k)}}\,\widehat{A}(\widehat{W}_{(k)})\] and \[\widehat{A}^+(y) = \frac{\widehat{W}_{(k+1)}-y}{\widehat{W}_{(k+1)}-\widehat{W}_{(k)}}\,\widehat{A}(\widehat{W}_{(k)}) + \frac{y-\widehat{W}{(k)}}{\widehat{W}_{(k+1)}-\widehat{W}_{(k)}}\,\widehat{A}(\widehat{W}_{(k+1)}),\] where we assume $\widehat{W}_{(-1)}=\widehat{W}_{(0)}=0$ and $\widehat{W}_{(n+1)}=1$. Then we have $\widehat{A}^-(y)\leq\widehat{A}(y)\leq\widehat{A}^+(y)$, and both $\widehat{A}^-(y)$ and $\widehat{A}^-(y)$ are continuous and increasing. Let $\widehat{y^*}^-(\alpha) = \sup\{y\in[0,1]:\widehat{A}^-(y)\leq\alpha\}$ and $\widehat{y^*}^+(\alpha) = \sup\{y\in[0,1]:\widehat{A}^+(y)\leq\alpha\}$. Then there exists a feasible choice of $\widehat{y^*}(\alpha)$ such that $\widehat{y^*}^+(\alpha)\leq \widehat{y^*}(\alpha)\leq \widehat{y^*}^-(\alpha)$, so it suffices to prove $\sup_{\alpha\in[0,1]}\big|\widehat{y^*}^+(\alpha) - y^*(\alpha)\big|\toprob 0$ and $\sup_{\alpha\in[0,1]}\big|\widehat{y^*}^-(\alpha) - y^*(\alpha)\big|\toprob 0$. First, we claim $\sup_{y\in[0,1]}\big|\widehat{A}^-(y) - A(y)\big|\toprob 0$ and $\sup_{y\in[0,1]}\big|\widehat{A}^+(y) - A(y)\big|\toprob 0$. In fact, since $A$ is uniformly continuous in $[0,1]$, for any $\delta>0$, there exists $\epsilon>0$ such that $|A(y_1)-A(y_2)|\leq \frac{\delta}{2}$ as long as $|y_1-y_2|\leq \epsilon$. Then
    \begin{align*}
    0&\leq\sup_{y\in[0,1]}\pth{\widehat{A}^+(y) - \widehat{A}^-(y)}=\sup_{y\in[0,1)}\pth{\widehat{A}(y+) - \widehat{A}(y)}\leq \sup_{y\in[0,1)}\pth{\widehat{A}(y+\epsilon) - \widehat{A}(y)}\\
    &\leq 2\sup_{y\in[0,1]}\big|\widehat{A}(y) - A(y)\big| + \sup_{y\in[0,1)}|A(y+\epsilon)-A(y)| \leq 2\sup_{y\in[0,1]}\big|\widehat{A}(y) - A(y)\big| + \frac{\delta}{2}
    \end{align*}
    by noting that $\widehat{A}$ is left-continuous and piecewise constant. It follows that 
    \[
    \Prob\pth{\sup_{y\in[0,1]}\big|\widehat{A}^+(y) - \widehat{A}^-(y)\big|>\delta} \leq \Prob\pth{\sup_{y\in[0,1]}\big|\widehat{A}(y) - A(y)\big| > \frac{\delta}{4}} \to 0
    \]
    by part 4, so $\sup_{y\in[0,1]}\big|\widehat{A}^+(y) - \widehat{A}^-(y)\big|\toprob 0$. Note that $\widehat{A}^-(y)\leq\widehat{A}(y)\leq\widehat{A}^+(y)$, we obtain the claim. 
    
    Next, we show that $\sup_{\alpha\in[0,1]}\big|\widehat{y^*}^-(\alpha) - y^*(\alpha)\big|\toprob 0$. If not, there exists $\epsilon_0>0$, $\eta_0>0$ and $N>0$ such that $\Prob\pth{\sup_{\alpha\in[0,1]}\big|\widehat{y^*}^-(\alpha) - y^*(\alpha)\big| > \epsilon_0} > 2\eta_0$ for any $n>N$. Let $2\delta_0=\inf\limits_{y\in[0,1-\epsilon_0]}\pth{A(y+\epsilon_0)-A(y)}$. Since $A$ is continuous and strictly increasing on $[0,1]$, we have $\delta_0>0$. Then there exists $N'>0$ such that $\Prob\pth{\sup_{y\in[0,1]}\big|\widehat{A}^-(y) - A(y)\big| < \delta_0} > 1-\eta_0$ for any $n>N'$. Consider the events $E_1^n=\left\{\sup_{\alpha\in[0,1]}\big|\widehat{y^*}^-(\alpha) - y^*(\alpha)\big| > \epsilon_0\right\}$ and $E^n_2=\left\{\sup_{y\in[0,1]}\big|\widehat{A}^-(y) - A(y)\big| < \delta_0\right\}$. When $n>\max(N,N')$, we have $\Prob\pth{E^n_1 \cap E_2^n} \geq \Prob\pth{E^n_1} + \Prob\pth{E^n_2} - 1 \geq 2\eta_0 + 1-\eta_0 - 1 = \eta_0 > 0$. On the event $E^n_1 \cap E_2^n$, if $\widehat{y^*}^-(\alpha) - y^*(\alpha)> \epsilon_0$ for some $\alpha\in[0,1]$, we have $y^*(\alpha) < 1-\epsilon_0$ and thus
    \begin{align*}
        \alpha \geq\widehat{A}^-(\widehat{y^*}^-(\alpha)) \geq \widehat{A}^-(y^*(\alpha)+\epsilon_0) > A(y^*(\alpha)+\epsilon_0) - \delta_0 \geq A(y^*(\alpha)) +2\delta_0-\delta_0=\alpha + \delta_0,
    \end{align*}
    which leads to contradiction. If $y^*(\alpha) - \widehat{y^*}^-(\alpha)> \epsilon_0$ for some $\alpha\in[0,1]$, we have $\widehat{y^*}^-(\alpha) < 1-\epsilon_0$ and thus
    \begin{align*}
        \alpha \geq A(y^*(\alpha)) \geq A(\widehat{y^*}^-(\alpha)+\epsilon_0) > A(\widehat{y^*}^-(\alpha)) + 2\delta_0 \geq \widehat{A}^-(\widehat{y^*}^-(\alpha)) +\delta_0=\alpha + \delta_0,
    \end{align*}
    which also leads to contradiction. Therefore, we have $\sup_{\alpha\in[0,1]}\big|\widehat{y^*}^-(\alpha) - y^*(\alpha)\big|\toprob 0$. Similarly, $\sup_{\alpha\in[0,1]}\big|\widehat{y^*}^+(\alpha) - y^*(\alpha)\big|\toprob 0$. Finally, we conclude
    \[\sup_{\alpha\in[0,1]}\big|\widehat{y^*}(\alpha) - y^*(\alpha)\big|\leq \sup_{\alpha\in[0,1]}\big|\widehat{y^*}^-(\alpha) - y^*(\alpha)\big| + \sup_{\alpha\in[0,1]}\big|\widehat{y^*}^+(\alpha) - y^*(\alpha)\big| \toprob 0.\]
    
    \item Show $\sup\limits_{\alpha\in[0,1]}\big|\widehat{b^*}(\alpha)- b^*(\alpha)\big|\toprob 0$, as $n\to\infty$:
    
    Since $B$ is uniformly continuous on $[0,1]$, for any $\delta>0$, there exists $\epsilon>0$ such that $|B(y_1)-B(y_2)|\leq \frac{\delta}{2}$ as long as $|y_1-y_2|\leq\epsilon$. If $\sup_{\alpha\in[0,1]}|\widehat{y^*}(\alpha)-y^*(\alpha)|\leq\epsilon$, we have
    \begin{align*}
        \sup_{\alpha\in[0,1]}\big|\widehat{b^*}(\alpha) -  b^*(\alpha)\big| = \sup_{\alpha\in[0,1]}\big|\widehat{B}(\widehat{y^*}(\alpha))-B(y^*(\alpha))\big| \leq \|\widehat{B}-B\|_\infty + \frac{\delta}{2}.
    \end{align*}
    Hence, 
    \begin{align*}
        \Prob\pth{ \big\|\widehat{b^*} -  b^*\big\|_\infty > \delta}
        &\leq \Prob\pth{\big\|\widehat{b^*} -  b^*\big\|_\infty > \delta,\,\big\|\widehat{y^*}-y^*\big\|_\infty\leq\epsilon}+\Prob\pth{\big\|\widehat{y^*}-y^*\big\|_\infty>\epsilon}\\
        &\leq \Prob\pth{ \big\|\widehat{B}-B\big\|_\infty > \frac{\delta}{2}} + \Prob\pth{\big\|\widehat{y^*}-y^*\big\|_\infty>\epsilon} \to 0
    \end{align*}
    by part 5 and 6. Therefore, $\sup_{\alpha\in[0,1]}\big|\widehat{b^*}(\alpha)- b^*(\alpha)\big|\toprob 0$.
    
    \item Show $\sup\limits_{\alpha\in[0,1]}\big|\widehat{\optFNR}(\alpha)-\optFNR(\alpha)\big|\toprob 0$, as $n\to\infty$:
    
    Let $\Delta=\|\widehat{b^*}- b^*\|_\infty$. For any $\alpha\in[0,1]$, we have $\widehat{\optFNR}(\alpha)\leq \widehat{b^*}(\alpha) \leq  b^*(\alpha) + \Delta$. Since $\widehat{\optFNR}(\cdot)-\Delta$ is convex and $\optFNR$ is the GCM of $b^*$, we have $\widehat{\optFNR}(\alpha)-\Delta\leq \optFNR(\alpha)$ for all $\alpha\in[0,1]$. Similarly, $\optFNR(\alpha)-\Delta\leq \widehat{\optFNR}(\alpha)$ for all $\alpha\in[0,1]$. Therefore, $\|\widehat{\optFNR}-\optFNR\|_\infty \leq \Delta = \|\widehat{b^*}- b^*\|_\infty\toprob 0$.
\end{enumerate}
We complete the proof of \prettyref{prop:uniform_consistency}.

\subsection{Proof of \prettyref{prop:consistency}}\label{app:ProofConsistency2}
For any fixed $\alpha\in[0,1]$, by the choices of $z_{k_0}$ and $p$, we have 
\begin{align}
    &p\,z_{k_0} + (1-p)\,z_{k_0+1} = \alpha,\label{eq:FDR_alpha}\\
    &p\,\widehat{b^*}(z_{k_0}) + (1-p)\,\widehat{b^*}(z_{k_0+1}) = \widehat{\optFNR}(\alpha).\label{eq:FNR_C*}
\end{align}
By the definitions of $\widehat{y^*}$ in \eqref{eq:y*_hat} and $\widehat{b^*}$ in \eqref{eq:OPT_3_hat}, we have
\begin{align}
    &\widehat{A}\pth{\widehat{y^*}(z_{k_0})} \leq z_{k_0},\quad\widehat{A}\pth{\widehat{y^*}(z_{k_0+1})} \leq z_{k_0+1},\label{eq:constr}\\
    &\widehat{b^*}(z_{k_0}) = \widehat{B}\pth{\widehat{y^*}(z_{k_0})},\quad\widehat{b^*}(z_{k_0+1}) = \widehat{B}\pth{\widehat{y^*}(z_{k_0+1})}.\label{eq:objective} 
\end{align}
Since $\omega_n\toprob 0$ and noting that $\omega_n\in[0,1]$, we have $\Expect[\omega_n]\to 0$, as $n\to\infty$. By the definition of $\delta^n$ in \eqref{eq:globally-randomized}, we have 
\begin{align*}
    \FDR(\delta^n) &= \Expect\left[\frac{\sum_{i=1}^n \delta_i(1-\theta_i)}{1\vee\sum_{i=1}^n \delta_i}\right] =  \Expect\left[\Expect\left[\frac{\sum_{i=1}^n \delta_i(1-\theta_i)}{1\vee\sum_{i=1}^n \delta_i}\bigg|X^n\right]\right] = \Expect\left[\Expect\left[\frac{\sum_{i=1}^n \delta_i W_i}{1\vee\sum_{i=1}^n \delta_i}\bigg|X^n\right]\right]\\
    &= \Expect\left[p\,\frac{\sum_{i=1}^n W_i\indc{\widehat{W}_i\leq \widehat{y^*}(z_{k_0})}}{1\vee\sum_{i=1}^n \indc{\widehat{W}_i\leq \widehat{y^*}(z_{k_0})}} + (1-p)\,\frac{\sum_{i=1}^n W_i\indc{\widehat{W}_i\leq \widehat{y^*}(z_{k_0+1})}}{1\vee\sum_{i=1}^n \indc{\widehat{W}_i\leq \widehat{y^*}(z_{k_0+1})}}\right]\\
    &\leq \Expect\left[p\,\widehat{A}\pth{\widehat{y^*}(z_{k_0})} + (1-p)\,\widehat{A}\pth{\widehat{y^*}(z_{k_0+1})} + \omega_n\right]\\
    &\overset{\eqref{eq:constr}}{\leq} \Expect[p\,z_{k_0} + (1-p)\,z_{k_0+1} + \omega_n] \overset{\eqref{eq:FDR_alpha}}{=} \alpha + \Expect[\omega_n].
\end{align*}
Hence, $\limsup\limits_{n\to\infty}\sup\limits_{\alpha\in[0,1]}\pth{\FDR(\delta^n)-\alpha}\leq\lim\limits_{n\to\infty}\Expect[\omega_n]= 0$. Similarly,
\begin{align*}
    \FNR(\delta^n) &= \Expect\left[p\,\frac{\sum_{i=1}^n (1-W_i)\indc{\widehat{W}_i> \widehat{y^*}(z_{k_0})}}{1\vee\sum_{i=1}^n \indc{\widehat{W}_i> \widehat{y^*}(z_{k_0})}} + (1-p)\,\frac{\sum_{i=1}^n (1-W_i)\indc{\widehat{W}_i> \widehat{y^*}(z_{k_0+1})}}{1\vee\sum_{i=1}^n \indc{\widehat{W}_i> \widehat{y^*}(z_{k_0+1})}}\right]\\
    &\leq \Expect\left[p\,\widehat{B}\pth{\widehat{y^*}(z_{k_0})} + (1-p)\,\widehat{B}\pth{\widehat{y^*}(z_{k_0+1})} + \omega_n\right]\\
    &\overset{\eqref{eq:objective}}{=} \Expect[p\,\widehat{b^*}(z_{k_0}) + (1-p)\,\widehat{b^*}(z_{k_0+1}) + \omega_n] \overset{\eqref{eq:FNR_C*}}{=} \Expect[\widehat{\optFNR}(\alpha)+\omega_n],
\end{align*}
and $\FNR(\delta^n)\geq \Expect[\widehat{\optFNR}(\alpha)-\omega_n]$. Thus, 
\[
\sup_{\alpha\in[0,1]}|\FNR(\delta^n)- \optFNR(\alpha)| \leq \Expect\left[\big\|\widehat{\optFNR} - \optFNR\big\|_\infty+\omega_n\right]\to 0
\]
by \prettyref{prop:uniform_consistency}. We complete the proof of \prettyref{prop:consistency}.

\section{Comparisons for a fixed FDR level}\label{app:ComparisonsFixedLevel}
We provide the following \prettyref{tab:comparison} to compare for a fixed FDR level $\alpha=0.2$, the estimated EFN, FNR, FDR, as well as the variance of FDP, for oracle BH, Sun\&Cai and \prettyref{algo:globally-randomized}. (When aiming to approximate the optimal FDR-EFN tradeoff, we replace the definition of $\widehat{B}(y)$ in \prettyref{eq:A_hat&B_hat} with $\widehat{B}(y)\triangleq  \frac{1}{n}\sum_{i=1}^n (1-\widehat{W}_i)\indc{\widehat{W}_i> y}$ in \prettyref{algo:globally-randomized}.)

\begin{table}[htbp]
\centering
\begin{tabular}{
  S[table-format=1.2]  
  S[table-format=1] 
  l                    
  S[table-format=1.2]  
  S[table-format=1.3]  
  S[table-format=1.3]  
  S[table-format=1.2]  
}
\toprule
{$\pi_0$} & {$\mu$} & {Procedure} & {FDR} & {FNR} & {EFN} & {$\var(\text{FDP})$} \\
\midrule
0.75 & 1 & oracle BH   & 0.20 & 0.238 & 0.233  & 0.02\\
     &   & Sun\&Cai    & 0.18 & 0.240 & 0.236 & 0.02\\
     &   & Algorithm 1 & 0.20 & 0.190 & 0.181 & 0.11\\
\midrule
0.75 & 2 & oracle BH   & 0.20 & 0.108 & 0.086  & 0.001\\
     &   & Sun\&Cai    & 0.18 & 0.116 & 0.094 & 0.001\\
     &   & Algorithm 1 & 0.21 & 0.106 & 0.084 & 0.007\\
\midrule
0.99 & 2 & oracle BH   & 0.20 & 0.009 & 0.009  & 0.09\\
     &   & Sun\&Cai    & 0.11 & 0.009 & 0.009 & 0.06\\
     &   & Algorithm 1 & 0.24 & 0.008 & 0.008 & 0.16\\
\midrule
0.99 & 3 & oracle BH   & 0.20 & 0.005 & 0.005  & 0.03\\
     &   & Sun\&Cai    & 0.13 & 0.006 & 0.006 & 0.03\\
     &   & Algorithm 1 & 0.25 & 0.005 & 0.005 & 0.12\\
\bottomrule
\end{tabular}
\caption{Comparisons among oracle BH, Sun\&Cai and \prettyref{algo:globally-randomized} for $\alpha=0.2$.}
\label{tab:comparison}
\end{table}

For $\pi_0=0.75$, all three procedures successfully control the FDR at level $\alpha=0.2$. \prettyref{algo:globally-randomized} has smaller FNR and EFN but higher variability of FDP than the others for $\mu=1$. Sun\&Cai has larger FNR and EFN than the others for $\mu=2$, and they all have very small variability of FDP. When the signal is very sparse, \ie $\pi_0=0.99$, \prettyref{algo:globally-randomized} has larger variance of FDP than the others and its FDR is higher than the desired level due to the density estimation errors for the given sample size.

\section{Optimal FDR-EFN tradeoff}\label{app:PoofsFDR-EFN}
\subsection{Proof of \prettyref{thm:FDR-ETD}} 

It suffices to show $\lim\limits_{n\to\infty}\optEFNnorm_n(\alpha)=\tilde{b}^{**}(\alpha)$. We first show $\optEFNnorm_n(\alpha) \leq  \tilde{b}^{**}(\alpha)$ for any $n$ and $\alpha$. Suppose $U$ is a feasible solution of \eqref{eq:opt_prob_a&b'_random} and consider a randomized rule $\delta^n$ of the form \prettyref{eq:construction_of_delta}. For a fixed $u\in[0,1]$, let $\delta^{n,u}=(\delta^u_1,\cdots,\delta^u_n)$ denote the rule $\delta^n$ conditioned on $U=u$. It's shown in the proof of \prettyref{prop:asymp_opt} that $\FDR(\delta^{n,u})\leq\Prob\pth{\theta_1 = 0\,|\,\delta^u_1=1}= a(u)$ and $\FNR(\delta^{n,u}) \leq\Prob\pth{\theta_1 = 1\,|\,\delta^u_1=0}= b(u)$. Hence, $\EFN(\delta^{n,u}) = \Prob\pth{\theta_1 = 1,\delta^u_1=0}= \Prob\pth{\delta^u_1=0}\Prob\pth{\theta_1 = 1\,|\,\delta^u_1=0}= (1-u)b(u)=\tilde{b}(u)$.
Then,     
$\FDR(\delta^n) = \Expect[\FDR(\delta^{n,U})] \leq \Expect[a(U)] \leq \alpha\text{ and }
\EFN(\delta^n)\leq \Expect[\tilde{b}(U)]$. Consequently, $\optEFNnorm_n(\alpha) \leq \EFN(\delta^n) \leq \Expect[\tilde{b}(U)]$. Optimizing over all feasible $U$, we have $\optEFNnorm_n(\alpha) \leq  \tilde{b}^{**}(\alpha)$. 

Next, we prove $\liminf\limits_{n\to\infty} \optEFNnorm_n(\alpha) \geq  \tilde{b}^{**}(\alpha)$ for any $\alpha$. Following \prettyref{eq:concentrate_b_tilde}, we have $$\Expect\left[\pth{\frac{1}{n}\sum_{i=K_n+1}^n (1-W_{(i)}) - \tilde{b}(K_n/n)}\cdot\indc{K_n\leq (1-\tau)n}\right] \to 0.$$ Then, by replacing $\frac{1}{n-K_n}\sum_{i=K_n+1}^n (1-W_{(i)})$ with $\frac{1}{n}\sum_{i=K_n+1}^n (1-W_{(i)})$, $b(\cdot)$ with $\tilde{b}(\cdot)$, and $\optFNR_n$ with $\optEFNnorm_n$ in the proof of \prettyref{prop:asymp_opt}, we can show that $ \tilde{b}^{**}(\alpha+2\epsilon)\leq \optEFNnorm_n(\alpha) + 3\epsilon$ for sufficiently large $n$. Letting $n\to\infty$, we get $\liminf\limits_{n\to\infty}\optEFNnorm_n(\alpha) \geq  \tilde{b}^{**}(\alpha+2\epsilon) - 3\epsilon$. Similar to the properties of $b^{**}$ stated in \prettyref{lmm:property_of_b_doule_star}, $\tilde{b}^{**}$ is the GCM of $\tilde{b}^*$ defined in \prettyref{eq:opt_prob_a&b'}. Consequently, $\tilde{b}^{**}$ is convex on $[0,1]$ and thus continuous on $(0,1)$. Letting $\epsilon\downarrow 0$, we have $\liminf\limits_{n\to\infty}\optEFNnorm_n(\alpha) \geq  \tilde{b}^{**}(\alpha)$ for any $\alpha\in(0,1)$. The corner cases of $\alpha=0$ and $\alpha=1$ can be verified easily.

To conclude, $\lim\limits_{n\to\infty}\optEFNnorm_n(\alpha)=\tilde{b}^{**}(\alpha)$, which is the GCM of the function $\tilde{b}^{*}=\optmEFNnorm$. Furthermore, procedure \prettyref{eq:construction_of_delta} asymptotically achieves $\tilde{b}^{**}(\alpha)$, by substituting the optimal solution of \prettyref{eq:opt_prob_a&b'_random} for $U$.

\subsection{Proof of \prettyref{thm:high_prob_formulation_1}}

We first prove the positive part. Let $\maroracle(\alpha)=(\delta_1,\ldots,\delta_n)$ be the optimal rule for \prettyref{eq:opt_prob_mFDR&EFN}, which has the form \prettyref{eq:optimal_sol_of_mFDR&mFNR}. It's shown in the proof of the positive part of \prettyref{thm:high_prob_formulation} that 
$\text{FDP}(\maroracle(\alpha)) \toprob a(u^*(\alpha)) \leq \alpha$, where $u^*(\alpha)$ is the optimal solution of \prettyref{eq:opt_prob_a&b'}. By the Law of Large Numbers, we have
\begin{align*} 
    \frac{1}{n}\sum_{i=1}^n\theta_i(1-\delta_i) \toprob \pi_1\Expect_1[1-S_{u^*(\alpha)}] = \Expect[(1-W)\indc{S_{u^*(\alpha)}=0}]=\tilde{b}(u^*(\alpha))=\tilde{b}^*(\alpha).
\end{align*}
Hence, $\FN(\maroracle(\alpha))= \frac{1}{n}\sum_{i=1}^n\theta_i(1-\delta_i) \toprob\tilde{b}^*(\alpha)=\optmEFNnorm(\alpha)$.

Next, we prove the negative part of \prettyref{thm:high_prob_formulation_1}. Suppose $\Prob\pth{\FDP(\delta^n) \leq \alpha} \geq 1-\tau_n$ for some $\tau_n=o(1)$ and $\EFN(\delta^n) \leq \beta$. We aim to show $\beta \geq \optmEFNnorm(\alpha)=\tilde{b}^*(\alpha)$. Thanks to \prettyref{eq:concentrate_b_tilde}, by replacing $\frac{1}{n-K_n}\sum_{i=K_n+1}^n (1-W_{(i)})$ with $\frac{1}{n}\sum_{i=K_n+1}^n (1-W_{(i)})$, $\FNR$ with $\EFN$, $b$ with $\tilde{b}$, and $b^*$ with $\tilde{b}^*$ in the proof of the negative part of \prettyref{thm:high_prob_formulation}, we can show that $ \tilde{b}^*\pth{\alpha + \sqrt{\tau_n}+\frac{\epsilon}{2}} \leq \beta + 3\epsilon$. By first letting $n\to\infty$ and then letting $\epsilon\downarrow 0$
and noting that $ \tilde{b}^*$ is right-continuous, we obtain $\beta\geq \tilde{b}^*(\alpha)=\optmEFNnorm(\alpha)$. 

\section*{Acknowledgment}
The authors are grateful to Yury Polyanskiy for helpful discussions. In particular, the idea behind the proof of \prettyref{lmm:W&W_tilde} is due to him.

\bibliographystyle{alpha}
\bibliography{FDR}

\end{document}